\documentclass[reqno,11pt]{preprint}

\usepackage[marginparwidth=1in]{geometry}
\usepackage[full]{textcomp}
\usepackage[osf]{newtxtext}
\usepackage{cabin}
\usepackage{enumerate}
\usepackage{enumitem}
\setlist[itemize]{label=\small\textbullet}
\usepackage{bbm}
\usepackage{dsfont}

\usepackage[makeroom]{cancel}

\usepackage{stmaryrd}
\usepackage{csquotes}

\usepackage{commath}
\usepackage{nicefrac}

\usepackage[varqu,varl]{inconsolata}
\usepackage[cal=boondoxo]{mathalfa}

\usepackage{mhcomment}
\usepackage{hyperref}
\usepackage{breakurl}
\usepackage{mathrsfs}
\usepackage{booktabs}
\usepackage{caption}

\usepackage{tikz}
\usetikzlibrary{decorations.pathreplacing,decorations.markings}
\usepackage{amsmath} 
\usepackage{mhequ} 
\usepackage{mhsymb} 
\usepackage{mhenvs} 
\usepackage{microtype}
\usepackage{amssymb} 
\usepackage{eufrak}

\usepackage{wasysym}
\usepackage{centernot}
\usepackage{scalerel}

\usepackage{pifont}

\usepackage{makecell}

\usepackage[utf8]{inputenc}
\usepackage[T1]{fontenc}

\usepackage{oldgerm}

\newcommand{\Exp}{\mathbf{E}}
\newcommand{\Prob}{\mathbf{P}}

\renewcommand{\R}{\mathbb{R}}
\renewcommand{\C}{\mathbb{C}}
\renewcommand{\N}{\mathbb{N}}
\renewcommand{\Q}{\mathbb{Q}}

\renewcommand{\Z}{\mathbb{Z}}
\renewcommand{\E}{\mathbb{E}}
\renewcommand{\P}{\mathbb{P}}

\renewcommand{\L}{\mathbb{L}}

\newcommand{\bd}{\mathbf{d}}

\newcommand{\BP}{\mathbf{P}}
\newcommand{\BE}{\mathbf{E}}

\newcommand{\cA}{\mathcal{A}}

\newcommand{\cC}{\mathcal{C}}
\newcommand{\cD}{\mathcal{D}}

\newcommand{\cF}{\mathcal{F}}
\newcommand{\cG}{\mathcal{G}}
\newcommand{\cH}{\mathcal{H}}

\newcommand{\cL}{\mathcal{L}}

\newcommand{\cN}{\mathcal{N}}
\newcommand{\cO}{\mathcal{O}}

\newcommand{\cS}{\mathcal{S}}

\newcommand{\cX}{\mathcal{X}}

\newcommand{\cg}{\mathcal g}

\renewcommand{\fE}{\mathfrak{E}}

\newcommand{\fe}{\mathfrak{e}}
\newcommand{\fw}{\mathfrak{w}}

\newcommand{\dd}{\mathrm{d}}      
\newcommand{\dis}{\mathrm{dis}}
\newcommand{\eff}{\mathrm{eff}}

\newcommand{\loc}{\mathrm{loc}}

\newcommand{\SH}{\mathscr{H}}

\renewcommand{\fE}{\mathfrak{E}}

\def\one{\mathrm{(I)}}
\def\two{\mathrm{(II)}}
\def\three{\mathrm{III}}

\newcommand{\1}{\mathds{1}}

\def\one{\mathrm{(I)}}
\def\two{\mathrm{(II)}}
\def\three{\mathrm{(III)}}

\newcommand{\fock}[3]{ 
    \def\order{#2}
    \ifx\order\empty
        \Gamma_{#1} L^2_{#3}
    \else
        \Gamma_{#1} H^{#2}_{#3}
    \fi
}

\newcommand{\core}{\mathbf{T}^\tau}

\newcommand{\gen}{\cL^\tau}

\newcommand{\geneff}{\cL^\eff}

\newcommand{\gensy}{\cS^\tau}
\newcommand{\gensyx}{\cL_0^{\fe_1}}
\newcommand{\gensyy}{\cL_0^{\fe_2}}

\newcommand{\gena}{\cA^{\tau}}
\def\genaF#1#2{{\cA^{\tau,\scalebox{0.6}{$#1$}}_{#2}}} 
\def\genaFsh#1#2{{\cA^{\tau,\scalebox{0.6}{$2$}}_{\scalebox{0.6}{$#1$}, \scalebox{0.6}{$#2$}}}} 

\newcommand{\genapmsh}{\cA^{\sharp, \tau}_{\pm}}

\newcommand{\genapmfl}{\cA^{\flat, \tau}_{\pm}}

\usepackage{tikz,pgfplots}
\pgfplotsset{compat=newest}
\usetikzlibrary{cd}
\usetikzlibrary{calc}
\usetikzlibrary{arrows, scopes}
\usetikzlibrary{decorations.pathmorphing}
\usetikzlibrary{positioning}
\usetikzlibrary{shapes.geometric}
\usetikzlibrary{backgrounds}
\usetikzlibrary{arrows,chains,matrix,positioning,scopes}

\tikzset{
	whitenode/.style={circle,fill=black!10,draw=black,inner sep=0pt, minimum size=1.5mm},
	blacknode/.style={circle,fill=black,draw=black,inner sep=0pt, minimum size=1.5mm},
}

\usepackage{relsize}

\usepackage{subfig}
\usepackage{caption}

\makeatletter
\renewcommand{\p@subfigure}{}
\makeatother

\newcommand{\eqlaw}{\stackrel{\mbox{\tiny law}}{=}}

\def\restriction#1#2{\mathchoice
              {\setbox1\hbox{${\displaystyle #1}_{\scriptstyle #2}$}
              \restrictionaux{#1}{#2}}
              {\setbox1\hbox{${\textstyle #1}_{\scriptstyle #2}$}
              \restrictionaux{#1}{#2}}
              {\setbox1\hbox{${\scriptstyle #1}_{\scriptscriptstyle #2}$}
              \restrictionaux{#1}{#2}}
              {\setbox1\hbox{${\scriptscriptstyle #1}_{\scriptscriptstyle #2}$}
              \restrictionaux{#1}{#2}}}
\def\restrictionaux#1#2{{#1\,\smash{\vrule height .8\ht1 depth .85\dp1}}_{\,#2}}

\let\Phi=\phi
\let\phi=\varphi
\let\epsilon=\varepsilon
\let\dis=\displaystyle

\colorlet{darkblue}{blue!90!black}
\colorlet{darkred}{red!90!black}
\colorlet{darkgreen}{green!50!black}
\colorlet{darkyellow}{yellow!90!black}

\title{Superdiffusive central limit theorem for a class of \\ driven diffusive systems at the critical dimension}

\begin{document}

\maketitle

\vspace{-2cm}

\noindent{\textbf{Giuseppe Cannizzaro$^1$}, \textbf{Tom Klose$^2$}, \textbf{Quentin Moulard$^3$}}
\newline

\noindent{\small $^1$University of Warwick, UK, \email{giuseppe.cannizzaro@warwick.ac.uk}\\
$^2$University of Oxford, UK, \email{tom.klose@maths.ox.ac.uk}\\
    $^3$Technical University of Vienna, Austria, \email{quentin.moulard@tuwien.ac.at}}
\newline

\begin{abstract}

We study the large-scale behaviour of a class of driven diffusive systems 
	modelled by a Stochastic Partial Differential Equation, the Stochastic Burgers Equation~(SBE) 
	with general nonlinearity, at the critical dimension and in infinite volume. 
	Our main result shows that, under a logarithmically superdiffusive space-time scaling, it is given by the same 
	explicit Gaussian Fixed point obtained in [G. Cannizzaro, Q. Moulard, \& F. Toninelli, 
	https://arxiv.org/abs/2501.00344, 2025] for 
	the quadratic SBE, but with suitably renormalised coefficients, thereby rigorously justifying 
	and partly correcting the classical Physics derivation of the SBE in  [H. van Beijeren, R. Kutner, \& H. Spohn, 
  Phys. Rev. Lett., 1986] based on Spohn's 
	theory of nonlinear fluctuating hydrodynamics.  
	Besides, ours is the first universality-type result for out-of-equilibrium systems and the first extension of 
	[M. Hairer, J. Quastel, Forum of Mathematics, Pi, Vol. 6, 2018, e3], to the critical dimension and beyond weak coupling. 
	The major challenge in our work is the mild growth condition on the nonlinearity which renders even the well-posedness of the microscopic equation non-trivial. 
	Additional key novelties include the derivation of fine estimates on the non-quadratic part of the generator as well as a new approximation for the resolvent associated to the solution of the quadratic SBE.  

\end{abstract}

\bigskip

\setcounter{tocdepth}{3} 
\tableofcontents

\section{Introduction} \label{sec:intro}

In the context of out-of-equilibrium statistical mechanics models, the classical theory of 
nonlinear fluctuating hydrodynamics~\cite{Spo1} allows to {\it heuristically} 
derive (non-linear) Stochastic Partial Differential 
Equations (SPDEs) as an effective description for their universal fluctuations 
around the macroscopic profile. 
Renowned examples include randomly growing surfaces, which lead 
to the celebrated KPZ equation~\cite{KPZ, BS, Ton}, 
lattice gas models, Navier--Stokes dynamics~\cite{Spo}, and
driven diffusive systems~\cite{vBKS85}. 

The goal of the present paper is to rigorously justify (and partly {\it correct}) 
this derivation: 
We show 
that the large-scale behaviour of a class of 
driven diffusive systems {\it at the critical dimension} obeys the same law
as that of the corresponding 
heuristically derived SPDE, the so-called Stochastic Burgers Equation, 
with {\it suitably renormalised coefficients}. 
\subsection{Physical motivation and background} \label{s:background}
We are interested in a family of random scalar conservation laws taking the form of the non-linear 
Langevin equation 
\begin{equ} \label{e:ConservationLaw}
\partial_t u + \nabla \cdot \vec{j}(u) = 0 \, , \quad \vec{j}(u) \eqdef - D(u) \nabla u + u \, \vec{w} (u) + \sigma(u) \,  \vec{\xi} \, .
\end{equ}
where $u=u(t,x)\in\R$ for $(t,x)\in\R_+\times\R^d$, and $D,\,\sigma, \vec{w}$ are arbitrary functions on $\R$, 
the first two $\R_+$-valued while the latter is $\R^d$-valued. 
The system in~\eqref{e:ConservationLaw} was 
introduced by van Beijeren, Kutner, and Spohn in~\cite{vBKS85} 
to model the fluctuations of a conserved scalar quantity $u$ in $d+1$ space-time dimensions.
The current density $\vec{j}$ is the result of three contributions: 
a \emph{diffusive term}, $-D(u) \nabla u$, consistent with Fick's law; 
an \emph{advective term} $u \,\vec{w}(u)$, which describes the transport of $u$ by the flow velocity $\vec{w}(u)$; 
and a \emph{stochastic term} $\sigma(u) \, \vec{\xi}$, which models random fluctuations 
via a vector-valued space-time white noise $\vec{\xi} = (\xi_1, \dots, \xi_d)$. 
%

To derive the effective equation governing the behaviour of~\eqref{e:ConservationLaw},  
in~\cite{vBKS85}, the authors perform a \emph{perturbative expansion} of the non-linear coefficient $\vec{w}$ 
around a constant reference density $u_0$ (the {\it macroscopic profile} 
alluded to earlier, which is flat as we are implicitly assuming to be at stationarity). 
This amounts to setting $\bar u \eqdef u-u_0$ and then Taylor expanding $\vec{w}$ at $u_0$, i.e. 
\begin{equ}[e:PerturbativeExpansion]
\vec{w}(u)= \vec{f}_0 + \vec{f}_1 \, \bar{u} + \vec{f}_2 \, \bar{u}^2 + \dots\,,\quad \text{with }\quad\vec{f}_i=\frac{\vec{w}^{\,(i)}(u_0)}{i!}\,.
\end{equ}
In principle, we should also expand $D$ and $\sigma$ but the perturbative corrections 
beyond the zero-th order term are generally neglected 
as they are expected to be irrelevant compared to those arising from the nonlinear advection term. 
We will follow this convention and assume $D$ and $\sigma$ to be  constant. 
For simplicity, we further take $D(u) = \frac{1}{2} \mathrm{Id}$ and $\sigma(u) = \mathrm{Id}$.
From~\eqref{e:ConservationLaw} and the previous discussion, we obtain an equation for $\bar u$ of the form
\begin{equ}[e:PerturbativeSBE]
\partial_t \bar{u} =\frac{1}{2} \Delta \bar{u} - \Big[\vec{f}_0\cdot \nabla \bar u + \vec{f}_1\cdot\nabla \bar{u}^2 + \vec{f}_2 \, \cdot\nabla\bar{u}^3 + \dots\Big] - \nabla \cdot \vec{\xi} \, .
\end{equ}
Now, the linear term, $\vec{f}_0\cdot \nabla \bar u$, can always be removed by performing 
the Galilean transformation $x \mapsto x + \vec{f}_0 \, t$. 
Concerning the rest, it is commonly believed (and similarly argued in~\cite{vBKS85}) 
that {\it perturbative terms of degree three and higher have a negligible effect compared to the quadratic one}, 
a belief which is supported by the following simple scaling argument. 

Consider the \emph{diffusive rescaling}, which corresponds to zooming out in space and time as 
$\bar{u}^\tau(t,x) = \bar{u}^\tau(t,x) = \tau^{d/4} \, \bar{u}(\tau t, \sqrt{\tau} x)$ for $\tau \gg 1$. Then, 
equation \eqref{e:PerturbativeSBE} becomes
\begin{equ} \label{e:ZoomingSBE}
\partial_t \bar{u}^\tau = \frac{1}{2} \Delta \bar{u}^\tau- \Big[\tau^{1-\frac{d}2}\, \vec{f}_1\cdot\nabla (\bar{u}^\tau)^2 + \tau^{1-d}\vec{f}_2 \, \cdot\nabla(\bar{u}^\tau)^3 + \dots\Big]  - \nabla \cdot \vec{\xi}^\tau \, ,
\end{equ}
where $\vec{\xi}^\tau$ denotes the appropriately rescaled space-time white noise. 
Observe that the nonlinear term of degree $m$ is now scaled by a factor $\tau^{1 - (m-1) d/2}$, 
whose order is decreasing as $m$ increases.
At large scales, it seems therefore heuristically justified to truncate the expansion 
in~\eqref{e:PerturbativeExpansion} beyond the quadratic contribution, 
leading to the so-called (quadratic) Stochastic Burgers Equation (SBE)
\begin{equ} \label{e:SBE}
\partial_t \bar{u} = \frac{1}{2} \Delta \bar{u} - \,\vec{f_1} \cdot \nabla \bar{u}^2 - \nabla \cdot \vec{\xi} \, .
\end{equ}
where the vector $\vec{f_1}$ should be chosen, according to the above heuristic, as in~\eqref{e:PerturbativeExpansion}. 

The quadratic SBE has been extensively studied in the Mathematics literature 
and the characterisation of its (small-scale, only in $d=1$, and) large-scale behaviour 
has been at the core of a rich and vibrant area of research. Depending on the spatial 
dimension $d$, the scaling argument in~\eqref{e:ZoomingSBE} and, in particular, 
the coefficient of the quadratic term $\tau^{1-\frac{d}2}$ determine whether 
the equation is {\it sub-critical}, $d=1$, {\it super-critical}, $d\geq 3$ or {\it critical}, $d=2$. 
At the sub-critical dimension $d=1$, the analysis of the small-scales of~\eqref{e:SBE} (and of its ``primitive'' 
$h$, i.e. $\partial_x h=u$, solving the KPZ equation~\cite{KPZ}) has lead to 
the development of the Theory of Regularity Structures~\cite{Hai13, Hai14}, while, at large scales, 
the SBE was shown to be {\it polynomially superdiffusive} in~\cite{BQS} and to converge 
to the KPZ Fixed Point~\cite{QS1, Virag}. The super-critical case, $d\geq 3$, instead 
was treated in~\cite{CGT}, where diffusive behaviour and asymptotic Gaussianity were determined. 
At last, the most recent work is~\cite{CMT}, which deals with the critical dimension $d=2$: 
The solution of SBE is proved to be {\it superdiffusive}, as in $d=1$, but only {\it logarithmically}, 
and to display, under a logarithmically superdiffusive scaling, {\it Gaussian fluctuations}, as in $d\geq 3$.  
\medskip

The present article is concerned with this latter case and it is the first work 
which \emph{rigorously formalises} the truncation of~\eqref{e:PerturbativeSBE}. 
We consider the system~\eqref{e:ConservationLaw} at the \emph{critical dimension}~$d=2$ in \emph{infinite volume},  
with $D=\frac12\sigma=\frac12 {\rm Id}$ and $\vec{w}$ {unidirectional and} smooth, satisfying 
extremely mild growth assumptions, and show that it converges, 
under a \emph{logarithmically superdiffusive scaling}, 
to the same \emph{universal Gaussian limit} as the rescaled solutions to~\eqref{e:SBE}, 
but with a choice of the vector~$\vec{f}_1$ {\it different} from that in~\eqref{e:PerturbativeExpansion}, 
which we explicitly identify 
(see Theorems~\ref{th:CLT} and~\ref{th:conv_semigroup} below for precise statements). 

\subsection{The model and main results}
As mentioned above, our focus is the system~\eqref{e:ConservationLaw} in dimension $d=2$. 
Since we assume $\vec{w}$ to be {unidirectional, we can write the advection term as 
$u \vec{w}(u) = F(u) \, \vec{\fw}$ for some function $F\colon \R\to\R$ and constant vector $\vec{\fw} \in \mathbb S^1$.} 
Then, the model of interest becomes 
\begin{equ} \label{e:FormalGeneralizedSBE}
\partial_t u = \frac{1}{2} \Delta u + \vec{\fw} \cdot \nabla F(u) + \nabla \cdot \vec{\xi} \quad \text{on} \quad \R_+ \times \R^2 \,.
\end{equ}
By rotational invariance of the spatial domain, the specific choice of $\vec{\fw}$ is irrelevant; 
without loss of generality, we fix $\vec{\fw} = \fe_1$, where $(\fe_1,\fe_2)$ denotes the canonical basis of $\R^2$.
%
%
%
%
%

As written, the SPDE~\eqref{e:FormalGeneralizedSBE} is severely ill-posed due to the 
singularity of the space-time white noise $\vec{\xi}$. Further, since the dimension $d=2$ is {\it critical},  
not even the pathwise theories of Regularity Structures~\cite{Hai14} or Paracontrolled Calculus~\cite{GIP15} 
apply. In fact, a local solution theory is not even expected to exist. 
That said, we are interested in {\it large scales} and thus it is legitimate 
to start from a microscopically well-defined model. Following the approach of~\cite{CMT}, 
let $\rho$ be a smooth, compactly supported, non-negative, and even function with unit total mass 
and such that $\|\rho\|_{L^2(\R^2)} = 1$, 
and consider the regularised version of~\eqref{e:FormalGeneralizedSBE} given by  
\begin{equ} \label{e:GeneralisedSBE}
	\partial_t u = \frac{1}{2} \Delta u + \cN^1_F(u) + \nabla \cdot \vec{\xi}^{\thinspace 1} \, \quad \text{on} \quad \R_+ \times \R^2, 
\end{equ}
where 
\begin{equ} \label{e:Regularization1}
	\cN_F^1(u) \eqdef \rho^{*2} * \partial_1 F(u) \, , \quad \vec{\xi}^{\thinspace 1} \eqdef \rho * \vec{\xi} \, . \
\end{equ}
We will require the function $F\colon\R\to\R$ to satisfy the following assumption. 

\begin{assumption}[On the non-linearity] \label{Assumption:F}
We assume that $F \in \cC^\infty(\R)$ and that there exists $\kappa \in [0,1/4)$ such that, for every $n \geq 0$ and $y \in \R$, it holds that  
\begin{equ} \label{e:assumptionF}
|F^{(n)}(y)| \leq \exp(\exp(o(n))) \, \exp(\kappa y^2) \,, 
\end{equ}
where $o(n)$ is the usual Landau symbol, i.e. $o(n)$ is a function of $n$ 
such that $o(n)/n$ converges to $0$ as $n\to\infty$. 
\end{assumption}

Roughly speaking, this assumption encompasses all reasonable smooth functions 
whose growth is slower than $\exp(x^2/4)$.
%
In particular, it guarantees that the constants~$c_1(F)$ and~$c_2(F)$, given by 
\begin{equ} \label{e:c1_c2}
	c_1(F) \eqdef \int_{\R} F'(y) \; \pi_1(\dif y), 
	\quad 
	c_2(F) \eqdef \int_{\R} F''(y) \; \pi_1(\dif y), 
\end{equ}
for $\pi_1$ the law of a standard mean-zero real-valued Gaussian random variable, 
are indeed finite (see Proposition~\ref{p:AssF} in Appendix~\ref{a:DecayCoeff}). 
We refer the reader to Section~\ref{s:discussion} for a more detailed discussion of Assumption~\ref{Assumption:F}, interesting examples for which it is satisfied, as well as a more systematic perspective on the constants in~\eqref{e:c1_c2}. 
\medskip

Even if the noise (and the nonlinearity) were smoothened in space, 
it is not at all obvious that~\eqref{e:GeneralisedSBE} 
admits a unique solution. 
Indeed, the growth condition~\eqref{e:assumptionF} is very weak and we are working 
in {\it infinite volume} so that classical analytic tools do not necessarily apply. 
Our first task in Section~\ref{sec:WellPosed} will then be to identify a good notion of solution 
(see Definition~\ref{def:SolSPDE}) and prove well-posedness of~\eqref{e:GeneralisedSBE} 
(see Proposition~\ref{p:MollifSol}). 
The results therein are summarised in the next loosely stated proposition. 

\begin{proposition}\label{p:WellPosed}
Let $F$ be such that Assumption~\ref{Assumption:F} holds. Then,~\eqref{e:GeneralisedSBE} admits 
a unique global-in-time solution $u$. Further, $u$ is Markov and has $\P^\rho$ as invariant measure, 
where $\P^\rho$ is the law of $\eta^1\eqdef \eta\ast\rho$ for $\eta$ a spatial white noise on $\R^{2}$.  
\end{proposition}

We now turn to the analysis of the large-scale behaviour of $u$. As we expect it to be analogous 
to that of the solution to~\eqref{e:GeneralisedSBE} with $F$ quadratic, the relevant 
scaling should be the one implied by~\cite[Thm. 1.1]{CMT}. Since further we need 
to remove the contribution coming from the linear term (see the discussion after~\eqref{e:PerturbativeSBE}), 
we also place ourselves to a moving frame and consequently define, for a given scale $\tau\geq 1$, 
the rescaled $u$ according to 
\begin{equ} \label{e:GoodScaling}
	u^\tau(t,x) \eqdef \tau^{1/2} \nu_\tau^{-1/4} \, u\big(\tau \, t, \tau^{1/2} R_\tau^{-1/2} x - \tau t c_1(F) \fe_1\big)  \,
\end{equ}
where $c_1(F)$ is as in~\eqref{e:c1_c2} and $R_\tau,\,\nu_\tau$ satisfy 
\begin{equ} \label{e:Rtau}
	R_\tau \eqdef 
	\begin{pmatrix}
	\nu_\tau & 0 \\
	0 & 1
	\end{pmatrix} \, , 
	\qquad \text{for }\quad
	\nu_\tau \eqdef 
	\frac{1}{1 \vee (\log\tau)^{2/3}} \,.
\end{equ}
Note that the prefactor in~\eqref{e:GoodScaling} is chosen in such a way  
that the law $\mathbb P^\tau$ of the mollified white noise $\eta \ast \rho_\tau$, 
with $\rho_\tau(\cdot) \eqdef \tau \nu_\tau^{-1/2} \, \rho(\tau^{1/2} R_\tau^{-1/2} \cdot)$,
is a stationary measure for $u^\tau$. 

%
The next is our main result, which is a \emph{superdiffusive central limit theorem} for~$u^\tau$.
It states that, as~$\tau \to \infty$, the finite dimensional distributions of~$u^\tau$ converge to those of~$u^{\eff}$, 
the solution to an \emph{anisotropic linear stochastic heat equation} 
with \emph{explicit renormalised} coefficients that depend solely on $c_2(F)$ as given in~\eqref{e:c1_c2}.

\begin{theorem} \label{th:CLT}
Let $F$ be such that Assumption~\ref{Assumption:F} holds and assume~$c_2(F) \neq 0$, 
for $c_2(F)$ as in~\eqref{e:c1_c2}. 
Let $\Q$ be a probability measure, absolutely continuous with respect to the law of a spatial 
white noise $\mathbb P$ and 
for $\tau>0$, let $(u^\tau(t))_{t \geq 0}$ be as in~\eqref{e:GoodScaling} 
with initial condition $u^\tau(0) =  \rho_\tau \ast u_0$ for $u_0 \eqlaw \mathbb Q$. 
Then, for any~$k \in \N$, times~$0 \leq t_1 \leq t_2 \leq \ldots \leq t_k$ and test 
functions~$\phi_1, \dots, \phi_k \in \cS(\R^2)$, we have
\begin{equ}
	\del[1]{\scal{u^\tau(t_1),\phi_1}, \ldots, \scal{u^\tau(t_k),\phi_k}}
	\overset{\rm law}{\Longrightarrow}
	\del[1]{\scal{u^{\eff}(t_1),\phi_1}, \ldots, \scal{u^{\eff}(t_k),\phi_k}}\qquad \text{as $\tau\to\infty$,}
\end{equ} 
where~$u^{\eff}$ solves 
\begin{equ} \label{e:limitingSHE}
	\partial_t u^{\eff} = \frac{1}{2} \nabla \cdot D^{\eff} \nabla u^{\eff} + \nabla \cdot \sqrt{D^{\eff}} \vec{\xi}
\end{equ}
started at~$u^{\eff}(0) = u_0$, driven by a $2$-dimensional space-time white noise $\vec \xi = (\xi_1, \xi_2)$, and with diffusion matrix $D^{\eff}$ given by
\begin{equ} \label{e:Dlim}
	\quad D^{\eff} \eqdef \begin{pmatrix}
	\displaystyle \frac{3^{2/3}}{2 \pi^{2/3}} \, |c_2(F)|^{4/3} & 0 \\
	0 & 1
	\end{pmatrix} \, .
\end{equ}
\end{theorem}

In fact, Theorem~\ref{th:CLT} follows from the more general Theorem~\ref{th:conv_semigroup} below, 
whose statement requires some additional notation. 
Let $(P^\tau_t)_{t \geq 0}$ (resp. $(P^\eff_t)_{t \geq 0}$) be  
the Markov semigroups of $(u^\tau_t)_{t \geq 0}$ (resp. $(u^\eff_t)_{t \geq 0}$). 
For an initial condition $u_0 \in \cS'(\R^2)$, we denote by $\BP_{u_0}^\tau$ (resp. $\BP^{\eff}_{u_0}$) 
the law of the process $(u^\tau_t)_{t \geq 0}$ (resp. $(u_t^{\eff})_{t \geq 0}$) started at $u_0^\tau = \rho^\tau * u_0$ 
(resp. $u_0$), which is well defined for $\mathbb P(\dd u_0)$-almost every initial condition.
Let~$\BE_{u_0}^\tau$ (resp. $\BE^{\eff}_{u_0}$) be the corresponding expectation. 
Finally, for~$\theta \in [1,\infty)$,
in order to compare elements~$G^\tau \in \L^\theta(\P^\tau)$ (on which~$P^\tau$ acts) to elements in~$\L^\theta(\P^{\eff})$ (on which~$P^{\eff}$ acts), we define the isometric embedding~$\iota_\tau: \L^{\theta}(\P^\tau) \to \L^\theta(\P)$ which maps the functional~$\eta^\tau \mapsto G(\eta^\tau)$ onto~$\eta \mapsto G(\rho^\tau * \eta)$, see Definition~\ref{def:iotatau} below for details.
By abuse of notation, we will say that the sequence~$(G^\tau)_\tau$ with~$G^\tau \in \L^\theta(\P^\tau)$ converges to~$G \in \L^\theta(\P)$ if~$\iota_\tau G^\tau \to G$ in~$\L^\theta(\P)$ as~$\tau \to \infty$. 

\begin{theorem} \label{th:conv_semigroup}
	In the same setting of Theorem~\ref{th:CLT}, let~$\theta \in [1,\infty)$.
	Then, for any sequence~$(G^\tau)_\tau$ which converges to~$G$ in~$\L^\theta(\P)$, the sequence~$(P^\tau_t G^\tau)_{\tau > 0}$ converges to $P_t^{\eff} G$ uniformly in time, that is, 
	\begin{equ}\label{e:semigroups_L_theta}
		\sup_{t \geq 0} \E\sbr[1]{\abs[0]{\iota_\tau P_t^\tau G^\tau - P_t^{\eff}G}^\theta} \longrightarrow 0 \quad \text{as} \quad \tau \to \infty \,.
	\end{equ}
	As a consequence, for every~$k \geq 1$ and continuous functions $f \colon \R^k \to \R$ 
	with at most polynomial growth, test functions $\phi_1, \dots, \phi_k \in \cS(\R^2)$, 
	and times $0 \leq t_1 \leq \dots \leq t_k$ 
	we have
 	\begin{equ} \label{e:CLT}
 		\int  \, \abs[2]{\BE_{u_0}^\tau\big[f\big(u^{\tau}_{t_1}(\phi_1), \dots, u^{\tau}_{t_k}(\phi_k)\big)\big] - \BE_{u_0}^{\eff}\big[f\big(u^{\eff}_{t_1}(\phi_1), \dots, u^{\eff}_{t_k}(\phi_k)\big)\big]}^\theta \mathbb P(\dd u_0) \to 0 \, ,
 	\end{equ}	
 	where, furthermore, the convergence is uniform in $0 \leq t_1 \leq \dots \leq t_k$.
\end{theorem}
%


\subsection{Discussion of the main results} \label{s:discussion}

Theorems~\ref{th:CLT} and~\ref{th:conv_semigroup} are the analogues of~\cite[Thm. 1.3 and 1.4]{CMT} and, as such, 
they not only determine the large-scale behaviour of~\eqref{e:GeneralisedSBE} but also 
provide information on the way in which such behaviour is achieved. 
Indeed, 
the uniformity in time of~\eqref{e:semigroups_L_theta} (due to the mixing properties 
of the limit $u^\eff$) guarantees a control 
on the law of $u^\tau$ {\it uniformly across all time scales}. The convergence of the finite-dimensional 
distributions that~\eqref{e:CLT} gives is in $\L^\theta(\P)$ for any $\theta\in[1,\infty)$ 
with respect to the initial condition, 
so that is {\it stronger than annealed} (i.e. the convergence in~\eqref{e:CLT} for $\theta = 1$ but \emph{without} the absolute value) and {\it stable} (and that of~\cite[Def. 2.6]{KLO}). 

Furthermore, the above statement realises our main goal as it gives  \emph{precise meaning} 
to the heuristic perturbative argument presented in Section~\ref{s:background}. 
Indeed, combining it with the above-mentioned~\cite[Thm. 1.3]{CMT}, it shows that, 
as $\tau\to\infty$, $u^\tau$ behaves as the quadratic SBE in~\eqref{e:SBE} but in which 
$\vec{f_1}$ is {\it not} that in~\eqref{e:PerturbativeExpansion} but $c_2(F)$ in~\eqref{e:c1_c2} 
which, in general, is not $0$ even if $F''(0)$ is. In other words, higher order terms 
of the {\it Taylor} expansion~\eqref{e:PerturbativeExpansion} {\it do matter}. 
Even though such phenomenon has already been observed in the literature 
(see Section~\ref{s:literature} for references), 
let us comment on how it arises and where such constant comes from, 
as this will also give us the chance to briefly discuss how the proof proceeds.
\medskip

The stationary solution to \eqref{e:GeneralisedSBE} satisfies 
$u(t,x) \eqlaw \pi_1$, for $\pi_1$ as in~\eqref{e:c1_c2}, for any given $(t,x) \in \R_+ \times \R^2$.
In view of Proposition~\ref{p:AssF}, Assumption~\ref{Assumption:F} ensures that $F$ (and all of its 
derivatives) belong to $L^2(\pi_1)$, so that it admits the Hermite decomposition
\begin{equs} \label{e:DecompositionF}
	F = \sum_{m \geq 0} c_m(F) \, H_m\,, \quad c_m(F) \eqdef m! \langle F, H_m \rangle_{L^2(\cN(0,1))} \, .
\end{equs}
where $(H_m)_m$ are the classical Hermite polynomials (see~\eqref{e:HermitePol}). 
Note that this decomposition of~$F$ is very different from a Taylor expansion:
for example, \emph{all even monomials} contribute to the constant~$c_2(F)$ and so it is not necessarily~$0$, 
even if~$F''(0) = 0$.
%
%

Now, let us discuss the contribution of the various terms in the above expansion. 
W.l.o.g. we can take $c_0(F)=0$ since the derivative in~\eqref{e:Regularization1} would {annihilate} it anyway. 
%
{In contrast,} the \emph{linear term} ($m=1$) contributes to~$\cN^1(u)$ by 
\begin{equ} \label{e:m=1contribution}
		c_1(F) \, \partial_1 u + c_1(F) \, \partial_1 (\rho^{*2} - \delta_0) * u \, .
	\end{equ}
The first summand is a transport term which was removed via the Galilean transformation (see~\eqref{e:GoodScaling}). 
The second is an artefact of our regularisation procedure and, as we shall see later, it is negligible at large scales.
Thus, the bulk of the proof of Theorems~\ref{th:CLT} and~\ref{th:conv_semigroup} boils down to 
showing that {\it all terms of degree $m \geq 3$ have a negligible effect} on the large-scale dynamics 
while the \emph{quadratic term} (provided $c_2(F) \neq 0$) gives a dominant superdiffusive contribution.\footnote{In contrast, when $c_2(F) = 0$, we expect the system~\eqref{e:GeneralisedSBE} to be \emph{diffusive} at large scales. }
Making this rigorous is highly non-trivial and a discussion of the difficulties 
we need to overcome, together with a heuristic as to how we will rigorously do so, can be found in Section~\ref{sec:Ideas}. 

{In summary,}
the naive perturbative argument in~\eqref{e:PerturbativeSBE} misses a crucial structural fact 
about the large-scale behaviour of~\eqref{e:GeneralisedSBE}: 
what matters in the limit is not the quadratic term in the {\it Taylor} expansion of the nonlinearity, 
but {the one} in its \emph{Hermite expansion}.

\begin{remark}
Despite the failure of~\eqref{e:PerturbativeSBE} to give the fully correct prediction, let us make the following 
observation. If one were to work with general $\sigma \eqdef \norm[0]{\rho}_{L^2(\R^2)} > 0$ 
and decompose $F \in L^2(\pi_\sigma)$ in terms of the rescaled Hermite polynomials~$\sigma^m H_m(\cdot/\sigma)$, 
then $\sigma^m H_m(X/\sigma) \to X^m/m!$ 
as~$\sigma \to 0$. In other words, 
the Hermite and Taylor decompositions agree in the limit when the mollifier~$\rho$ becomes flat. 
\end{remark}

Let us conclude this paragraph by commenting on Assumption~\ref{Assumption:F}. 
As will be clearer from the analysis in the following sections, the reason why we require it is that we need 
the coefficients of the Hermite expansion of $F$ in~\eqref{e:DecompositionF} to decay sufficiently fast in $m$. 
That said, in Appendix~\ref{a:DecayCoeff} (and in particular Lemma~\ref{lem:Analytic}), we show that such assumption 
holds for any analytic function since the growth 
of their derivatives is at most $n!\ll \exp(\exp(o(n)))$; for example, $F(y)= \sqrt{1  + y^2}$ (see Corollary~\ref{coro:examples_F}). We mention the last $F$  
explicitly as it has played a pivotal role in the original derivation 
of the KPZ equation from the Eden model~\cite{KPZ}. 

\subsection{Comparison to the literature} \label{s:literature}
We conclude the introduction by comparing our result with those present in the 
literature. 

\begin{itemize}[leftmargin=*, noitemsep, topsep=1.5pt]
\item {\it Phenomenological derivation of quadratic SBE in dimension $d=1$ at small scales.} In $d=1$, all the existing 
results focus on {\it small scales}, i.e. the goal is to derive the solution of the SBE itself (and not its large-scale 
behaviour prescribed by the derivative of KPZ Fixed Point) 
as a suitable scaling limit of~\eqref{e:FormalGeneralizedSBE}. 
To do so, one has to tame the exploding factor $\tau^{1/2}$ which appears 
in front of the quadratic term in the diffusive rescaling (see~\eqref{e:ZoomingSBE}) 
and thus the ``microscopic'' model considered is~\eqref{e:FormalGeneralizedSBE} 
but in {\it weak coupling}, i.e. with a scale-dependent coefficient $\tau^{-1/2}$ multiplying $F$. 
In this setting, the heuristic argument in Section~\ref{s:background} 
was first rigorously formalised in the pioneering work~\cite{hairer_quastel_18}, 
where it is shown that~\eqref{e:FormalGeneralizedSBE} with $F$ an
even polynomial converges to the solution of the SBE with an effective quadratic non-linearity 
whose strength is the $c_2(F)$ as in~\eqref{e:c1_c2}.  
Several authors have since extended their work to different classes 
of non-Gaussian driving noises~\cite{HS17, SX18}, general nonlinearities~\cite{HX18, HX19, FG19, KZ22}, 
or both~\cite{KWX24}, as well as more general smoothing mechanisms~\cite{erhard_xu_22}.  
While all of these generalisations present demanding technical challenges 
that lead to interesting and novel results, in contrast to ours, 
they are restricted to the~\emph{subcritical dimensions} in the~\emph{weak coupling regime} 
and {\it finite volume}. 
In particular, equations in this category are within the scope of the pathwise solution theories~\cite{Hai14, GIP15} 
we mentioned earlier which generally lead to stronger convergence results 
and typically non-Gaussian limits \emph{without} the additional assumption 
that the equation is started from stationarity.

Another class of results in $d=1$ is based on 
the notion of energy solutions in~\cite{GJ14, GP18}, which, as for us, requires 
a rather explicit knowledge of the invariant measure already at the microscopic level. 
In particular, Gubinelli and Perkowski in~\cite{GP16} 
have given another proof of~\cite{hairer_quastel_18} in the case of a general Lipschitz $F$, 
and their work was later generalised to 
non-stationary initial conditions and infinite volume in~\cite{Y23, Y25}. 
Roughly speaking, their argument relies on the fact that the time-averaged non-linearity, 
in probabilistic square norm, gains \emph{strictly more than one power} of the time-increment 
which they then leverage to get tightness (and ultimately convergence). 
It is easy to convince oneself that it is precisely at the critical dimension that this strategy breaks down, 
as one would only gain \emph{exactly one power} by the same argument. 
On the technical side, one ramification of this observation is that, in contrast to~\cite{GP16}, 
we have to control the action of the generator (and of the resolvent) much more thoroughly, 
see Section~\ref{sec:Ideas} for a more precise discussion of this and further aspects at the heart of our strategy.

\item {\it SPDEs and other models at the critical dimension.} Turning to 
SPDEs and other statistical mechanics models at the critical dimension, recent years have witnessed 
a flurry of activity most of which though (a) work in  {\it weak coupling scaling} (compared to $d=1$, 
the factor that multiplies the non-linearity, the noise, or the initial condition is $1/\sqrt{\log\tau}$), 
and (b) feature a {\it polynomial} (mostly quadratic) non-linearity. This is the case of the KPZ 
equation~\cite{CD,Gu2020,CSZ}, the AKPZ equation~\cite{CES,AKPZweak}, the SBE~\cite{CGT}, 
the fractional KPZ equation~\cite{GPP}, 
the (Landau--Lifshits--)Navier--Stokes equation~\cite{CK, KRYZ}, 
(self-)interacting diffusions~\cite{CG,yang2024weak}\footnote{Even though these works do not 
directly discuss SPDEs but SDEs, quadratic SPDEs still appear at the level of the so-called {\it environment 
seen by the particle}. }, 
the Allen--Cahn equation~\cite{GRZallen}, 
and the multiplicative linear Stochastic Heat Equation in $d=2$ 
~\cite{BC, CSZ1, GHL, caravenna2023critical, Tsai}. 
Notable results that avoid (a) (but satisfy (b)) 
include the so-called DCGFF model of~\cite{TothValko, CHT, Morfe, armstrong} 
and the derivation of scaling exponents for ASEP~\cite{Yau}, AKPZ~\cite{CET}, SBE~\cite{de2024log}, 
while those that avoid (b) (but satisfy (a)) include the non-linear stochastic heat equation~\cite{DG, Tao, DHL} and the Allen--Cahn equation~\cite{CasDun}. 

Compared to the above results, ours is the first in {\it strong coupling} with a {\it generic non-linearity} satisfying 
a weak growth condition, 
but we still require to be close to the (explicit) invariant measure already at the microscopic scale and 
impose the non-linearity to be anisotropic (i.e. the flow velocity $\vec{w}$ in~\eqref{e:ConservationLaw} 
is unidirectional). 
\end{itemize}

\subsection*{Organisation of the article}
The rest of the article is organised as follows. 
In Section~\ref{s:preliminaries}, we recall basic tools from Wiener space analysis and introduce the notations 
we will use throughout the article.
In Section~\ref{sec:WellPosed}, we introduce a suitable notion of solution to the microscopic
equation~\eqref{e:GeneralisedSBE} for a fixed value of~$\tau > 0$, show that such a solution exists and is unique, 
and prove that it defines a stationary, skew-reversible Markov process whose invariant measure is a smoothened 
spatial white noise.
%
%
In Section~\ref{sec:generator}, we then obtain properties of the Markov semigroup associated with~\eqref{e:GeneralisedSBE}, identify a sufficiently rich set in the domain of its generator, and explicitly describe the action of the latter on the corresponding Fock spaces.
The key challenge here is the fact that the nonlinearity has components in every chaos: We overcome it by a polynomial approximation argument for the antisymmetric part of the generator and present a novel, direct proof, based on the It\^{o} trick, that the resolvent is continuous between suitable spaces; we believe the latter to be of independent interest.
Section~\ref{s:proof_main} contains the proofs of our main results which, on a heuristic level, we anticipate in Section~\ref{sec:Ideas}. 
Therein, we present a high-level overview of the major challenges, outline the ideas to overcome them, and describe a roadmap for their implementation.
Section~\ref{s:gen_est_nonq} comprises novel estimates on the non-quadratic part of the generator while 
Section~\ref{s:ansatz_QSBE} presents a new ansatz for the resolvent of the quadratic~SBE 
which is truncated both in chaos and in Fourier. 
Finally, we collect all previous results in Section~\ref{s:pf_summary} to obtain the desired resolvent convergence 
and then show how it implies Theorems~\ref{th:CLT} and~\ref{th:conv_semigroup}.      
At last, Appendix~\ref{a:DecayCoeff} collects some facts about the connection between the regularity of 
a function and the decay of its Hermite coefficients, complemented by some representative examples 
of functions satisfying Assumption~\ref{Assumption:F}.
Appendix~\ref{a:technical} contains some technical proofs concerning the antisymmetric part of the generator in the polynomial case.

\subsection*{Acknowledgements}
G.~C. gratefully acknowledges financial support
via the UKRI FL fellowship ``Large-scale universal behaviour of Random
Interfaces and Stochastic Operators'' MR/W008246/1. 
TK is supported by a UKRI Horizon Europe Guarantee MSCA Postdoctoral Fellowship (UKRI, SPDEQFT, grant reference EP/Y028090/1). Views and opinions expressed are however those of the authors only and do not necessarily reflect those of UKRI. 
In particular, UKRI cannot be held responsible for them.
The research of Q.~M. was funded 
by the Austrian Science Fund (FWF) 10.55776/F1002. 
For open access purposes, the authors have applied a CC BY public copyright 
license to any author accepted manuscript version arising from this submission.

\section{Preliminaries} \label{s:preliminaries}

Let us collect in this section a few preliminary tools and notations which are mainly taken from~\cite[Section 1.3]{CMT}. 

\subsection{Notations, Function spaces and Wiener space analysis}\label{sec:Pre}

Let $\fe_1$ and $ \fe_2$ be the canonical basis vectors of $\R^2$ and $|\cdot|$ be
the usual Euclidean norm. For $I$ a finite subset of $\N$ of cardinality $|I|$ and a set of $|I|$ vectors 
$\{p_i\in\R^2\colon i\in I\}$, we set $p_I\in\R^{2|I|}$ to be the vector $(p_i)_{i\in I}$ and,
given a $2 \times 2$ matrix $M$, $M p_{I}\eqdef(M p_i)_{i\in I}$.
We write (for $e \in \R^2$)
\begin{equ} \label{e:NotationsR2n}
p_{[I]} = \sum_{i\in I} p_i \, , \quad |p_{I}| \eqdef \Big(\sum_{i\in I}|p_i|^2\Big)^\frac{1}{2}\,,\qquad |e \cdot p_{I}| \eqdef \Big(\sum_{i\in I} |e \cdot p_i|^2\Big)^\frac{1}{2} \,.
\end{equ}
For $\tau>0$, let the matrix $R_\tau$ and $\nu_\tau > 0$ be respectively given according to
\begin{equ}[e:nutau]
R_\tau\eqdef\begin{pmatrix}
\dis \nu_\tau & 0 \\
0 & 1
\end{pmatrix}\qquad\text{and}\qquad\nu_\tau \eqdef \frac{1}{\big(1 \vee \log\tau\big)^{2/3}} \, ,
\end{equ}
and note that it holds
\begin{equ} \label{e:NotationsQ}
|\sqrt{R_\tau} p_{I}|^2 = \nu_\tau |\fe_1 \cdot p_{I}|^2 + |\fe_2 \cdot p_{I}|^2 \, .
\end{equ}
In the specific case of $I=\{1,\dots,n\}$, we write $I=1:n$ so that, in particular, $p_{1:n}=(p_1,\dots,p_n)$. 

We denote by $\cS(\R^d)$ the Schwartz space of smooth functions whose derivatives decay faster
then any polynomial and by $\cS'(\R^d)$ its dual. The Fourier transform of $\phi\in\cS(\R^d)$ is
\begin{equ}
\cF(\phi)(p) = \hat \phi(p) \eqdef \frac{1}{(2 \pi)^{d/2}} \int_{\R^d} \phi(x) \, e^{- \iota \, p \cdot x} \, \dd x \, ,
\end{equ}
while, for $f\in\cS'(\R^d)$, it is given via duality.

Let $\rho \in \cS(\R^2)$ be a given mollifier, i.e. $\rho$ is 
smooth, compactly supported, even-symmetric, of mass $1$ and of $L^2(\R^2)$-norm $1$. 
For $\tau>0$, we set
\begin{equ}[e:Mollifiers]
\rho_\tau(\cdot) \eqdef \tau \nu_\tau^{-1/2} \, \rho(\tau^{1/2} R_\tau^{-1/2} \cdot)\,,\quad\text{and}\quad
\Theta_\tau(\cdot) \eqdef (2 \pi \hat{\rho}_\tau(\cdot))^2=(2 \pi \hat{\rho}(\tau^{-1/2} R_\tau^{1/2} \cdot))^2\,,
\end{equ}
and, for $n\in\N$, we further define the measure $\Xi_n^\tau$ on $\R^{2n}$ according to
\begin{equ}[e:RegMeas]
\Xi_n^\tau(\dd p_{1:n}) \eqdef \prod_{i=1}^n \Theta_\tau(p_i) \, \dd p_{1:n} \, ,\qquad p_{1:n}\in\R^{2n}
\end{equ}
with the convention that, for $\tau = \infty$, $\Theta_\tau \equiv 1$ and $\Xi_n=\Xi_n^\infty$ is the Lebesgue measure.

For $\tau> 0$ and $\gamma\in\R$, let $H^\gamma_\tau(\R^{2n})$ be the $\tau$-dependent
(anisotropic) Sobolev space with respect to $\Xi_n^\tau$, i.e. $H^\gamma_\tau(\R^{2n})$ is 
the completion of $\cS(\R^{2n})$ under the norm
\begin{equ} \label{e:normHeps}
\|f\|_{H^\gamma_\tau(\R^{2n})}^2 \eqdef \int \Big(1 + \frac{1}{2} \, |\sqrt{R_\tau} p_{1:n}|^2\Big)^\gamma \, |\hat f(p_{1:n})|^2 \, \Xi_n^\tau(\dd p_{1:n}) \, .
\end{equ}
For $\gamma = 0$, $H^0_\tau(\R^{2n})=L_\tau^2(\R^{2n}) \eqdef L^2(\R^{2n}, \Xi_n^\tau)$, with scalar product $\langle \cdot, \cdot \rangle_{L^2_\tau(\R^{2n})}$.
\medskip

Finally, we will write $a \lesssim b$ if there exists a constant $C > 0$ independent of any quantity relevant for the result,
such that $a \leq C b$ and $a \asymp b$ if $a \lesssim b$ and $b \lesssim a$.
If we want to highlight the dependence of the constant $C$ on a specific quantity $Q$, we write instead $\lesssim_Q$.

\subsection{Gaussian white noise and Fock spaces} 
\label{sec:notation}

Let $\eta$ be a space white noise on $\R^2$, 
and, for $\tau\geq 0$, set $\eta^\tau\eqdef\rho_\tau\ast\eta$ (with $\eta^\infty\equiv\eta$).
Then, $\eta^\tau$ (whose law is denoted by $\P^\tau$ is a centred Gaussian field with covariance
\begin{equ}
\E^\tau[\eta^\tau(\phi)\eta^\tau(\psi)] =\langle\phi,\psi\rangle_{L_\tau^2(\R^2)} = \int_{\R^{2}}\hat \phi(p)\,\overline{\hat \psi(p)}\,\Xi_1^\tau(\dd p)\,,\qquad \phi,\psi\in L^2_\tau(\R^2)\,.
\end{equ}
Cylinder random variables $\phi\in\L^\theta(\P^\tau)$ are random variables 
of the form $\phi=f(\eta^\tau(h_1),\dots,\eta^\tau(h_n))$, with $f\colon \R^n\to\R$ smooth and
growing at most polynomially at infinity, and $h_1,\dots,h_n\in\cS(\R^2)$.  
The set of cylinder random variables is dense in
$\L^\theta(\P^\tau)$, for $\theta<\infty$.
\medskip

By~\cite[Theorem 1.1.1]{Nualart},
$\L^2(\P^\tau)= \overline{\bigoplus_{n \geq 0} \SH_n^\tau}$, where
$\SH_n^\tau$ is the {\it $n$-th homogeneous Wiener chaos}, i.e.
the closure in $\L^2(\P^\tau)$ of
$\mathrm{Span}\big\{H_n(\eta^\tau(h)) \, | \, h \in \cS(\R^2), \, \|h\|_{L_\tau^2(\R^2)} = 1\big\}$
and $H_n$ is the $n$-th Hermite polynomial (see~\eqref{e:HermitePol}).
Let $I^\tau$ be the canonical isometry onto $\L^2(\P^\tau)$ (see~\cite[Theorem 1.1.2]{Nualart})
obtained as the unique extension of the map that, for every $n\geq 0$,
assigns $h^{\otimes n}$ to $n! \, H_n(\eta^\tau(h)) \in \cH^\tau_n$ for
$h\in\cS(\R^2)$ with $L^2_\tau(\R^2)$-norm equal to $1$.
$I^\tau$ is defined on the Fock space $\fock{}{}{\tau}=\oplus_{n\geq 0} \fock{n}{}{\tau}$,
where $\fock{n}{}{\tau}\subset L^2_\tau(\R^{2n})$ contains 
functions which are symmetric with respect to
permutations of their variables. We endow $\fock{}{}{\tau}$ with the norm
\begin{equ}[e:normFock]
\|f\|_\tau^2=\|f\|_{\fock{}{}{\tau}}^2\eqdef \sum_{n\geq 0} \|f_n\|_{\fock{n}{}{\tau}}^2\eqdef \sum_{n\geq 0} n! \|f_n\|_{L^2_\tau(\R^{2n})}^2\,,
\end{equ}
for $f=(f_n)_{n\geq 0}\in \oplus_{n\geq 0}\fock{n}{}{\tau}$, and define, for $\gamma\in\R$, 
$\fock{}{\gamma}{\tau}$ as the Sobolev space whose norm is the same as that in~\eqref{e:normFock}
but with $\|\cdot\|_{H^\gamma_\tau(\R^{2n})}^2$ replacing $\|\cdot\|_{L^2_\tau(\R^{2n})}^2$. 
Recall that by~\cite[Theorem 1.1.2]{Nualart}, for any $F\in\L^2(\P^\tau)$
there exists $f=(f_n)_n\in\fock{}{}{\tau}$ such that $F=\sum_n I_n(f_n)$ and $\E^\tau[F^2]=\|f\|_\tau^2$, 
where $I^\tau_n$ is the restriction of $I^\tau$ to $\fock{n}{}{\tau}$ (which is itself an isometry). 
\medskip

Thanks to the isometry $I^\tau$, we will abuse notation and identify operators acting on $\L^2(\P^\tau)$
with the corresponding operator acting on $\fock{}{}{\tau}$.
An important notion is that of {\it diagonal} operator (see~\cite[Def. 1.5]{CMT}). 

\begin{definition}(Def. 1.5 in~\cite{CMT})\label{def:DiagOp}
An operator $\cD$ on $\fock{}{}{\tau}$ is said to be {\it diagonal}
if there exists a family of measurable real-valued kernels $\mathbf{d} = (\bd_n)_{n \geq 0}$
such that for all $n\geq 0$ and $\phi\in\fock{n}{}{\tau}\cap\cS(\R^{2n})$, it holds
$\cF{\cD \phi} (p_{1:n}) = \bd_n(p_{1:n}) \, \widehat \phi(p_{1:n})$ for (Lebesgue-almost-) every $p_{1:n} \in \R^{2n}$.
If $\cD^{(1)}$ and $\cD^{(2)}$ are diagonal, we write $\cD^{(1)} \leq \cD^{(2)}$ if their kernels satisfy
$\bd^{(1)} \leq \bd^{(2)}$ and say that $\cD^{(1)}$ is positive if $\cD^{(1)} \geq 0$.
Given a diagonal operator $\cD$ with kernel $\bd=(\bd_n)_{n \geq 0}$
and a measurable scalar function $f$, we define $f(\cD)$ to be the diagonal operator
whose kernel is $f(\bd)=(f \circ \bd_n)_{n \geq 0}$.
\end{definition}


\begin{definition}\label{def:NoOp}
We define the number operator $\cN$ to be the diagonal operator
acting on $\psi=(\psi_n)_n\in\fock{}{}{\tau}$ such that $\|\psi_n\|_\tau$ decays faster than any polynomial in $n$,
as $\cN\psi_n\eqdef n\psi_n$.
\end{definition}

Let us recall the operators $\iota^\tau$ and $j^\tau$ that appeared in \cite{CMT} (see~\cite[Lemma 1.7]{CMT}).

\begin{definition} \label{def:iotatau}
Let us define
$\iota^\tau \colon \L^\theta(\P^\tau) \to \L^\theta(\P)$, $\theta\in[1,\infty)$, to be the operator that maps
the functional $\eta^\tau \mapsto F(\eta^\tau)$ to the functional $\eta \mapsto F(\rho_\tau * \eta)$.
This is a bijective isometry which, for $\theta=2$, translates to a bijective isometry from $\fock{}{}{\tau}$ to $\fock{}{}{\infty}$.
We denote $j^\tau = (i^\tau)^{-1} \colon \fock{}{}{\infty} \to \fock{}{}{\tau}$ its inverse.
\end{definition}

We will also need the product rule for the Wiener-Itô integrals. For $f \in \fock{n_1}{}{\tau}$ and $g \in \fock{n_2}{}{\tau}$, 
it reads
\begin{equ} \label{e:ProdRule}
I_{n_1}^\tau(f) I_{n_2}^\tau(g) = \sum_{k=0}^{n_1 \wedge n_2} k! \binom{n_1}{k} \binom{n_2}{k} I_{n_1+n_2-2k}\big(\mathrm{Sym} (f \otimes_k g)\big) \, ,
\end{equ}
where $\mathrm{Sym}$ is the symmetrization operator and
\begin{equ} \label{e:ProdMixing}
\cF (f \otimes_k g) \, (p_{1:n_1-k}, q_{1:n_2-k}) \eqdef \int \Xi^\tau_k(\dd r_{1:k}) \hat{f}(p_{1:n_1-k}, r_{1:k}) \, \hat{g}(q_{1:n_2-k}, -r_{1:k}) \, .
\end{equ}

\section{Analysis of the microscopic equation} \label{s:microscopic_equation}

Let us begin by collecting some basic results concerning our microscopic model, i.e. 
equation~\eqref{e:GeneralisedSBE}. Compared to the analysis carried out in Section 2 and Appendix A of~\cite{CMT}
the arguments below are more delicate due to the very mild assumptions on the function $F$ 
but, as the setting is similar, we will borrow the same notations (that will be recalled along the way). 
We will first prove well-posedness of the SPDE~\eqref{e:GeneralisedSBE} and subsequently detail 
the derivation of the generator (and semigroup) of its solution. 
\medskip

Throughout the rest of the paper, let us fix a probability space $(\cO,\cF,\Prob)$   
supporting a couple of independent spatial white noise $\eta$ on $\R^2$ and
$2$-dimensional space-time white noise $\vec{\xi}=(\xi_1,\xi_2)$ on $\R_+ \times \R^2$,
whose laws are respectively denoted by $\P$ and $P$, and $\Prob \eqdef \P \otimes P$.
Let $(\cF_t)_{t \geq 0}$ be the filtration in which $\cF_t$ is the (completed)
$\sigma$-algebra generated by $\eta$ and $\restriction{{\vec \xi} \,}{[0,t]}$.
Let $\rho \in \cS(\R^2)$ be the mollifier introduced above
and, for $\tau>0$,
$\rho_\tau$ be as in~\eqref{e:Mollifiers}. 
Let $\eta^\tau \eqdef \eta\ast\rho_\tau$ and $\vec{\xi}^{\tau} \eqdef \vec{\xi}\ast\rho_\tau=(\xi_1\ast\rho_\tau,\xi_2\ast\rho_\tau)$
the rescaled white noises, whose joint law is $\Prob^\tau \eqdef \P^\tau \otimes P^\tau$.

\subsection{Well-posedness for the microscopic equation}\label{sec:WellPosed} 

Our first goal is to prove that the rescaled version of the regularised equation~\eqref{e:GeneralisedSBE} 
admits a unique solution globally in space and time, and that such solution is Markovian and stationary. 
Throughout this subsection the scale parameter $\tau>0$ will be fixed and bounds 
and estimates below will (badly) depend on it. 
By scaling (and shifting) the solution $u$ of~\eqref{e:GeneralisedSBE} as in~\eqref{e:GoodScaling}, 
we obtain 
\begin{equ}[e:GSBE] 
  \partial_t u^\tau = \frac{1}{2} \, \nabla \cdot R_\tau \, \nabla u^\tau + \cN_F^\tau[u^\tau] + \nabla \cdot \sqrt{R_\tau} \, \vec{\xi}^\tau \, ,
\end{equ}
where $R_\tau$ is the $2\times2$ matrix in~\eqref{e:Rtau} and, for $F\colon\R\to\R$, 
the rescaled nonlinearity $\cN_F^\tau$ is 
\begin{equ}[e:GSBEnnlin]
\cN_F^\tau[u^\tau] \eqdef \tau \nu_\tau^{1/4} \, \partial_1 \rho_\tau^{*2} * F(\tau^{-1/2} \nu_\tau^{1/4} u^\tau) - c_1(F) \, \tau^{1/2} \nu_\tau^{1/2} \, \partial_1 u^\tau\,
\end{equ}
for $c_1(F)$ as in~\eqref{e:c1_c2}. 

Even though, {\it apriori}, the function $F$ should be subject to Assumption~\ref{Assumption:F}, for the sake of this section 
it will be sufficient to require that $F\in\cC^1(\R)$ and that there exists $\kappa\in(0,1/4)$ (which will be fixed throughout) such that 
both conditions 
\begin{equs}[e:AssF]
\,&\|F\|_{\exp} \eqdef\sup_{x \in \R} \del[1]{|F(x)|\vee |F'(x)|}  e^{- \kappa x^2} < +\infty\,,\\
&\lim_{|x|\to\infty} \del[1]{|F(x)|\vee |F'(x)|}  e^{- \kappa x^2}=0\,. 
\end{equs}
hold\footnote{The second condition is always satisfied up to taking any larger $\kappa$. }. 
The spaces in which we will look for a solution have to take into account the growth of $F$. Thus, 
for $\beta>0$, we define $\cX_\beta$ to be the metric space completion of smooth compactly supported 
functions on $\R^2$ under the distance 
\begin{equ}[e:distance]
d_\beta(f,g) \eqdef \sum_{k \geq 0} \min\Big\{2^{-k}, \sup_{x \in \R^2} |f(x)-g(x)| \exp(- 2^{-k} \langle x \rangle^\beta)\Big\} \, ,
\end{equ}
where, for $x\in\R^2$, we set $\langle x\rangle\eqdef (1+|x|^2)^{1/2}$. Let us point out that 
if $\beta>\beta'>0$ then $\cX_\beta\subset\cX_{\beta'}$ and that the choice of the distance 
in~\eqref{e:distance} guarantees that if $f\in\cX_\beta$ then {\it for all $C>0$}
\begin{equ}[e:GrowthCond]
\sup_{x\in\R^2}e^{-C\langle x\rangle^\beta}|f(x)|<\infty\,.
\end{equ}
The notion of solution we will be adopting is the following. 

\begin{definition}\label{def:SolSPDE}
Let $\tau>0$, $\kappa\in(0,1/4)$ and $F\in\cC^1(\R)$ be such that~\eqref{e:AssF} holds. 
We say that $u^\tau$ is a mild solution of~\eqref{e:GSBE} driven by $(\eta^\tau,\vec{\xi}^{\, \tau})$
if 
\begin{enumerate}[noitemsep]
\item $u^\tau$ is $(\cF_t)_{t \geq 0}$-adapted, 
\item $\mathbf P$-almost surely, 
$u^\tau \in C(\R_+, \cX_1)$ and, for $\alpha\in(4\kappa,1)$, 
\begin{equ}[e:Nalpha]
t\mapsto \fE_\alpha(u^\tau_t)\eqdef \sup_{x \in \R^2} \frac{e^{\kappa |u^\tau_t(x)|^2}}{\langle x \rangle^{\alpha}}  \in L^1_{\mathrm{loc}}(\R_+), 
\end{equ}
\item for every $t \geq 0$
\begin{equ}[e:SolSPDE]
u^\tau_t = K^\tau_t\eta^\tau + \int_0^t K^\tau_{t-s} \cN_F^\tau[u_s] \, \dd s + \int_0^t K^\tau_{t-s} \, \nabla \cdot \sqrt{R_\tau} \vec{\xi}^{\, \tau}(\dd s) \, ,
\end{equ}
where, for $t\geq 0$,  $K^\tau_t$ is the heat semigroup associated with $\tfrac12 \nabla \cdot R_\tau \, \nabla u^\tau$. 
\end{enumerate}
\end{definition}
Let us briefly comment on Definition~\ref{def:SolSPDE}. The first and the last condition are natural as they 
require $u^\tau$ to be a stochastic process adapted to the filtration induced by the noise 
that indeed solves~\eqref{e:GSBE} (in its mild formulation). 
As we will soon see, the second instead is needed to obtain uniqueness. Even though~\eqref{e:Nalpha} 
might appear quite strong, in our setting it can be shown to hold since the $u^\tau$ 
we will construct is stationary. This means that the $\Prob$-almost sure 
boundedness of the $L^1$-norm (in time) of $\fE_\alpha(u^\tau)$ 
follows by the boundedness of the expectation of $\fE_\alpha(\eta^\tau)$ (and this is the 
reason why we require $L^1$, see Lemma~\ref{l:Nwhitenoise}), 
which in turn can be easily deduced as $\eta^\tau$ is a smooth Gaussian process 
with finite range of dependence.   

We are now ready to state the main result of this subsection. 

\begin{proposition} \label{p:MollifSol}
Let $\tau>0$ be fixed, $\kappa\in(0,1/4)$ and $F\in\cC^1(\R)$ be such that~\eqref{e:AssF} holds. Then, 
there exists a unique mild solution $u^\tau$ of~\eqref{e:GSBE} driven by $(\eta^\tau,\vec{\xi}^{\, \tau})$ 
according to Definition~\ref{def:SolSPDE}. 
Furthermore, for any $h\in\cC_c^\infty(\R^2)$, we have that, $\mathbf P$-almost surely, for every $t \geq 0$ it holds
\begin{equ}[e:WeakFormulation]
u^\tau_t(h) = \eta^1(h) + \frac12\int_0^t u_s(\nabla \cdot R_\tau \, \nabla h) \, \dd s + \int_0^t \cN^\tau_F[u_s](h) \, \dd s + \int_0^t \vec{\xi}^{\, \tau}(\dd s, \sqrt{R_\tau}\nabla h) \, .
\end{equ}
Besides, $u^\tau$ is a Markov process which is stationary, i.e. for every $t \geq 0$, $u^\tau_t \eqlaw \eta^\tau$, and 
skew-reversible, in the sense that for any $T > 0$, under $\Prob$, the process $(u^\tau_{T-t})_{t \in [0,T]}$ 
has the same law as the process $(\tilde u^\tau_t)_{t \in [0,T]}$, 
for $\tilde u^\tau$ the mild solution of \eqref{e:GSBE} driven by 
$(\tilde \eta^\tau,\vec{\tilde \xi}^{\, \tau})\eqlaw (\eta^\tau,\vec{\xi}^{\, \tau})$ 
but with $F$ replaced by $-F$.
\end{proposition}

\begin{remark}\label{rem:Growth}
The reason why the weak formulation of~\eqref{e:GSBE} in~\eqref{e:WeakFormulation} 
is stated for compactly supported smooth functions $\phi$ and not for $\phi\in\cS(\R^2)$ is that, in space, 
the solution $u$ only belongs to $\cX_1$ (so that in particular it might grow at infinity exponentially fast) 
and thus the pairing $u_\cdot(\phi)$ might not be well-defined. 
\end{remark}

As usual, the proof of the previous proposition consists of showing that a solution to~\eqref{e:GSBE}
(1) exists, (2) is unique and (3) enjoys all the stated properties.  
For (1), we will consider the sequence $(u^{\tau,M})_{M\in\N}$, where $u^{\tau,M}$ solves~\eqref{e:GSBE} 
on a torus of side-length $M$ (and then extended periodically to $\R^2$) whose nonlinearity is not $F$ 
but a polynomial $P_M$ with $P_M$ converging to $F$ is a suitable sense (see Lemma~\ref{l:DensityPolynomial}). 
We will show that 
such sequence is Cauchy and that its limit satisfies~\eqref{e:SolSPDE} (and~\eqref{e:WeakFormulation}), 
thus obtaining (1) 
and, as a byproduct, (3), since, by the arguments presented in~\cite[Appendix A]{CMT},  
it is not hard to prove that each of the $u^{\tau,M}$ is Markov, stationary and skew-reversible 
(and this is the reason why we chose such approximation) and these properties are carried through in the limit. 

The next lemma allows us to simultaneously prove that the sequence $(u^{\tau,M})_{M\in\N}$ 
is Cauchy and that the solution to~\eqref{e:GSBE} is unique, as it shows that the solution map is Lipschitz continuous  
with respect to {\it both} the driving noise {\it and} the nonlinearity, thus allowing us to  
carry out the outlined strategies for (1) and (2) (and (3)) at once. 

\begin{lemma}\label{l:Apriori}
Let $\tau>0$ be fixed, $T > 0$, $F, G\in\cC^1(\R)$ satisfying~\eqref{e:AssF}, $\beta\in(0,1]$ and 
$X^\tau, Y^\tau\in C([0,T], \cX_\beta)$. 
Assume there exist $u^\tau, v^\tau\in C([0,T], \cX_\beta)$ such that, for $\alpha\in(0,\beta)$, 
$\fE_\alpha(u^\tau), \fE_\alpha(v^\tau)\in L^1(0,T)$ and, for every $t \in [0,T]$, they respectively satisfy
\begin{equ} \label{e:MildApriori}
u^\tau_t = X^\tau_t + \int_0^t K^\tau_{t-s} \cN^\tau_F[u^\tau_s] \, \dd s \, , \qquad v^\tau_t = Y^\tau_t + \int_0^t K^\tau_{t-s}\cN^\tau_G[v^\tau_s] \, \dd s \, .
\end{equ}
Then there exists a constant $C=C(\tau, T,\alpha,\beta)>0$ that depends continuously on $\|F\|_{\exp},$ 
$\|\fE_\alpha(u^\tau)\|_{L^1(0,T)}$ and $\|\fE_\alpha(v^\tau)\|_{L^1(0,T)}$ (where $\|\cdot\|_{\exp}$ and $\fE_\alpha(\cdot)$ 
are respectively defined in~\eqref{e:AssF} and~\eqref{e:Nalpha}), 
such that for every $t \in [0,T]$ and $x \in \R^2$
\begin{equ}[e:Apriori]
|u^\tau_t(x) - v^\tau_t(x)| \leq C e^{C \langle x \rangle^\beta} \, \Big(\|F - G\|_{\exp} + \sup_{(t,x) \in [0,T]\times\R^2} e^{- \langle x \rangle^\beta}\, |X^\tau_t(x) - Y^\tau_t(x)| \Big)\, .
\end{equ}
\end{lemma}
\begin{proof}
Throughout the proof, we will denote by $C$ a constant as in the above statement, 
but whose value may increase from line to line and, since the specific value of $\tau$ is irrelevant, we will 
take it to be $1$ and omit the corresponding index (e.g. we write $u, X$ for $u^\tau, X^\tau$ etc.). 
Recall that, for~$\tau = 1$, we have~$\nu_\tau = 1$ and thus~$R_\tau = \operatorname{Id}_{2}$ and~$K_t \equiv K_t^1$ is the classical heat semigroup w.r.t.~$\frac{1}{2}\Delta$.

Let us define $w \eqdef u - v$. Then, from \eqref{e:MildApriori} and the definition of $\cN^1$ in~\eqref{e:GSBEnnlin}, 
we see that for every $t \in [0,T]$, $w$ equals
\begin{equs}
w_t &= \int_0^t K_{t-s} \big(\cN_F^1[u_s]-\cN_F^1[v_s]\big) \, \dd s + \int_0^t K_{t-s} \, \big(\cN_F^1[v_s]-\cN_G^1[v_s]\big) \, \dd s +\big(X_t - Y_t\big)\\
&=:\one_t+\two_t+\three_t\label{e:w}
\end{equs}
and we will separately estimate each of the terms above. Let us begin with $\three$, which is the easiest 
as, for every $x\in\R^2$, we immediately have 
\begin{equ}[e:K]
|\three_t(x)| \leq \, e^{\langle x\rangle^\beta} \sup_{(t,y) \in [0,T]\times\R^2} e^{- \langle y \rangle^\beta} \, |X_t(y) - Y_t(y)|\, . 
\end{equ}

The bound on $\two$ is an immediate consequence of two bounds. The first 
follows by the definition of $\|\cdot\|_{\exp}$ in~\eqref{e:AssF} and of $\fE_\alpha$ in~\eqref{e:Nalpha}, 
and reads 
\begin{equ}
|(F-G)(v_s(x))| \leq \|F - G\|_{\exp} \exp(\kappa v_s(x)^2) \leq \|F - G\|_{\exp} \fE_\alpha(v_s) \, \langle x \rangle^{\alpha} \, ,
\end{equ}
while the second is a simple convolutional bound (which can be obtained by using e.g. Feynman-Kac formula) 
that takes the form 
\begin{equ}[e:ConvBound]
K_{r} \, \big[|\partial_1 \rho^{*2}| * \langle \cdot \rangle^{\alpha}\big](x) \leq C \langle x \rangle^{\alpha} \,, 
\end{equ}
for some constant $C>0$, uniformly over $r\in [0,T]$. 
Putting them together, we obtain 
\begin{equ}[e:J]
|\two_t(x)| \leq C \|\fE_\alpha(v)\|_{L^1(0,T)} \, \|F - G\|_{\exp} \, \langle x \rangle^\alpha \leq C \, \|F - G\|_{\exp} \, e^{\langle x\rangle^\beta} \,,
\end{equ}
where the constant $C$ in the last bound clearly depends (linearly) on $\|\fE_\alpha(v)\|_{L^1(0,T)}$. 

We now turn to $\one$, which is the most delicate. We first notice that
\begin{equs}[e:FI]
|F(u_s(x)) - F(v_s(x))| &\leq \|F\|_{\exp} \big(e^{\kappa u_s(x)^2} + e^{\kappa v_s(x)^2}\big) \, |w_s(x)| \\
&\leq \|F\|_{\exp} (\fE_\alpha(u_s) + \fE_\alpha(v_s)) \, \langle x \rangle^\alpha \, |w_s(x)| \, .
\end{equs}
Now, since both $u$ and $v$ satisfy~\eqref{e:GrowthCond} by assumption, so does $w$, which 
means that morally we should have $|w_s(x)|\lesssim e^{A\langle x\rangle^\beta+B}$ for some $A, B>0$. 
If we plug this into the previous bound and back to the definition of $I$ in~\eqref{e:w}, 
we will need to estimate 
\begin{equ}[e:ConvBound2]
K_{r} \,  \big[|\partial_1 \rho^{*2}| \ast\big(\langle \cdot \rangle^\alpha e^{A \langle \cdot \rangle^\beta + B}\big)\big](x) \leq e^{C'(1+A^2)} \, \langle x \rangle^\alpha e^{A \langle x \rangle^\beta + B}
\end{equ}
where, for some $C'=C'(T,\beta)>0$, the bound is uniform over~$r \in [0,T]$  and can be derived similarly to~\eqref{e:ConvBound}. At this point though, we see a problem. 
Indeed, while we are assuming $w$ to grow as $e^{A\langle x\rangle^\beta+B}$, at the 
r.h.s. there is an extra polynomial factor in $x$ so that in principle 
it grows {\it faster} than what is expected and it is unclear how to close the bound. 
What we have neglected is the integral in time, and to exploit it we will 
choose $A$ and $B$ to be {\it time-dependent}. 
More specifically, for $A, B\colon [0,T]\to\R$ to be chosen afterwards and $s\in[0,T]$, set  
\begin{equ}[e:W]
W_s\eqdef \sup_{x \in \R^2} |w_s(x)| e^{-A_s \langle x \rangle^\beta - B_s}\qquad \text{and}\qquad W_t^\ast\eqdef \sup_{s \in [0,t]} W_s\,.
\end{equ}
With these definitions at hand together with~\eqref{e:FI} and~\eqref{e:ConvBound2}, we can estimate $\one$ as 
\begin{equs}
|\one_t(x)| &\leq \int_0^t \dd s \, W_s \, \|F\|_{\exp} (\fE_\alpha(u_s) + \fE_\alpha(v_s)) \, e^{C'(1+A_s^2)}\langle x \rangle^\alpha  e^{A_s \langle x \rangle^\beta + B_s} \, \\
&\leq W_t^* \int_0^t \dd s \, \|F\|_{\exp} (\fE_\alpha(u_s) + \fE_\alpha(v_s)) \, \big[\langle x \rangle^\beta + e^{C'(1+A_s^2)}\big] e^{A_s \langle x \rangle^\beta + B_s}\,.\label{e:I1}
\end{equs}
At this point, we want to choose $A$ and $B$ in such a way that we can control the time integral 
by a small constant (say $1/2$) times the exponential of $A_t \langle x \rangle^\beta + B_t$ so that, 
once we get back to~\eqref{e:w}, we can reabsorb this term at the l.h.s. . 
A convenient choice is to set, for every $t\in[0,T]$, 
\begin{equs}
A_t &= 1 + 2 \int_0^t \|F\|_{\exp} (\fE_\alpha(u_s) + \fE_\alpha(v_s)) \dd s \,, \\
B_t &= 2 \int_0^t \|F\|_{\exp} (\fE_\alpha(u_s) + \fE_\alpha(v_s))  e^{C' (1+A_s^2)}  \dd s
\end{equs}
as these are well-defined, non-negative, non-decreasing and bounded by some constant $C$ depending 
(continuously) on $T,\alpha,\beta, \|F\|_{\exp},$ 
$\|\fE_\alpha(u)\|_{L^1(0,T)}$ and $\|\fE_\alpha(v)\|_{L^1(0,T)}$, and, even more importantly,  
allow us to take care of the polynomial factor. Indeed, from~\eqref{e:I1} we can bound $\one$ as     
\begin{equ}[e:I]
|\one_t(x)| \leq W_t^* \int_0^t \dd s \, \frac{1}{2} \big(\dot{A}_s \langle x \rangle^\beta + \dot{B}_s\big) e^{A_s \langle x \rangle^\beta + B_s} \leq \frac{1}{2} \, W_t^\ast \,e^{A_t \langle x \rangle^\beta + B_t} \, ,
\end{equ}
which concludes the bound on $\one$. 

We now turn to~\eqref{e:w} and apply~\eqref{e:K},~\eqref{e:J} and~\eqref{e:I}, so that we deduce 
\begin{equ}
W_t \leq \frac{1}{2} W_t^* + C \Big(\|F-G\|_{\exp} + \sup_{(t,y) \in [0,T]\times\R^2} e^{- \langle y \rangle^\beta} \, |X_t(y) - Y_t(y)|\Big) \, ,
\end{equ}
where we further used that, since $A_0=1$ and $A,B$ are non-decreasing and non-negative, 
$\langle x \rangle^\beta\leq A_s \langle x \rangle^\beta + B_s$ for every $(s,x)\in[0,T]\times\R^2$. 
Note, in particular, that it is precisely this estimate which takes care of the exponential factor~$e^{\langle x\rangle^\beta}$ in~\eqref{e:K} as, by definition,~$W_t$ contains multiplication by~$e^{-A_t \langle x \rangle^\beta - B_t}$.
By assumption, both $u$ and $v$ (and thus $w$) satisfy~\eqref{e:GrowthCond}, and 
$A$ and $B$ are bounded, which means that, for every $t\in[0,T]$, $W^\ast_t<\infty$. 
As a consequence, we can take  bring the $W_t^*/2$ to the l.h.s. , and deduce 
that for every $t\in[0,T]$
\begin{equs}
|w_t(x)| &\leq 2 C \,e^{A_t \langle x \rangle^\beta + B_t} \Big(\|F-G\|_{\exp} + \sup_{(t,y) \in [0,T]\times\R^2} e^{- \langle y \rangle^\beta} \, |X_t(y) - Y_t(y)|\Big)\\
&\leq 2 C \,e^{A_T \langle x \rangle^\beta + B_T} \Big(\|F-G\|_{\exp} + \sup_{(t,y) \in [0,T]\times\R^2} e^{- \langle y \rangle^\beta} \, |X_t(y) - Y_t(y)|\Big)
\end{equs}
where in the last step we exploited once again the fact that $A$ and $B$ are non-decreasing. 
By rebranding the constant $C$ according to the prefactor at the r.h.s.,~\eqref{e:Apriori} follows at once. 
\end{proof}

In order to make use of the previous lemma, we need to ensure that, for the approximation 
outlined after the statement of Proposition~\ref{p:MollifSol}, the constant 
$C$ for which~\eqref{e:Apriori} holds (that depends continuously $\|\fE_\alpha(u^\tau)\|_{L^1(0,T)}$ and, 
for us, will be random) is  
tight and the other terms at the r.h.s. do indeed go to $0$. 
Our approximation consists of solutions to~\eqref{e:GSBE} with a polynomial nonlinearity 
and driven by periodic noise and initial 
condition. These were introduced in~\cite[eq. (A.7)]{CMT} and we now recall their 
definition. Let $(\eta^\tau, \vec{\xi}^\tau)$ be as above. 
For $M > 0$, let $(\eta^\tau_M, \vec{\xi}^\tau_M)$, whose law is denoted $\mathbb P^\tau_M \otimes P^\tau_M$, be given by
\begin{equ}[e:PerNoise]
    \eta^\tau_M \eqdef \sum_{k \in \Z^2} \restriction{\eta^\tau}{[-M, M]^2}(2Mk + \cdot) \, , \quad \vec{\xi}^\tau_M(\dd s) \eqdef \sum_{k \in \Z^2} \restriction{\vec{\xi}^\tau(\dd s)}{[-M, M]^2}(2Mk + \cdot) \, ,
\end{equ}
and if $M = \infty$ then we set $(\eta^\tau_\infty, \vec{\xi}^\tau_\infty) \eqdef (\eta^\tau, \vec{\xi}^\tau)$.

In the next lemma we show that, for any stationary process $t\mapsto u^\tau_t(\cdot)$, the quantity 
$\|\fE_\alpha(u^\tau)\|_{L^1(0,T)}$ is indeed tight. 

\begin{lemma} \label{l:Nwhitenoise}
Let $\tau>0$, $\alpha > 4 \kappa$ and, for $M>0$, define $\eta^\tau_M$ according to~\eqref{e:PerNoise}. 
Then, 
\begin{equ} \label{e:Nwhitenoise}
\sup_{M \in (0, +\infty]} \E[\fE_\alpha(\eta^\tau_M)] < +\infty \, .
\end{equ}
As a consequence, for every $T > 0$ and $\delta > 0$, 
there exists a constant $C = C(\tau, \alpha, T, \delta) > 0$ such that any 
$\P_M^\tau$-stationary locally continuous process $u^\tau$ (that is for every $t > 0$, $u^\tau_t \eqlaw \eta^\tau_M$) it holds
\begin{equ} \label{e:Nstationary}
\mathbf{P}\big(\|\fE_\alpha(u^\tau)\|_{L^1(0,T)} \geq C\big) \leq \delta \, .
\end{equ}
\end{lemma}
\begin{proof}
By a straightforward scaling argument using~\eqref{e:GoodScaling} 
and since we do keep track of the dependence in $\tau$, we take $\tau=1$. 
First, \eqref{e:Nstationary} follows by~\eqref{e:Nwhitenoise}, after a straightforward application of 
Markov's inequality and Fubini's theorem.
For \eqref{e:Nwhitenoise}, note that we always have $\fE_\alpha(\eta^1_M) \leq \fE_\alpha(\eta^1)$, 
as can be directly checked using that for any $x \in [-M, M]^2$ and $k \in 2 M \Z^2$, 
$\langle x+k \rangle \geq \langle x \rangle$. 
Thus it is sufficient to prove that $\E[\fE_\alpha(\eta^1)]$ is finite.

To prove that the latter, recall that $\eta^1$ is a Gaussian field such that, for all $y\in\R^2$, $\eta^1(y)$ 
is a centred Gaussian random variable of variance $1$. Define, for $k\in\Z^2$, the random variable 
$S_k = \sup_{x \in [-1/2,1/2]^2} |\eta^1(k+x)|$ and note that, by stationarity, 
the $S_k$'s are identically distributed and each of them is almost surely finite. By Borell--TIS inequality, 
we thus deduce that there exists $C > 0$ such that for every $A > 0$ it holds
\begin{equ}
\P(S_k > A) \leq C \exp(- A^2/2) \, . 
\end{equ}
By union bound and using that there exists $C' > 0$ such that $\langle x+k \rangle \geq (C')^{-1} \langle k \rangle$ for every $k \in \Z^2$ and $x \in [-1/2,1/2]^2$, we deduce that for every $\lambda > 0$ it holds
\begin{equ}
\P[\fE_\alpha(\eta^1) > \lambda] \leq \sum_{k \in \Z^2} \P\Big(e^{\kappa S_k^2}/\langle k \rangle^\alpha > C' \lambda\Big)\lesssim \sum_{k \in \Z^2} \frac{1}{\lambda^{\frac{1}{2 \kappa}} \langle k \rangle^{\frac{\alpha}{2 \kappa}}} \lesssim \frac{1}{\lambda^{\frac{1}{2 \kappa}}} \, ,
\end{equ}
where, in the last bound, we used that $\alpha > 4 \kappa$. 
Since $\kappa < 1/4$ (actually, here we only need $\kappa < 1/2$), 
we get $\E[\fE_\alpha(\eta^1)] < +\infty$, which concludes the proof.
\end{proof}

We are now ready to complete the proof of Proposition~\ref{p:MollifSol}. 

\begin{proof}[of Proposition~\ref{p:MollifSol}]
Uniqueness follows immediately from Lemma \ref{l:Apriori} applied with $F = G$ and 
$X^\tau_t = Y^\tau_t = K^\tau_t \eta^\tau + \int_0^t K^\tau_{t-s} \nabla \cdot \sqrt{R_\tau} \, \vec{\xi}^\tau(\dd s)$.

To prove existence and the stated properties, let $(P_M)_M$ be a sequence of polynomials 
such that $\|F-P_M\|_{\exp}\to 0$ as $M\to\infty$, whose existence is guaranteed by 
Lemma~\ref{l:DensityPolynomial}, and for any $M$, let $u^{\tau, M}$ 
be the unique (periodic) mild solution of~\eqref{e:GSBE} driven by $(\eta^\tau_M, \vec{\xi}^\tau_M)$, 
defined according to~\eqref{e:PerNoise},  
and with nonlinearity $P_M$. That such solution exists, is unique, Markov, $\P_M^1$-stationary, skew-symmetric  
and satisfies the weak formulation is shown in~\cite[Lemma~A.4]{CMT} for the 
case of $P_M$ quadratic, but the arguments therein work {\it mutatis mutandis} for a generic polynomial 
of any given degree. In particular,~\eqref{e:SolSPDE} holds and $\Prob$-almost surely, 
for every $t\geq 0$, we have 
\begin{equ}
u^{\tau,M}_t = X^{\tau,M}_t + \int_0^t K^\tau_{t-s} \cN^\tau_{P_M}[u^{\tau,M}_s] \, \dd s \, ,
\end{equ}
where $X^M$ is defined by 
\begin{equ}
 X^M_t \eqdef K^\tau_t \eta^\tau_M + \int_0^t K^\tau_{t-s} \nabla \cdot \sqrt{R_\tau} \, \vec{\xi}^\tau(\dd s) \, .
\end{equ}
We want to show that the sequence $(u^{\tau,M})_M$ is Cauchy. 
Let $4 \kappa < \alpha < \beta < 1$ and $T > 0$. Lemma~\ref{l:Apriori}, 
applied to $u = u^{\tau,M}$ and $v = u^{\tau,M'}$, gives that $\Prob$-almost surely for every $(t,x)\in[0,T]\times\R^2$, 
we have 
\begin{equs}
\thinspace & 
|u_t^{\tau,M}(x) - u_t^{\tau,M'}(x)| \\[0.5em]
& \leq \ 
C_M e^{C_M \langle x \rangle^\beta} \, \Big(\|P_M - P_{M'}\|_{\exp} + \sup_{(t,x) \in [0,T]\times\R^2} e^{- \langle x \rangle^\beta}\, |X^{\tau,M}_t(x) - X^{\tau,M'}_t(x)| \Big) \, ,
\end{equs}
where we recall that $C_M$ is a random positive constant that depends continuously on 
$\|P_M\|_{\exp}$, $\|\fE(u^{\tau,M})\|_{L^1(0,T)},\|\fE(u^{\tau,M'})\|_{L^1(0,T)}$ 
(the other quantities on which $C_M$ depends are independent of $M$). 
Since the sequence $(\|P_M\|_{\exp})_M$ converges and the family $(\|\fE_\alpha(u^M)\|_{L^1(0,T)})_{M \geq 1}$ is tight
in view of Lemma \ref{l:Nwhitenoise} (recall that for every $M$, $u^M$ is $\P_M^1$-stationary), 
we conclude that the sequence $(C_M)_M$ is tight. Further, the two summands 
in parenthesis converge to $0$ as $M,M'\to\infty$: the first deterministically by Lemma~\ref{l:DensityPolynomial}, 
the second in probability thanks to~\cite[Lemma A.2]{CMT} (this would be even true replacing the 
$\exp(- \langle x\rangle^\beta)$ by any $\langle x\rangle^{- \eps}$, $\eps > 0$). 

Therefore, as $\beta<1$, we deduce that the sequence of process $(u^{\tau,M})$ is 
Cauchy in the complete metric space consisting of processes living in 
$\cC([0,T], \cX_1)$, equipped with the distance $d(u,v) = \E[1 \wedge \sup_{t \in [0,T]} d_1(u_t,v_t)]$. 
Thus, up to extracting a subsequence, we conclude that there exists a process $u^\tau\in\cC([0,T], \cX_1)$ 
which is the $\Prob$-almost surely limit (along the subsequence) of $(u^{\tau,M})_{M \geq 1}$ in $\cC([0,T], \cX_1)$, 
and, via a diagonal argument, one can prove that $u^\tau \in \cC(\R_+, \cX_1)$ and that 
the convergence takes place in the same space. 

It is then immediate to pass through the limit all the properties of $u^{\tau,M}$. In particular, 
the limit is $(\cF_t)_{t \geq 0}$-adapted; $\P^1$-stationary, and thus $t\mapsto \fE(u^\tau_t)\in L^1_{\loc}(\R_+)$; 
it satisfies~\eqref{e:SolSPDE}, and therefore it is the necessarily unique mild solution to~\eqref{e:GSBE} 
driven by $(\eta^\tau, \vec{\xi}^\tau)$ according to Definition~\ref{def:SolSPDE}; 
the weak formulation~\eqref{e:WeakFormulation} holds; it is skew-symmetric. 
As for the Markov property, it is an immediate consequence of the existence, 
adaptedness and uniqueness of solutions of \eqref{e:GSBE}. 
\end{proof}

\subsection{The semigroup and the generator of the generalised SBE}\label{sec:generator}

In view of Proposition~\ref{p:MollifSol}, for every $\tau > 0$,~\eqref{e:GSBE} admits a
unique solution $u^\tau$ which is a stationary Markov process.
Thus, its semigroup, $P^\tau$, given by 
\begin{equ}[e:semigroup]
P_t^\tau \phi \eqdef \Exp^\tau\big[\phi(u^\tau_t) \, \big| \, \cF_0\big] \, ,
\end{equ}
is well-defined. Let us state a preliminary lemma providing basic properties of $P^\tau$ and 
the generator $\cL^\tau$ associated to it. 

\begin{lemma}\label{lem:Contractivity}
Let $\theta\in[1,\infty)$. Then, $(P^\tau_t)_{t\geq0}$ in~\eqref{e:semigroup}
defines a contractive strongly continuous semigroup on $\L^\theta(\P^\tau)$
such that, for every $\phi\in \L^\theta(\P^\tau)$, 
\begin{equ}[e:Contractivity]
\|P_t^\tau \phi\|_{\mathbb L^\theta(\mathbb P^\tau)} \leq \|\phi\|_{\mathbb L^\theta(\mathbb P^\tau)}
\end{equ}
holds uniformly in $\tau$ and also for $\theta=\infty$.
In particular,
the generator $\cL^\tau \colon \mathrm{Dom}(\gen) \subset \L^\theta(\P^\tau) \to \L^\theta(\P^\tau)$
associated to $P^\tau$ is closed and, for every $\mu>0$, 
its resolvent $(\mu-\gen)^{-1}\colon \L^2(\P^\tau)\to\L^2(\P^\tau)$ is bounded uniformly in $\tau$ by $\mu^{-1}$ 
and has range given by $\mathrm{Dom}(\gen)$.
\end{lemma}
\begin{proof}
The proof follows exactly the same steps as that of~\cite[Lemma 2.6]{CMT} 
and is thus omitted. The only point in which some care is needed is in the proof of 
the strong continuity of the semigroup: in view of the continuity implied by~\eqref{e:Contractivity}, 
it suffices to check it on a dense subset of $\L^\theta(\P^\tau)$, $\theta\in[1,\infty)$, 
which should here be taken to be 
the set of bounded cylinder functionals depending on smooth compactly supported test functions. 
In the above-mentioned reference, bounded cylinder functions 
depending on more general Schwartz test functions were used instead, but in the present context, a priori, $u^\tau_\cdot\in\cX_1$ 
cannot be tested against these (see Remark~\ref{rem:Growth}). 
\end{proof}

Our next goal is to identify a suitable subset of $\L^2(\P)$ (that, by abusing notation, 
is identified from here on with the Fock space $\fock{}{}{\tau}$) which is contained in the domain of the generator,  
on which the action of $\gen$ is explicit and which contains all those functionals we need 
in order to establish the convergence stated in our main theorem. 
In view of the estimates in Section~\ref{s:gen_est_nonq}, it turns out 
that a convenient choice is that of the subspace $\core$ consisting of 
$f = (f_n)_{n \geq 0} \in \fock{}{}{\tau}$ such that for every $n$ the Fourier transform of $f_n$ 
is compactly supported and $f_n\neq 0$ only for finitely many $n$'s. 
An example of $\phi\in\core$ is $\phi=\otimes^m h\in\fock{m}{}{\tau}$ (corresponding to 
$m! H_m(\eta^\tau(h))\in\L^2(\P^\tau)$ for $H_m$ the $m$-th Hermite polynomial in~\eqref{e:HermitePol}) 
for $m\in\N$ and 
$h$ such that $\mathrm{Supp}(\cF(h))$ is compact. 
Notice that, even for such a simple example, since $h$ cannot have compact support in real space 
it is not clear that $t\mapsto \phi(u^\tau_t)$ is continuous (see Remark~\ref{rem:Growth}) 
and thus It\^o's formula cannot be applied to easily conclude that $\phi\in\mathrm{Dom}(\gen)$. 

In the next lemma, we prove a slightly more general statement, namely, 
that cylinder random variables are contained in $\mathrm{Dom}(\gen)$ 
and thus infer that is $\core$ as well. 

\begin{lemma}\label{l:CylDomain}
Let $\tau>0$ be fixed and $\cL^\tau$ the generator associated to the semigroup $P^\tau$ in~\eqref{e:semigroup}. 
Then, any cylinder function $\phi$, i.e. $\phi=f(\eta^\tau(h_1),\dots,\eta^\tau(h_n))$ 
for $f\colon \R^n\to\R$ smooth and of at most polynomial growth at infinity, and $h_1,\dots,h_n\in\cS(\R^2)$, 
belongs to $\mathrm{Dom}(\gen)$ and we have $\gen\phi(\eta^\tau)=S^\tau\phi(\eta^\tau)+A^\tau\phi(\eta^\tau)$, 
where 
\begin{equs}[e:GenCylFunctional]
S^\tau\phi(\eta^\tau)\eqdef&\sum_{i=1}^n \big\langle \frac{1}{2} \nabla \cdot R_\tau \nabla \eta^\tau, h_i \big\rangle_{L^2(\R^2)} \, \partial_i f(\eta^\tau(h_1),\dots,\eta^\tau(h_n))  \\
 &\qquad+ \frac{1}{2} \sum_{1 \leq i,j \leq n} \langle R_{\tau} \nabla h_i, \nabla h_j\rangle_{L^2(\R^2)} \, \partial_i \partial_j f(\eta^\tau(h_1),\dots,\eta^\tau(h_n))\,,\\
 A^\tau\phi(\eta^\tau)\eqdef&\sum_{i=1}^n \big\langle \cN^\tau_F[\eta^\tau], h_i \big\rangle_{L^2(\R^2)} \, \partial_i f(\eta^\tau(h_1),\dots,\eta^\tau(h_n))\,.
\end{equs} 
As a consequence, $\core\subset \mathrm{Dom}(\gen)$. 
\end{lemma}
\begin{proof}
We will first derive the action of $\gen$ on {\it good} cylinder functionals, 
i.e. cylinder functionals $\phi=f(\eta^\tau(h_1),\dots,\eta^\tau(h_n))$ such that both $f$ and 
$h_1,\dots,h_n$ have compact support, and then extend it cylinder functionals 
using that $\gen$ is a closed operator as shown in Lemma~\ref{lem:Contractivity}. 

Let $\phi=f(\eta^\tau(h_1),\dots,\eta^\tau(h_n))$ be a good cylinder functional.   
Since $v^\tau$ takes values in $\cC(\R_+, \cX_1)$ and the test functions $h_1,\dots,h_n$ 
belong to $\cC_c^\infty(\R^2)$, the weak formulation \eqref{e:WeakFormulation} 
and Itô's formula imply that, $\mathbf{P}$-almost surely, for every $t \geq 0$
\begin{equ}[e:ItoFormula]
\Phi(u^\tau_t) = \Phi(u_0^\tau) + \int_0^t L^\tau \Phi(u_s^\tau) \, \dd s + M^\tau_t(\Phi) \, , 
\end{equ}
where $L^\tau\phi=S^\tau\phi+A^\tau\phi$ with $S^\tau\phi,\,A^\tau\phi$ 
defined according to~\eqref{e:GenCylFunctional}, and 
$(M^\tau_t(\Phi))_{t \geq 0}$ is a true continuous martingale. 
Notice that the functional $L^\tau \Phi(\cdot)$ is continuous on $\cX_1$ 
(so that $L^\tau \Phi(u^\tau)$ is $\mathbf{P}$-almost surely continuous) and 
$L^\tau \Phi \in \mathbb L^{2+\eps}(\mathbb P^\tau)$ (so that 
$\sup_{s \geq 0} \E[|L^\tau \Phi(v_s^\tau)|^{2+\eps}] < +\infty$ by stationarity). 
To check the last claim, recall the definition of $\|\cdot\|_{\exp}$ in~\eqref{e:AssF} 
and fix $\eps\in(0,\frac1{2\kappa}-2)$. Since $\|F\|_{L^{2+\eps}(\cN(0,1))} \lesssim \|F\|_{\exp}$, 
it follows that, for any $h \in \cC^\infty_c(\R^2)$ and $p\geq 1$, it holds
\begin{equ}[e:NonLinL2+eps]
\E[| \cN^\tau_F[\eta^\tau](h)|^{2+\eps}]^{\frac{1}{2+\eps}} \lesssim_{\tau} \|F\|_{\exp} \, \|h\|_{L^1(\R^2)} \, ,\quad\text{and}\quad \E[|\eta^\tau( h )|^p]^{\frac{1}{p}} \lesssim_{\tau} \|h\|_{L^1(\R^2)}\,.
\end{equ}
Therefore, we can make  sense of both $\cN^\tau_F[\eta^\tau](h)$ and $\eta^\tau( h )$ as elements of $\L^2(\P)$ 
for any $h \in L^1(\R^2)$ (and they satisfy~\eqref{e:NonLinL2+eps}). 

The continuity and integrability of $L^\tau \Phi$, together with the continuity of $s\mapsto v^\tau_s$ ensure 
that both $\Prob$-almost surely and in $\L^2(\P^\tau)$ we have 
\begin{equ}
\frac{1}{t} \int_0^t L^\tau \Phi(v_s^\tau) \, \dd s \xrightarrow[t \to 0]{} L^\tau \phi(v_0^\tau) \, ,
\end{equ}
which, by taking conditional expectation with respect to $\cF_0$ and using~\eqref{e:ItoFormula}, 
leads to  
\begin{equ}
\frac{P_t^\tau \Phi - \Phi}{t} \xrightarrow[t \to 0]{\mathbb L^2(\mathbb P)} L^\tau \phi(v_0^\tau) \, . 
\end{equ}
In other words, $\Phi$ belongs to the domain of $\gen$ and $\gen \Phi = L^\tau \Phi$.
\medskip

Now, let $\phi,\phi'$ be two good cylinder functionals with $\phi=f(\eta^\tau(h_1),\dots,\eta^\tau(h_n))$ and 
$\phi'=f'(\eta^\tau(h'_1),\dots,\eta^\tau(h'_n))$. Thanks to the fact that $\gen \Phi = L^\tau \Phi$,
~\eqref{e:NonLinL2+eps} and H\"older inequality, 
it is straightforward to check that for every $N \geq 0$ it holds that
\begin{equs}
\E[|&\Phi(\eta^\tau) - \Phi'(\eta^\tau)|^2] \vee \E[|\gen\Phi(\eta^\tau) - \gen \Phi'(\eta^\tau)|^2] \\
&\lesssim_{F,  \tau, n, N} \max_{1 \leq i \leq n} \big\{1, \|h_i\|_{L^1(\R^2)}, \|\nabla h_i'\|_{L^2(\R^2)}\big\}^{2 N+4} \times \bigg(\sup_{y_{1:n} \in \R^{n}} \frac{|(f' - f)(y_{1:n})|^2}{1+|y_{1:n}|^{2N}}  \\
&\qquad+\sup_{y_{1:n} \in \R^{n}, \, \alpha \in \{0,3\}^2} \frac{|\partial_\alpha^{|\alpha|} f(y_{1:n})|^2}{1+|y_{1:n}|^{2N}} \max_{1 \leq i \leq N }\{\|h'_i - h_i\|_{L^1(\R^2)}, \|\nabla (h_i' - h_i)\|_{L^2(\R^2)}\}\bigg) \, ,
\end{equs}
which implies that $\gen$ is continuous on the set of good cylinder functionals endowed with its natural 
Fr\'echet metric on the set of cylinder functionals. Since the set of good cylinder functionals 
is clearly dense both in the set of cylinder functionals and in $\core$, 
and further $\gen$ is closed on $\L^2(\P^\tau)$, it follows that 
both cylinder functionals and $\core$ belong to the domain of $\cL^\tau$ and~\eqref{e:GenCylFunctional} holds, 
so that the proof is complete. 
\end{proof}

We are now ready to derive an explicit expression for the action of $\gen$ on $\core\subset\fock{}{}{\tau}$. 
For $n \geq 0$, we write $\core_n = \core \cap \fock{n}{}{}$.

\begin{lemma}\label{lem:ActionGen}
Let $\tau>0$ be fixed. Let $\gensyx$, $\gensyy$ and $\gensy$ be the diagonal operators 
acting on $\psi \in \core_n$ as
\begin{equs}
\cF(\cL_0^{\fe_i}\psi)\,(p_{1:n}) &\eqdef -\frac{1}{2}|\fe_i\cdot p_{1:n}|^2 \, \hat\psi(p_{1:n})\,,\qquad i=1,2, \label{e:L0w} \\
\cF(\gensy \psi)\,(p_{1:n}) &\eqdef \cF\big((\nu_\tau\cL_0^{\fe_1}+\cL_0^{\fe_2})\psi\big)\,(p_{1:n})=-\frac{1}{2} \,  |\sqrt{R_\tau} p_{1:n}|^2 \, \hat\psi(p_{1:n}) \, . \label{e:gensy}
\end{equs}
and, for $m \geq 1$, $j \in \{0,...,m-1\}$ and $a=m-2j-1$,  
define the operator $\genaF{m}{a}$ which, for any $n \in \N$, maps $\core_n$ to $\core_{n+a}$ and acts on $\psi \in \core_n$ as follows: if $j+1> n$ then $\genaF{m}{a}\psi=\genaF{m}{m-2j-1}\psi\equiv 0$, otherwise for $m=1$, 
\begin{equ}[e:Gena1]
\cF(\genaF{1}{0} \psi) \, (p_{1:n}) \eqdef - \iota \, \tau^{1/2} \nu_\tau^{1/2} \Big(\sum_{i=1}^n (\fe_1 \cdot p_i) \, (\Theta^\tau(p_i) - 1)\Big) \hat{\psi}(p_{1:n}) \, ,
\end{equ}
and for $m \geq 2$, 
\begin{equs}[e:Genamj]
\cF(\genaF{m}{a}\psi) \, (p_{1:n+a}) \eqdef& - \iota \, \frac{\tau^{1-\frac{m}2} \nu_\tau^{\frac{1+m}{4}}}{(2\pi)^{m-1}}\, \frac{n!}{(n+a)! (j+1)!} \times \\
&\times\sum_{I\in\Delta_{m-j}^{n+a}} (\fe_1 \cdot p_{[I]}) \int_{r_{[1:j+1]} = p_{[I]}} \Xi_{j+1}^\tau(\dd r_{1:j+1}) \, \hat{\psi}(r_{1:j+1}, p_{1:n+a\setminus I}) \, ,
\end{equs}
where for any $u \leq v$ positive integers, $\Delta_{u}^{v} \eqdef \{(i_1,\dots,i_u)\in \N^{u}\colon 1 \leq i_1 < ... < i_{u} \leq v\}$. 
We also set
\begin{equ}[e:AmA]
	\genaF{m}{} \eqdef \sum_{j=0}^{m-1} \genaF{m}{m-2j-1}\,, \qquad \text{and}\qquad
	\gena \eqdef \sum_{m \geq 1} c_m(F) \, \genaF{m}{}\,, 
	%
%
\end{equ}
the last sum converging in $\fock{}{}{\tau}$ uniformly on compact subsets of $\core$. 
Then, on $\core\subset {\rm Dom}(\gen)$, $\gen$ coincides with $\gensy + \gena$, 
$\gensy$ is symmetric and each $\genaF{m}{}$ (and thus also $\gena$) is skew-symmetric 
as it holds $(\genaF{m}{a})^* = - \genaF{m}{-a}$ for every $m \geq 1$, $a \in \{-m+1, -m+3, \dots, m-1\}$.
\end{lemma}

Compared to~\cite{CMT}, in the present setting the action of the generator  
is much more involved. Along the lines of the analogy made therein, 
we can think of $\fock{n}{}{\tau}$ as the span of $n$-particle states, with each 
particle carrying a momentum $p_i$, and $\gena$ as an operator acting on the 
particles by changing their number and modifying their momenta. In this perspective, while 
the $\gena$ in~\cite{CMT} coincides with 
$c_2(F)\genaF{2}{}=c_2(F)\sum_{a=\pm 1}\genaF{2}{a}$ and thus 
creates (for $a=1$) and annihilates (for $a=-1$) {\it one single} particle, 
here, $\gena$ creates or annihilates an {\it arbitrarily large number} of particles 
as the $m$ in the sum in~\eqref{e:AmA} is unbounded. 
More precisely, for $m\in\N$, $a\in\{-m+1,-m+3,\dots,m-1\}$ and $j$ such that $a=m-2j-1$, 
one can regard $\genaF{m}{a}$, 
as an operator that takes $j+1$ particles 
with momenta $r_1,\dots,r_{j+1}$ whose sum is $p_{[I]}$, for $I\subset\{1,\dots,n\}$ 
of cardinality $m-j$, and produces $m-j$ particles, 
whose respective momenta are $\{p_{i}\colon i\in I\}$. 

Since $m$ is unbounded it is not a priori obvious that $\gena$ in~\eqref{e:AmA} 
is well-defined {\it even for finite $\tau$} and, to ensure that this is the case, 
a crucial role will be played by the decay of the coefficients of the Hermite 
polynomial expansion of the nonlinearity $F$, which in turn is tightly connected 
to its regularity (see Appendix~\ref{a:DecayCoeff}). While this is the first problem 
one encounters in the proof of the above Lemma, the bulk of it 
consists of showing that the action of $\gen$ on $\core$ 
in~\eqref{e:GenCylFunctional} indeed coincides with 
the expressions in~\eqref{e:gensy} and~\eqref{e:AmA} for which 
instead we will approximate $F$ by a polynomials using Lemma~\ref{l:DensityPolynomial}, 
prove that the statement holds for such $F$ (Lemma~\ref{l:PolGen}) and 
finally verify that we are allowed to remove the approximation in $\L^2(\P^\tau)$. 

\begin{proof}
Before deriving the action of $\gen$ on the Fock space $\fock{}{}{\tau}$ and 
show that~\eqref{e:GenCylFunctional} indeed reduces to the formulas in the statement, 
let us mention that, as a consequence of Lemma~\ref{l:ContOp}, 
the operators $\gensyx$, $\gensyy$, $\gensy$ and $\genaF{m}{a}$, for 
$m \geq 1$ and $a \in \{-m+1,-m+3...,m-1\}$, respectively 
given in~\eqref{e:L0w},~\eqref{e:gensy},~\eqref{e:Gena1} and~\eqref{e:Genamj} 
are continuous on $\core$. 
For $\gena$ in~\eqref{e:AmA}, the argument is more subtle and 
relies on the quantitative estimate~\eqref{e:SummabilityA} and Assumption~\ref{Assumption:F}. 
Indeed, for $n\in\N$ and $\psi\in\core_n$, the former implies
\begin{equs}
\|\gena\psi\|_\tau\leq \sum_{m\geq 1}\frac{c_m(F)}{\sqrt{m!}}\|\sqrt{m!}\genaF{m}{}\psi\|_\tau\lesssim_n \|\psi\|_\tau \sum_{m\geq 1}\frac{c_m(F)}{\sqrt{m!}} m^{n/2}
\end{equs}
and, since by Proposition~\ref{p:cm(F)asymptotic} and Assumption~\ref{Assumption:F} $c_m(F)/\sqrt{m!}$ 
decays faster than any inverse power, the sum at the r.h.s. is bounded (by an $n$-dependent constant) 
and thus, $\gena$ is well-defined and continuous from $\core$ to $\fock{}{}{\tau}$. 

It remains to show that for $\phi\in\core$,  $\gen\phi=\gensy\phi + \gena\phi$ and that 
$\gensy$ and $\gena$ satisfy the stated symmetry properties. In view of their 
continuity, it suffices to prove them for $\phi$ in a dense subset 
of $\core$ which we take to be the set of $\phi$ of the form $I_n^\tau(h^{\otimes n}) = n! H_n( \eta^\tau(h))$ 
for $n\in\N$ and $h$ such that $\hat h\in \cC_c^\infty(\R^2)$ and $\|h\|_{L^2_\tau(\R^2)} = 1$.
Notice that such $\phi$ is a cylinder functional so that, by Lemma~\ref{l:CylDomain}, $\phi\in{\rm Dom}(\gen)$ 
and $\gen\phi=S^\tau\phi+A^\tau\phi$, with $S^\tau\phi,\,A^\tau\phi$ as in~\eqref{e:GenCylFunctional}. 
As shown in, e.g.,~\cite{GPGen}, the Fock space representation of $S^\tau\phi$ coincides with $\gensy\phi$ 
in~\eqref{e:gensy}, which is clearly symmetric, so that we are left to verify that $A^\tau \Phi=\gena\phi$ 
and that the latter is skew-symmetric. Now, when $F$ is a polynomial, both properties hold thanks to Lemma~\ref{l:PolGen} 
and the linearity of $\gena$, and we now show that this implies 
the same is true for any $F$ satisfying Assumption~\ref{Assumption:F}, via an approximation argument. 

Denote by $A_P^\tau\phi$ the quantity $A^\tau\phi$ in~\eqref{e:GenCylFunctional} with $F=P$ and $P$ polynomial 
and assume that, for any $P$, $A_P^\tau\phi=\cA_P^\tau\phi\eqdef \sum_{m \geq 1} c_m(P) \, \cA^{\tau, m} \Phi$, 
with the sum at the r.h.s. having only finitely many terms. 
Let $(P_M)_{M \geq 1}$ be a sequence of polynomials converging to $F$ in the sense of 
Lemma~\ref{l:DensityPolynomial}. The claim follows provided we prove that, as $M\to\infty$,   
the sequences $(A_{P_M}^\tau\phi)_M$ and $(\cA_{P_M}^\tau\phi)_M$ converge respectively 
to $A^\tau\phi$ and $\cA^\tau\phi$, in $\L^2(\P^\tau)$. 
For the former, notice that for any $\eps\in(0,\frac{1}{2\kappa},2)$
\begin{equ}
\E[| \cN^\tau_F[\eta^\tau](h)-\cN^\tau_{P_M}[\eta^\tau](h)|^{2+\eps}]^{\frac{1}{2+\eps}} =\E[| \cN^\tau_{F-P_M}[\eta^\tau](h)|^{2+\eps}]^{\frac{1}{2+\eps}} \lesssim_{\tau} \|F-P_M\|_{\exp} \, \|h\|_{L^1(\R^2)}
\end{equ}
where in the last step we applied~\eqref{e:NonLinL2+eps} with $F-P_M$ in place of $F$ 
and the norm $\|\cdot\|_{\exp}$ was defined in~\eqref{e:AssF}. 
Hence, by Lemma~\ref{l:DensityPolynomial} the r.h.s. converges to $0$ and, 
using the representation~\eqref{e:GenCylFunctional} and H\"older's inequality, we 
conclude that $A_{P_M}^\tau\phi\to A^\tau\phi$ as $M\to\infty$ in $\L^2(\P^\tau)$. 
For the other, recalling the definition of the coefficients $c_m$ in~\eqref{e:DecompositionF} and 
using the estimate~\eqref{e:SummabilityA}, we have 
\begin{equs}
\|&\cA_{P_M}^\tau\phi-\cA^\tau\phi\|_\tau\leq \sum_{m\geq 1}|c_m(P_M)-c_m(F)|\|\genaF{m}{}\phi\|_\tau\\
&\lesssim \|\phi\|_\tau\Big(\sum_{m < n+3}\frac{m^{n/2}}{\sqrt{m}}+\sum_{m\geq n+3}\frac{m^{n/2}}{\sqrt{m\dots(m-n+1)}}\Big)\sup_{\substack{x\in\R\\ m\leq n+3}} |P_M^{(m)}(x)-F^{(m)}(x)|e^{\kappa x^2}
\end{equs}
where we applied Lemma~\ref{l:HerAppDecay} with $\ell=m-1$ for $m\leq n+3$ and $\ell=n+3$ elsewhere. 
The quantity in parenthesis is bounded (and its value depends on $n$) and the supremum converges to $0$ 
by Lemma~\ref{l:DensityPolynomial}.
\end{proof}

As we will see in the next section, the main ingredient in the proof of our main 
result is the convergence of the resolvent of $u^\tau$ to that of $u^{\eff}$ in~\eqref{e:limitingSHE} 
in $\fock{}{}{}$. To do so, one is led to estimating first the $\fock{}{1}{}$-norm 
of the (difference of the) resolvent(s), as this is the norm that naturally appears 
when testing the resolvent equation with its solution (recall that  
$\gena$ is skew-symmetric), in terms of the $\fock{}{-1}{}$-norm (by duality 
of $\fock{}{1}{}$ and $\fock{}{-1}{}$) of the function 
to which the resolvent is applied. 
However, this is only possible 
provided ${\rm Dom}(\gen)\subset \fock{}{1}{}$, which ensures that 
the $\fock{}{1}{}$-norm of the resolvent is indeed 
bounded. While in~\cite{CMT} Lemma 2.10 therein verifies such condition, 
in the present context, the singularity 
of $\gena$ due to the sum defining it in~\eqref{e:AmA} 
having infinitely many terms, makes it unclear whether this is the case (and even 
if the sum were up to a finite $m$ with $m>2$, we would not know how to prove it)\footnote{The problem 
essentially reduces to the fact that 
we do not know whether $\core$ is a {\it core} for $\gen$ for $\tau$ finite. 
Indeed, even for fixed $\tau$, the graded sector condition of~\cite[Section 3.7.4]{KLO} does not hold 
due to the growth in the number operator (see e.g.~\eqref{e:SummabilityA}). }. 
That said, as the next lemma shows, the so-called It\^o's trick in~\cite{GJ13} 
allows us to bypass the problem as it directly proves that the resolvent is continuous from 
$\fock{}{-1}{\tau}$ to $\fock{}{}{\tau}$ uniformly in $\tau$.   

\begin{lemma} \label{l:ItoTrick}
There exists a constant $C>0$ independent of $\tau$ such that for every $\Phi \in \fock{}{}{\tau}$ it holds that
\begin{equ}[e:H-1L2Res]
\|(1-\gen)^{-1} \Phi\|_\tau \leq C \, \|\Phi\|_{\fock{}{-1}{\tau}} \, .
\end{equ}
\end{lemma}
\begin{proof}
Since $(1-\gen)^{-1}$ is bounded by $1$ from $\fock{}{}{\tau}$ to itself 
by Lemma~\ref{lem:Contractivity}, 
a density argument ensures that it suffices to prove~\eqref{e:H-1L2Res} on 
a dense subspace of $\fock{}{}{\tau}$. For reasons that will become clear in what follows, 
we will take it to be the set of functionals 
$\phi$ such that $(1-\gensy)\phi=\psi$ with $\psi$ a {\it good} cylinder functionals as 
those used in the proof of Lemma~\ref{l:CylDomain}, i.e. 
$\psi=f(\eta^\tau(h_1),\dots,\eta^\tau(h_n))$ such that $f$ and 
$h_1,\dots,h_n$ are smooth with compact support.

Notice first that, by integration by parts and the contractivity bound~\eqref{e:Contractivity}, the resolvent satisfies 
\begin{equ}
(1 - \gen)^{-1} \Phi = \int_0^{+\infty} e^{-t} P^\tau_t \Phi \, \dd t = \int_0^{+\infty} e^{-t} \Big(\int_0^t P^\tau_s \Phi \, \dd s\Big) \, \dd t\,,
\end{equ}
where $P^\tau$ is the semigroup in~\eqref{e:semigroup}. As a consequence, denoting by $u^\tau$ 
the unique mild solution to~\eqref{e:GSBE} driven 
by $(\eta^\tau, \vec{\xi}^\tau)$, we get 
\begin{equ}[e:PreIto]
\|(1 - \gen)^{-1} \Phi \|_\tau\leq \Big\|\int_0^{+\infty} e^{-t} \E\Big[\int_0^t \Phi(u^\tau_s) \, \dd s\Big|\cF_0\Big] \, \dd t\Big\|_\tau\leq \int_0^\infty e^{-t}\Exp\bigg[\Big|\int_0^t \Phi(u^\tau_s) \, \dd s\Big|^2\bigg]^{\frac{1}{2}}\dd t
\end{equ}
an our goal is to bound the inner expectation in terms of the $\fock{}{-1}{\tau}$-norm of $\phi$. 
To do so, we will exploit the same strategy first presented in the proof of the 
celebrated It\^o's trick in~\cite{GJ13}. 

Proposition~\ref{p:MollifSol} and It\^o's formula give  
\begin{equ}
\psi(u_t^\tau) - \psi(u_0^\tau) - \int_0^t \cS^\tau \psi(u_s^\tau) \, \dd s - \int_0^t \cA^\tau \psi(u_s^\tau) \, \dd s = M_t^\tau(\psi) \, ,
\end{equ}
where we further used that $\psi\in\core$ so that, by Lemma~\ref{lem:ActionGen}, $\gen\psi=\gensy\psi+\gena\psi$, 
and $M^\tau(\psi)$ is a true continuous martingale whose quadratic variation can be readily checked to 
satisfy 
\begin{equ}
\Exp[\langle M^\tau(\psi), M^\tau(\psi) \rangle_t] 
= 2 t \, \|(- \cS^\tau)^{1/2} \psi\|_\tau^2\,.
\end{equ}
The very definition of quadratic variation then implies 
\begin{equ} \label{e:ItoTrick1}
\Exp\Big[\Big|\psi(u_t^\tau) - \psi(u_0^\tau) - \int_0^t \cS^\tau \psi(u_s^\tau) \, \dd s - \int_0^t \cA^\tau \psi(u_s^\tau) \, \dd s\Big|^2\Big] \leq 2 t \, \|(- \cS^\tau)^{1/2} \psi\|_\tau^2 \, .
\end{equ}
The same argument applied to the mild solution $\tilde u^\tau$ of~\eqref{e:GSBE} 
with $-F$ in place of $F$ together with the fact that, thanks to Proposition~\ref{p:MollifSol}, 
$(\tilde u^\tau_s)_{s\in[0,t]}\eqlaw (u^\tau_{t-s})_{s\in[0,t]}$, also gives 
\begin{equ} \label{e:ItoTrick2}
\Exp\Big[\Big|\psi(u_0^\tau) - \psi(u_t^\tau) - \int_0^t \cS^\tau \psi(u_s^\tau) \, \dd s + \int_0^t \cA^\tau \psi(u_s^\tau) \, \dd s\Big|^2\Big] \leq 2 t \, \|(- \cS^\tau)^{1/2} \psi\|_\tau^2 \, .
\end{equ}
Summing \eqref{e:ItoTrick1} and \eqref{e:ItoTrick2} (and using the basic inequality 
$2a^2+2b^2\geq (a+b)^2$, $a,b\in\R$), 
we end up with
\begin{equ}
\Exp\Big[\Big|\int_0^t (-\cS^\tau) \psi(u_s^\tau) \, \dd s\Big|^2\Big] \leq 2t \, \|(- \cS^\tau)^{1/2} \psi\|_\tau^2\,.
\end{equ}
We bound the l.h.s.  from below as 
\begin{equs}
\Exp\Big[\Big|\int_0^t (-\cS^\tau) \psi(u_s^\tau) \, \dd s\Big|^2\Big]&\geq \frac12\Exp\Big[\Big|\int_0^t (1-\cS^\tau) \psi(u_s^\tau) \, \dd s\Big|^2\Big]-\Exp\Big[\Big|\int_0^t \psi(u_s^\tau) \, \dd s\Big|^2\Big]\\
&\geq \frac12\Exp\Big[\Big|\int_0^t (1-\cS^\tau) \psi(u_s^\tau) \, \dd s\Big|^2\Big]-t^2\|\psi\|_\tau 
\end{equs}
the last step being a consequence of Jensen's inequality. Now, rearranging the terms 
and recalling that $(1-\gensy)\phi=\psi$, deduce 
\begin{equs}
\Exp\Big[\Big|\int_0^t \phi(u_s^\tau) \, \dd s\Big|^2\Big]\leq 4 t \|(- \cS^\tau)^{1/2} \psi\|_\tau^2 +2 t^2\|\psi\|_\tau^2\leq 4 (t+t^2)\|\psi\|^2_{\fock{}{1}{\tau}}=4 (t+t^2)\|\phi\|^2_{\fock{}{-1}{\tau}}
\end{equs}
which can be plugged at the r.h.s. of~\eqref{e:PreIto} to obtain~\eqref{e:H-1L2Res}. 
\end{proof}

\begin{remark}\label{rem:NoNeedForH1}
Let us point out that the proof of the previous lemma uses only the stationarity of 
$u^\tau$ and the boundedness of its resolvent so that its statement holds much more generally. 
\end{remark}

\section{Proof of  the main result} \label{s:proof_main}

This section represents the core of the paper as it contains the novel ideas and 
the crucial estimates thanks to which we can prove our main results, Theorems~\ref{th:CLT} and~\ref{th:conv_semigroup}. 
Before delving into the details, let us heuristically present the necessary steps 
and provide a roadmap of what will follow. 

\subsection{Main ideas and roadmap}\label{sec:Ideas}

The bulk of the proof of Theorems~\ref{th:CLT} and~\ref{th:conv_semigroup} consists of showing that 
$\gen$ converges to $\geneff$, the generator of $u^\eff$ in~\eqref{e:limitingSHE}, 
in the resolvent sense, i.e. that for a sufficiently large class 
of $\phi\in\fock{}{}{}$ we formally have 
\begin{equ}[e:ResConv]
\lim_{\tau\to\infty}\|(1-\gen)^{-1}\phi-(1-\geneff)^{-1}\phi\|=0\,.
\end{equ}
We will not prove~\eqref{e:ResConv} directly as we have very little control of  
the resolvent of $\gen$. Instead, since according to the heuristic presented in the 
introduction we expect the quadratic component of the nonlinearity $F$ (corresponding to 
its projection onto the space generated by the second Hermite polynomial) to be dominant, 
we will proceed in two steps and show instead 
that both the following limits hold
\begin{equs}
\lim_{\tau\to\infty}\|(1-\gen)^{-1}\phi-(1-\cL^{\tau,2})^{-1}\phi\|&=0\,,\label{e:gengen2}\\
\lim_{\tau\to\infty}\|(1-\cL^{\tau,2})^{-1}\phi-(1-\geneff)^{-1}\phi\|&=0\label{e:gen2geneff}
\end{equs}
where $\cL^{\tau,2}\eqdef\gensy+c_2(F)\genaF{2}{}$ is the generator of 
the solution of~\eqref{e:GSBE} with nonlinearity $\cN_G^\tau$ with $G(x)=c_2(F)x^2$. 
The validity of~\eqref{e:gen2geneff} is a direct consequence of what was 
shown in~\cite{CMT}. To see why we can expect~\eqref{e:gengen2} to be satisfied, 
notice that by~\eqref{e:H-1L2Res}, we (formally) have 
\begin{equs}[e:HeuH-1]
\|(1-\gen)^{-1}\phi-(1-\cL^{\tau,2})^{-1}\phi\|_\tau&=\|(1-\gen)^{-1}\big(\phi-(1-\gen)[(1-\cL^{\tau,2})^{-1}\phi]\big)\|_\tau\\
&\lesssim \|\phi-(1-\gen)[(1-\cL^{\tau,2})^{-1}\phi]\|_{\fock{}{-1}{\tau}}\\
&=\|\genaF{\neq2}{}[(1-\cL^{\tau,2})^{-1}\phi]\|_{\fock{}{-1}{\tau}}
\end{equs}
where $\genaF{\neq2}{}\eqdef \gena-c_2(F)\genaF{2}{}$ (and the last equality explains 
why we chose $G$ as such). 
To develop some intuition as to why the r.h.s. should vanish, let us replace $(1-\cL^{\tau,2})^{-1}\phi$ by a 
function $\psi$ in $\fock{1}{1}{}=H^1(\R^2)$. Then, by Lemma~\ref{lem:ActionGen}, 
since $j+1\leq n=1$ implies $j=0$, $\genaF{m}{a}\psi\neq 0$ only if $a=m-1$, 
so that $\genaF{\neq2}{}\psi=\sum_{m\neq 2}c_m(F)\genaF{m}{m-1}\psi$.  
By orthogonality of Fock spaces with different indices (and neglecting $m=1$ for the sake of the present discussion), 
we get  
\begin{equs}[e:1chaosHeu]
\|\genaF{\neq2}{}\psi\|_{\fock{}{-1}{\tau}}^2&=\sum_{m> 2}c_m(F)^2\|\genaF{m}{m-1}\psi\|^2_{\fock{m}{-1}{\tau}}\\
&=\sum_{m> 2} \frac{c_m(F)^2}{m!} \frac{\tau^{2-m}\nu_\tau^{\frac{1+m}{2}}}{(2\pi)^{m-1}}\int (\fe_1 \cdot p_{[1:m]})^2 |\psi(p_{[1:m]})|^2 \frac{ \Xi_{m}^\tau(\dd p_{1:m})}{1+\frac12|\sqrt{R_\tau} p_{1:m}|^2}\\
&=\sum_{m> 2} \frac{c_m(F)^2}{m!} \frac{\tau^{2-m}\nu_\tau^{\frac{1+m}{2}}}{(2\pi)^{m-1}}\int \dd p\,(\fe_1 \cdot p)^2 |\psi(p)|^2\int_{p_{[1:m]}=p} \frac{ \Xi_{m}^\tau(\dd p_{1:m})}{1+\frac12|\sqrt{R_\tau} p_{1:m}|^2}
\end{equs}
where we used~\eqref{e:Genamj}. 
Now, the inner integral can be bounded as 
\begin{equs}
\int_{p_{[1:m]}=p} \frac{ \Xi_{m}^\tau(\dd p_{1:m})}{1+\frac12|\sqrt{R_\tau} p_{1:m}|^2}&= \tau^{m-2} \nu_\tau^{\frac{1-m}{2}}\int_{p_{[1:m]}=\tau^{1/2}R_\tau^{-\frac12}p}\; \frac{ \Xi_{m}^1(\dd p_{1:m})}{\tau^{-1}+\frac12|p_{1:m}|^2}\\
&\leq (2\pi)^{m-1}\tau^{m-2} \nu_\tau^{\frac{1-m}{2}}\int_{\R^4}\; \frac{ \Xi_{2}^1(\dd p_{1:2})}{\frac12|p_{1:2}|^2}
\end{equs}
and the last integral is finite and independent of $m$. Plugging it back into~\eqref{e:1chaosHeu}, we conclude 
\begin{equ}[e:H-1anis]
\|\genaF{\neq2}{}\psi\|_{\fock{}{-1}{}}^2\lesssim \sum_{m> 2} \frac{c_m(F)^2}{m!}\int \dd p\,\nu_\tau(\fe_1 \cdot p)^2 |\psi(p)|^2\leq \|F\|^2_{L^2(\pi_1)} \|(-\nu_\tau  \gensyx)^{\frac12}\psi\|^2 
\end{equ}
where the norm on $F$ is defined in~\eqref{e:L2NormGaussian}. Therefore, 
provided the anisotropic norm of $\psi$ at the r.h.s. converges to $0$, as is the case since $\psi\in H^1(\R^2)$, 
then so does the $\fock{}{-1}{\tau}$-norm of $\genaF{\neq2}{}\psi$, thus suggesting that 
the operator $\genaF{\neq2}{}$ should not contribute to the limit. 
\medskip

Even though the arguments above give a useful indication as to what we should expect, 
they are by no means enough. Let us mention that 
computations similar to those in~\eqref{e:1chaosHeu} were first presented 
in~\cite{GP16}, where the authors proved a universality statement for the SBE in dimension $d=1$ 
in the so-called weak coupling scaling. But while for them, studying functionals 
in the first chaos was sufficient, in the present setting this is {\it not the case}. 
As a matter of fact, there are several difficulties we need to overcome:
\begin{enumerate}[label=(\Roman*),noitemsep]
\item\label{i:I} {\bf Chaos Growth:} to bound the r.h.s. of~\eqref{e:HeuH-1}, we need to control the $\fock{}{-1}{\tau}$-norm 
of $\genaF{\neq2}{}$ applied to $(1-\cL^{\tau,2})^{-1}\phi$. Now, even for $\phi$ in the first chaos, 
the latter will have non-trivial components in {\it every} chaos. This means that we need 
to estimate $\fock{}{-1}{\tau}$-norm of $\genaF{m}{a}$ for {\it every $m$ and $a$},  
when applied to rather general elements of the Fock space. 
As it turns out (and true even for $m=2$ and $a=\pm1$), these operators naturally introduce 
a dependence in the so-called number operator (see Definition~\ref{def:NoOp}), which in turn has to be tamed. 
\item\label{i:II} {\bf Well-posedness:} the first equality in~\eqref{e:HeuH-1} is {\it not rigorous}. Indeed, 
for it to make sense we should know that $(1-\cL^{\tau,2})^{-1}\phi$ belongs to the 
domain of $\gen$, which is all but obvious. According to Lemma~\ref{lem:ActionGen}, we only 
know that $\core\subset {\rm Dom}(\gen)$, for $\core$ the set of  
$f = (f_n)_{n \geq 0} \in \fock{}{}{\tau}$ such that for every $n$ the Fourier transform of $f_n$ 
is compactly supported and $f_n\neq 0$ only for finitely many $n$'s, 
but $(1-\cL^{\tau,2})^{-1}\phi\notin \core$. 
\end{enumerate}

This being said, in order to prove~\eqref{e:ResConv}, 
there was no need to choose {\it precisely} $(1-\cL^{\tau,2})^{-1}\phi$ in~\eqref{e:gengen2} 
and~\eqref{e:gen2geneff} in the first place. To retain (at least 
part of) the argument above, and in particular that in~\eqref{e:HeuH-1} 
and the results of~\cite{CMT}, all we need is an Ansatz 
that {\it approximates  $(1-\cL^{\tau,2})^{-1}\phi$ sufficiently well}. 
Let us mention that, already in~\cite{CMT}, an Ansatz was introduced to show~\eqref{e:gen2geneff}, 
but we want to modify it in such a way that, for every fixed $\tau$, it 
belongs to $\core$ (so that~\ref{i:II} holds). More precisely, our Ansatz $s^\tau=(s^\tau_n)_{n\in\N}$ 
will have finite chaos expansion, i.e. $s^\tau_n\equiv 0$ for $n\geq \bar n$ with $\bar n$ suitably 
depending on $\tau$, and, 
for each $n$, $s^\tau_n$ will have compactly supported Fourier transform, 
whose radius will depend on {\it both $n$ and $\tau$}. 
It is this latter property that will allow us to tame the growth in the chaos 
and ultimately overcome point~\ref{i:I} above. 
\medskip

The rest of the paper is organised as follows. In Section~\ref{s:gen_est_nonq}, 
we derive fine estimates on $\genaF{\neq2}{}\phi$, 
for $\phi\in\core$, that keep track of the size of the 
support of the Fourier transform of the kernels of $\phi$. This is the place at which 
the Assumption~\ref{Assumption:F} on $F$ is used in its full strength. 
Then, in Section~\ref{s:ansatz_QSBE}, we recall the Ansatz 
introduced in~\cite{CMT}, describe the novel Ansatz $s^\tau$ 
and show it still satisfies (the analogue of)~\eqref{e:gengen2} and~\eqref{e:gen2geneff}. 
At last, Section~\ref{s:pf_summary} is devoted to the proof of Theorems~\ref{th:CLT} and~\ref{th:conv_semigroup}.

\subsection{Estimates for the non-quadratic part of the generator} \label{s:gen_est_nonq}

In the following proposition we estimate the  $\fock{}{-1}{\tau}$-norm of $\genaF{m}{a}$, for every $m\neq 2$ and 
$a\in\{-m+1,\dots, m+1\}$. The elements to which these operators are applied 
belong to $\core$ and there are two crucial aspects whose dependence in the bounds we want to keep track of:
\begin{enumerate}[noitemsep]
\item the size of the support of the Fourier transform of their kernels, and 
\item the anisotropic $\fock{}{1}{\tau}$-norm of $\phi$, i.e. $\|(-\nu_\tau  \gensyx)^{\frac12}\phi\|$ (see~\eqref{e:H-1anis}). 
\end{enumerate}
The first is essential to tame the growth in the number operator, the latter guarantees that 
these terms stand a chance of having a vanishing contribution (see~\eqref{e:H-1anis} again).

\begin{proposition} \label{p:HighDegreeEstimate}
There exists a constant $C > 0$ such that for every $\tau > 0$, $m,\, n\in\N$, $m \neq 2$, $\chi\in[0,1]$  
and $\Phi \in \fock{n}{}{\tau}$ for which $\mathrm{Supp}(\hat{\Phi})\subset\{p_{1:n} \, \colon \, |\sqrt{R_\tau p_{1:n}}|^2 \leq  \tau \chi\}$, it holds that
\begin{equ}[e:BoundAmne2]
\|\sqrt{m!}\cA^{\tau, m} \Phi\|_{\fock{}{-1}{\tau}} \leq C \Big(n^{3} m^{3} + \chi \, 4^n m^{n}\Big) \|(- \nu_\tau \gensyx)^{\frac12}\Phi\|_\tau^\alpha\, \|\Phi\|_{\fock{}{1}{\tau}}^{1-\alpha} \, .
\end{equ}
for any $\alpha\in(0,1)$. 
\end{proposition}
\begin{proof}
Let us first deal with the case $m=1$. From \eqref{e:Gena1}, we have 
\begin{equ}
\|\genaF{1}{} \Phi\|_{\fock{}{-1}{\tau}}^2 = n! \tau \nu_\tau \int \Xi^\tau_n(\dd p_{1:n}) \, |\hat{\Phi}(p_{1:n})|^2 \, \frac{\big(\sum_{i=1}^n \fe_1 \cdot p_i \, (\Theta^\tau(p_i) - 1)\big)^2}{1 + |\sqrt{R_\tau} p_{1:n}|^2/2} \, .
\end{equ}
where the factor $n!$ comes from the definition of the $\fock{}{-1}{\tau}$-norm 
(see~\eqref{e:normFock} with the $H^{-1}_\tau$-norm in~\eqref{e:normHeps} in place of the $L^2_\tau$ one). 
Recall from~\eqref{e:Mollifiers} the definition of $\Theta^\tau$, i.e. 
$\Theta^\tau(p) = \big(2 \pi \hat{\rho}(\tau^{-1/2} \sqrt{R_\tau} p)\big)^2$, and that 
$\rho$ is smooth, compactly supported and of mass $1$ so that $\Theta^\tau(0) = 1$. 
This means that $\Theta^\tau$ is globally Lipschitz and we have 
$|\Theta^\tau(p) - 1| \lesssim \tau^{-1/2} \, |\sqrt{R_\tau} p|$ uniformly in $p\in\R^2$. 
Thus, Cauchy--Schwarz inequality leads to
\begin{equ}
\tau \nu_\tau \, \frac{\big(\sum_{i=1}^n \fe_1 \cdot p_i \, (\Theta^\tau(p_i) - 1)\big)^2}{1 + |\sqrt{R_\tau} p_{1:n}|^2/2} \lesssim \nu_\tau \frac{|\fe_1 \cdot p_{1:n}|^2 \, |\sqrt{R_\tau} p_{1:n}|^2}{1+|\sqrt{R_\tau} p_{1:n}|^2} \leq \nu_\tau |\fe_1 \cdot p_{1:n}|^2 \, 
\end{equ}
from which we immediately deduce~\eqref{e:BoundAmne2} with a prefactor independent of $n$. 
\medskip

Now, let $m \geq 3$, $j \in \{0, \dots, m-1\}$ and $a = m-2j-1$. From \eqref{e:Genamj}, we get 
\begin{equs} \label{e:Gena3+Bound1}
{}&\|\genaF{m}{a} \Phi\|_{\fock{}{-1}{\tau}}^2 = (n+a)!  \, \frac{n!^2}{(n+a)!^2 (j+1)!^2} \frac{\tau^{- m + 2} \nu_\tau^{\frac{m+1}{2}}}{(2 \pi)^{2(m-1)}} \times \\
&\times \int \frac{\Xi_{n+a}^\tau(\dd p_{1:n+a})}{1 + \frac12|\sqrt{R_\tau} p_{1:n+a}|^2} 
\ \abs[4]{\sum_{I \in \Delta_{m-j}^{n+a}} \fe_1 \cdot p_{[I]} \int_{r_{[1:j+1]} = p_{[I]}} \Xi_{j+1}^\tau(\dd r_{1:j+1}) \, \hat{\Phi}(r_{1:j+1}, p_{1:n+a\setminus I})}^2 \, .
\end{equs}
To take into account the assumption that $\phi$ has compactly supported Fourier transform, 
let $V_\chi^\tau \eqdef \bigsqcup_{k \geq 0} \{x_{1:k} \in \R^{2k} \, , \ |\sqrt{R_\tau} x_{1:k}|^2 \leq \chi \, \tau\}$, 
so that in particular ${\rm Supp}(\hat\phi)\subset V_\chi^\tau$. Note that $V_\chi^\tau$ 
is stable under restriction of components meaning that, for subsets~$A \subseteq B\subset \N$, 
if $p_B=(p_i)_{i\in B} \in V_\chi^\tau$, then so does $p_A$. 
Hence, for any $p_{1:n}\in\R^{2n}$ and $I\subset \{1:n\}$, we have 
\begin{equ}[e:Support]
\hat\phi(p_{1:n})=\1_{p_I\in V^\tau_\chi} \hat\phi(p_{1:n})\,,
\end{equ}
a fact that will be repeatedly used throughout the proof. 
With this in mind, let us expand the square in~\eqref{e:Gena3+Bound1} 
\begin{equs}
\Big|&\sum_{I \in \Delta_{m-j}^{n+a}} \fe_1 \cdot p_{[I]} \int_{r_{[1:j+1]} = p_{[I]}} \Xi_{j+1}^\tau(\dd r_{1:j+1}) \, \hat{\Phi}(r_{1:j+1}, p_{1:n+a\setminus I})\Big|^2\\
&\leq \sum_{I,J \in \Delta_{m-j}^{n+a}} \1_{p_{I\setminus J}\in V^\tau_\chi}|\fe_1 \cdot p_{[I]}|\int_{r_{[1:j+1]} = p_{[I]}} \Xi_{j+1}^\tau(\dd r_{1:j+1}) \, |\hat{\Phi}(r_{1:j+1}, p_{1:n+a\setminus I})|\\
&\qquad\qquad\times\1_{p_{I\setminus J}\in V^\tau_\chi}|\fe_1 \cdot p_{[J]}|\int_{r_{[1:j+1]} = p_{[J]}} \Xi_{j+1}^\tau(\dd r_{1:j+1}) \, |\hat{\Phi}(r_{1:j+1}, p_{1:n+a\setminus J})|\\
&\leq \binom{n+a}{m-j} \sum_{I \in \Delta_{m-j}^{n+a}} \1_{p_{I \setminus \{1:m-j\} \in V_{\chi}^{\tau}}} |\fe_1 \cdot p_{[I]}|^2 \times \\
&\qquad \qquad \times \Big(\int_{r_{[1:j+1]} = p_{[I]}} \Xi_{j+1}^\tau(\dd r_{1:j+1}) \, |\hat{\Phi}(r_{1:j+1}, p_{1:n+a\setminus I})|\Big)^2\label{e:square_term}
\end{equs}
where, in the first step, we used~\eqref{e:Support} and 
that both $I\setminus J$ and $J\setminus I\subseteq (\{1:n+a\}\setminus I)\cup  (\{1:n+a\}\setminus J)$, 
while in the second the trivial bound $2ab\leq a^2+b^2$  
(fixing w.l.o.g. $J=\{1:m-j\}$) and that 
the cardinality of $\Delta_{m-j}^{n+a}$ is $\binom{n+a}{m-j}$. 
To bound the square of the integral, let $\delta,\beta\in[0,1]$ be arbitrary for now 
(we will choose them later on) and apply Cauchy--Schwarz, so that 
\begin{equs}
\bigg(&\int_{r_{[1:j+1]} = p_{[1:m-j]}} \Xi^\tau_{j+1}(\dd r_{1:j+1}) \, |\hat{\Phi}(r_{1:j+1}, p_{1:n+a \setminus I})|\bigg)^2 \\
&\leq\bigg(\int_{ r_{[1:j+1]} = p_{[I]}} \Xi^\tau_{j+1}(\dd  r_{1:j+1}) \, (\nu_\tau |\fe_1 \cdot  r_{1:j+1}|^2)^\delta |\sqrt{R_\tau}  r_{1:j+1}|^{2 \beta} \,  |\hat{\Phi}( r_{1:j+1}, p_{1:n+a \setminus I})|^2\bigg) \\
&\qquad\times \bigg(\int_{\tilde r_{[1:j+1]} = p_{[I]}} \Xi_{j+1}^\tau(\dd \tilde r_{1:j+1}) \frac{\1_{\tilde r_{1:j+1} \in V_{\chi}^{\tau} }}{(\nu_\tau |\fe_1 \cdot \tilde r_{1:j+1}|^2)^\delta |\sqrt{R_\tau} \tilde r_{1:j+1}|^{2 \beta}}\bigg)\\
&\leq \bigg(\int_{ r_{[1:j+1]} = p_{[I]}} \Xi^\tau_{j+1}(\dd  r_{1:j+1}) \, (\nu_\tau |\fe_1 \cdot  r_{1:j+1}|^2)^\delta |\sqrt{R_\tau}  r_{1:j+1}|^{2 \beta} \,  |\hat{\Phi}( r_{1:j+1}, p_{1:n+a \setminus I})|^2\bigg) \\
&\qquad\times \bigg(\int \Xi_{j}^\tau(\dd \tilde r_{1:j}) \frac{\1_{\tilde r_{1:j} \in V_{\chi}^{\tau} }}{(\nu_\tau |\fe_1 \cdot \tilde r_{1:j}|^2)^\delta |\sqrt{R_\tau} \tilde r_{1:j}|^{2 \beta}}\bigg)\label{e:aux_bound:1}
\end{equs}
where we inserted the indicator function of $\{r_{1:j+1}\in V^\tau_\chi\}$ for free thanks to~\eqref{e:Support} 
and in the last bound we removed an integration variable and the corresponding constraint on the integral.  
For the scalar product multiplying the square in~\eqref{e:square_term}, notice that, 
since $p_{[I]}=r_{[1:j+1]}$, we can estimate $|\fe_1 \cdot p_{[I]}|^2 \leq (m-j) |\fe_1 \cdot p_{I}|^2$ 
and $ |\fe_1 \cdot p_{[I]}|^2 \leq (j+1) |\fe_1 \cdot r_{1:j+1}|^2$. Therefore, since further 
$(m-j)\vee(j+1)\leq m+1$, for any $\gamma\in[0,1]$, we get 
\begin{equ}
|\fe_1 \cdot p_{[I]}|^2\leq (m+1) |\fe_1 \cdot p_{I}|^{2(1-\gamma)}|\fe_1 \cdot r_{1:j+1}|^{2\gamma}
\end{equ}
which then gives 
\begin{equ}[e:aux_bound:2]
\frac{ |\fe_1 \cdot p_{[I]}|^2}{1 + |\sqrt{R_\tau} p_{1:n+a}|^2/2} \lesssim \nu_\tau^{-1}(m+1) \,  \frac{(\nu_\tau |\fe_1 \cdot r_{1:j+1}|^2)^{\gamma}}{|\sqrt{R_\tau} p_{I}|^{2 \gamma}} \, .
\end{equ}
Plugging~\eqref{e:square_term},~\eqref{e:aux_bound:1} and~\eqref{e:aux_bound:2} 
into~\eqref{e:Gena3+Bound1}, relabelling the variables $(r_{1:j+1}, p_{1:n+a\setminus I})=q_{1:n}$ 
and exploiting the multiplicative structure of $\Xi^\tau_{n+a}$, 
we deduce 
\begin{equs}
\|&\genaF{m}{a} \Phi\|_{\fock{}{-1}{\tau}}^2 \lesssim   (n+a)! \frac{n!^2}{(n+a)!^2 (j+1)!^2} \binom{n+a}{m-j} (m+1)\frac{\tau^{- m + 2} \nu_\tau^{\frac{m-1}{2}}}{(2 \pi)^{2(m-1)}} \times \\
&\quad\times \int \Xi^\tau_{n}(\dd q_{1:n}) (\nu_\tau |\fe_1\cdot q_{1:j+1}|^2)^{\delta+\gamma} |\sqrt{R_\tau}q_{1:j+1}|^{2\beta} |\hat\phi(q_{1:n})|^2\times\\
&\quad\times \int \Xi_{j}^\tau(\dd \tilde r_{1:j}) \frac{\1_{\tilde r_{1:j} \in V_{\chi}^{\tau} }}{(\nu_\tau |\fe_1 \cdot \tilde r_{1:j}|^2)^\delta |\sqrt{R_\tau} \tilde r_{1:j}|^{2 \beta}} \sum_{I \in \Delta_{m-j}^{n+a}}\int_{p_{[I]}=q_{[1:j+1]}}\Xi_{m-j}^\tau(\dd p_{I}) \frac{\1_{p_{I\setminus\{1:m-j\}}\in V^\tau_\chi}}{|\sqrt{R_\tau}p_I|^2} \\
&=n!  \, \frac{n!}{(n+a)! (j+1)!^2} \binom{n+a}{m-j} (m+1)\frac{1}{(2 \pi)^{2(m-1)}} \times \\
&\quad\times \int \Xi^\tau_{n}(\dd q_{1:n}) (\nu_\tau |\fe_1\cdot q_{1:j+1}|^2)^{\delta+\gamma} |\sqrt{R_\tau}q_{1:j+1}|^{2\beta} |\hat\phi(q_{1:n})|^2 \,\mathbf{A}^j_{\delta,\beta}\sum_{I \in \Delta_{m-j}^{n+a}}\mathbf{B}_{\gamma}^I\label{e:Gena3+Bound2}
\end{equs}
where in the second step we applied the change of variables $\tau^{-\frac12}\sqrt{R_\tau}\tilde r_i=x_i$, 
$i=1,\dots, j+1$ to the integral over $\tilde r_{1:j+1}$ and 
similarly for that over $p_{I}$. For $\delta,\beta,\gamma\in[0,1]$ and $I\in\Delta^{n+a}_{m-j}$, 
$\mathbf{A}^j_{\delta,\beta}$ and 
$\mathbf{B}_{\gamma}^I$ are the (essentially) $\tau$-independent integrals defined by 
\begin{equs}
\mathbf{A}^j_{\delta,\beta}&\eqdef \int \Xi_{j}^1(\dd x_{1:j}) \frac{\1_{x_{1:j} \in V_{\chi}^{1} }}{|\fe_1 \cdot  x_{1:j}|^{2\delta} | x_{1:j}|^{2 \beta}}\,,\label{e:A1}\\
\mathbf{B}^I_{\gamma}&\eqdef \int_{y_{[I]}=\tau^{-\frac12}\sqrt{R_\tau}q_{[1:j+1]}}\Xi_{m-j}^1(\dd y_{I}) \frac{\1_{y_{I\setminus\{1:m-j\}}\in V^1_\chi}}{|y_I|^{2\gamma}}\,.\label{e:A2}
\end{equs}
We now choose $\delta,\beta,\gamma\in[0,1]$ depending on $m$ and $j$ 
so that~\eqref{e:A1} and~\eqref{e:A2} are finite 
and bounded by a quantity that, at least for large enough $m$, depends explicitly on $\chi$. 
To do so, we will consider four different cases: (1) $j\geq 3,\,m\geq 4$; (2) $j=2$; (3) $j=1$; and (4) $j=0$. 
In each of them, our choice, which is neither unique nor optimal, 
will be made to minimise the value of $\beta$ as this is 
responsible for the full $\fock{}{1}{\tau}$-norm of $\phi$ at the r.h.s. of~\eqref{e:BoundAmne2} (see~\eqref{e:Gena3+Bound3}). Our bound is   
obtained by the following procedure: remove from the denominator the maximal number of variables 
for which the resulting integral is finite; remove from the indicator function 
those variables which are still at the denominator; and integrate one of the remaining variables 
(provided there is one) 
using the trivial estimate 
\begin{equ}[e:ReallyBasic]
\int \Xi^1_1(\dd z) \1_{z\in V^1_\chi}=(2\pi)^2 \int \hat\rho(z)^2 \1_{z\in V^1_\chi} \lesssim (2\pi)^2 \chi\,.
\end{equ}
More concretely, in case (1) (i.e. $j\geq 3,\,m\geq 4$) we take $\beta=\gamma=0$, 
and estimate $\mathbf{A}^j_{\delta,0}$ as  
\begin{equ}[e:A11]
\mathbf{A}^j_{\delta,0}\leq 
\int  \frac{\Xi_{3}^1(\dd x_{1:3})}{|\fe_1 \cdot  x_{1:3}|^{2\delta}} \bigg(\int \Xi_{1}^1(\dd x_{1:j}) \1_{x_{4:j} \in V_{\chi}^{1}} \bigg)^{j-3}\lesssim (2\pi)^{2j}\chi^{\1_{j\geq 4}}
\end{equ}
the last bound being a consequence of the fact that for any $x\in\R^2$, $\fe_1\cdot x$ is a 
one-dimensional variable (meaning that the first integral is effectively over $\R^3$) so that it 
holds for any $\delta'\in[0,1]$, while for $\mathbf{B}^I_{0}$ we have  
\begin{equ}[e:21]
\mathbf{B}^I_{0}\lesssim \int_{y_{[I]}=\tau^{-\frac12}\sqrt{R_\tau}q_{[1:j+1]}}\Xi_{m-j}^1(\dd y_{I}) \1_{y_{I\setminus\{1:m-j\}}\in V^1_\chi}\lesssim (2\pi)^{2(m-j-1)}\chi^{\1_{|I\setminus\{1:m-j\}|\geq 2}}\,.
\end{equ}
In case (2), i.e. $j=2$, we take $\gamma=0$, and bound $\mathbf{B}^I_{0}$ as in~\eqref{e:21}, 
and $\mathbf{A}^j_{\delta,\beta}$ as 
\begin{equ}
\mathbf{A}^2_{\delta,\beta}\leq \int  \frac{\Xi_{2}^1(\dd x_{1:2})}{|\fe_1 \cdot  x_{1:2}|^{2\delta} | x_{1:2}|^{2 \beta}}\leq \int \frac{\Xi_{2}^1(\dd x_{1:2})}{|\fe_1 \cdot  x_{1:2}|^{2\delta} | \fe_2\cdot x_{1:2}|^{2 \beta}} 
\end{equ}
that is bounded by $(2\pi)^2$ for any $\delta,\beta$ such that $\delta\vee\beta<1$. 
In case (3), i.e. $j=1$, take $\beta=0$, so that   
\begin{equs}
\mathbf{A}^1_{\delta,0}&\leq \int \frac{\Xi_{1}^1(\dd x_1)}{|\fe_1 \cdot  x_{1}|^{2\delta} }\lesssim (2\pi)^2
\end{equs}
for any $\delta<1/2$, and
\begin{equ}[e:A23]
\mathbf{B}_{\gamma}^I\leq \int_{y_{[I]}=\tau^{-\frac12}\sqrt{R_\tau}q_{[1:2]}}\Xi_{m-1}^1(\dd y_{I}) \frac{\1_{y_{I\setminus\{1:m-1, i_1,i_2\}}\in V^1_\chi}}{|(y_{i_1},y_{i_2})|^{2\gamma}}\lesssim (2\pi)^{2(m-2)}\chi^{\1_{|I\setminus\{1:m-j\}|\geq 3}}\,,
\end{equ}
where we took $i_1,\,i_2\in I$, and the last bound is valid for any $\gamma\in[0,1]$. 
At last, in case (4), i.e. $j=0$, $\mathbf{A}^j_{\delta,\beta}=1$ as there is no integral (so we take $\delta=\beta=0$), 
and the other is estimated as in~\eqref{e:A23} for an arbitrary $\gamma\in[0,1]$.  

Overall, getting back to~\eqref{e:Gena3+Bound2}, 
we have shown that, for suitable values of $\delta,\beta,\gamma\in[0,1]$, and since $a=m-2j-1$, 
we have 
\begin{equs}[e:Gena3+Bound3]
\|&\sqrt{m!}\genaF{m}{a} \Phi\|_{\fock{}{-1}{\tau}}^2 \lesssim   n! \frac{n! m!}{(n+a)! (j+1)!^2} \binom{n+a}{m-j} (m+1)\times \label{e:Gena3+Bound3}\\
&\quad\times \int \Xi^\tau_{n}(\dd q_{1:n}) (\nu_\tau |\fe_1\cdot q_{1:j+1}|^2)^{\delta+\gamma} |\sqrt{R_\tau}q_{1:j+1}|^{2\beta} |\hat\phi(q_{1:n})|^2 \,\chi^{\1_{j\geq 4}}\sum_{I \in \Delta_{m-j}^{n+a}} \chi^{\1_{|I\setminus\{1:m-j\}|\geq 3}}\\
&= \binom{n}{j}\binom{m+1}{j+1}\chi^{\1_{j\geq 4}} \sum_{k=0}^{m-j} \binom{n-j-1}{k}\binom{m-j}{k}\chi^{\1_{k\geq 3}}\times\\
&\quad\times n! \int \Xi^\tau_{n}(\dd q_{1:n}) (\nu_\tau |\fe_1\cdot q_{1:j+1}|^2)^{\delta+\gamma} |\sqrt{R_\tau}q_{1:j+1}|^{2\beta} |\hat\phi(q_{1:n})|^2\\
&\leq  \binom{n}{j}\binom{m+1}{j+1}\chi^{\1_{j\geq 4}}\sum_{k=0}^{m+1} \binom{n}{k}\binom{m+1}{k}\chi^{\1_{k\geq 3}} \|(- \nu_\tau \gensyx)^{\frac12}\Phi\|_\tau^{2(\delta+\gamma)} \| (- \gensy)^{\frac12} \Phi\|_\tau^{2\beta}
\end{equs}
where in the last step, we used monotonicity of the binomial, 
chose $\delta+\beta+\gamma=1$, applied H\"older's inequality to 
the integral (with exponents $1/(\delta+\gamma)$ and $1/\beta$), and 
included the $n!$ in the norm. 
As a consequence, upon summing the previous over $j\in\{0,m-1\}$, we get 
\begin{equs}
\|&\sqrt{m!}\genaF{m}{} \Phi\|_{\fock{}{-1}{\tau}}^2\lesssim \bigg(\sum_{j=0}^{m+1} \binom{n}{j}\binom{m+1}{j}\chi^{\1_{j\geq 4}}\bigg)^2\|(- \nu_\tau \gensyx)^{\frac12}\Phi\|_\tau^{2(\delta+\gamma)} \| (- \gensy)^{\frac12} \Phi\|_\tau^{2\beta}
\end{equs}
and we can bound the sum as 
\begin{equ}
\sum_{j = 0}^{m+1} \binom{n}{j} \binom{m+1}{j} \chi^{\1_{\ell \geq 4}} \lesssim n^3 m^3 + \chi \, 4^n \,  m^n \, ,
\end{equ}
Therefore, the statement follows at once upon choosing $\alpha$ in the statement to be $\delta+\gamma$ 
herein. Indeed, the anisotropic $\fock{}{-1}{\tau}$-norm $\|(- \nu_\tau \gensyx)^{\frac12}\cdot\|_\tau$ 
is upperbounded by the full $\fock{}{1}{\tau}$-norm and for any of the choices of $\delta,\beta,\gamma\in[0,1]$ 
such that $\delta+\beta+\gamma=1$, the only restriction we have is $\delta+\gamma<1$ and $\beta>0$ 
(which are due to case (2) above, in all the others we could have chosen $\delta+\gamma=1$ and $\beta=0$).  
\end{proof}

In order to deduce a bound on $\genaF{\neq2}{}$ from Proposition~\ref{p:HighDegreeEstimate}, 
it remains to use the fact 
that $\genaF{\neq2}{}=\genaF{}{}-\genaF{2}{}$, for $\genaF{}{}$ defined in~\eqref{e:AmA}, and 
choose suitably the way in which the size of the Fourier support of the elements of $\core$ 
grows with the number operator so that 
Assumption~\ref{Assumption:F} can be exploited.

\begin{corollary}\label{cor:HighDegreeEstimate}
Under Assumption \ref{Assumption:F}, for every $\Phi \in \core$ for which, for every $n \geq 1$,  
$\mathrm{Supp}(\hat{\Phi}_n)\subset\{p_{1:n} \, \colon \, |\sqrt{R_\tau p_{1:n}}|^2 \leq \tau\chi_n\}$ 
with $\chi_n=\exp(-\exp(c n))$ for some $c>0$, 
there exists a constant $C = C(c)>0$ such that, for any $\alpha\in(0,1)$, we have 
	\begin{equ}  \label{e:BoundAne2}
		\|\genaF{\neq 2}{} \Phi\|_{\fock{}{-1}{\tau}} \leq C  \, \|\cN^{\frac{3}{\alpha}}(- \nu_\tau \gensyx)^{\frac12}\Phi\|_\tau^\alpha \,\| \Phi\|_{\fock{}{1}{\tau}}^{1-\alpha} \,.
	\end{equ}
\end{corollary}

\begin{proof}
In essence, the statement is a consequence of Proposition~\ref{p:HighDegreeEstimate} and Proposition~\ref{p:AssF}. 
Indeed, we have 
\begin{equs}
\|\cA^{\tau, \ne 2} \Phi_n\|_{\fock{}{-1}{\tau}}&\leq \sum_{m\geq 3} \frac{|c_m(F)|}{\sqrt{m!}}\|\genaF{m}{}\Phi_n\|_{\fock{}{-1}{\tau}}\\
&\lesssim \|(- \nu_\tau \gensyx)^{\frac12}\Phi_n\|_\tau^\alpha \|\Phi_n\|_{\fock{}{1}{\tau}}^{1-\alpha} \sum_{m\geq 3} \frac{c_m(F)}{\sqrt{m!}} \Big(n^{3} m^{3} + \chi_n \, 4^n m^{n}\Big) \,
\end{equs}
and we can estimate the sum at the r.h.s.  by using~\eqref{e:expexpMomentGrowth}. 
In formulas, for every $\gamma>0$, we have 
\begin{equs}
 \sum_{m\geq 3} \frac{|c_m(F)|}{\sqrt{m!}} \Big(n^{3} m^{3} + \chi_n \, 4^n m^{n}\Big)\lesssim n^{3} \exp(\exp(3\gamma/2)) + \chi_n 4^{n}  \exp(\exp(n\gamma/2))
\end{equs}
so that, since $\chi_n=\exp(-\exp(c n))$, it suffices to choose $\gamma\in(0,2c)$, and the statement 
follows at once. 
\end{proof}

\subsection{The amended Ansatz for the resolvent of the quadratic Burgers equation} \label{s:ansatz_QSBE}

Before presenting the modification we want to perform, let us recall the original Ansatz introduced  
in~\cite[Sec. 4.1]{CMT} for the solution of the Burgers equation with quadratic 
nonlinearity, and state some of the properties it enjoys. 
For a heuristic explaining how it was deduced, we refer the interested reader to~\cite[Sec. 2.3]{CMT}. 
\medskip

Let $\phi$ be an element of $\fock{n_0}{}{\tau}$ for some fixed $n_0\ge 1$. 
We define $\tilde s^\tau$ according to~\cite[Eq. (4.7)]{CMT}, i.e.  
\begin{equ}[e:AnsatzFormula]
\tilde s^\tau \eqdef \sum_{k=0}^{+\infty} c_2(F)^k \,(\cG_M^\tau \genaFsh{1}{\sharp})^k \, \cG_M^\tau \phi \, .
\end{equ}
where, for $M\geq 0$, $\cG_M^\tau$ is the diagonal operator, according to Definition~\ref{def:DiagOp}, 
in~\cite[Eq. (3.18)]{CMT} given by 
\begin{equ}[e:cgG]
\cG_M^\tau \eqdef \big(1 - \gensy + \cg_M^\tau\big)^{-1}\qquad \text{with}\qquad \cg_M^\tau 
\eqdef -g_M^\tau(L^\tau(- \cS^\tau)) \, \gensyx
\end{equ}
for $g_M^\tau(x)\eqdef  (M \nu_\tau) \vee g^\tau(x)$, and $g^\tau$ and $L^\tau$ the functions on $\R_+$ 
respectively defined by 
\begin{equ}[e:gMLtau]
g^\tau(x) \eqdef \big(\tfrac{3}{2} \, x + \nu_\tau^{3/2}\big)^{2/3} - \nu_\tau\quad\text{and}\quad L^\tau(x) \eqdef \frac{c_2(F)^2 \nu_\tau^{3/2}}{\pi} \log \Big(1 + \frac{\tau}{1+x}\Big)\,,
\end{equ}
while $\genaFsh{1}{\sharp}$ is the operator in~\cite[Eq. (3.7)]{CMT} whose action in Fourier 
on a given element $\psi\in\fock{n}{}{\tau}$ reads 
\begin{equs}[e:A+sharp]
\cF(\genaFsh{1}{\sharp} \psi) \, (p_{1:n+1}) &\eqdef  -\iota\frac{\nu_\tau^{3/4}}{2\pi} \frac{2}{n+1} \sum_{1 \leq i < j \leq n+1} [\fe_1 \cdot (p_i+p_j)]\Theta_\tau(p_i+p_j)  \times \\
&\qquad\qquad\qquad\times\chi_{p_i, p_j}^{\sharp,\tau}(p_i+p_j, p_{1:n+1\setminus\{i,j\}}) \hat{\psi}(p_i+p_j, p_{1:n+1\setminus\{i,j\}})
\end{equs}
for $\chi^{\sharp,\tau}$ the multiplier
\begin{equ}\label{e:ChiSharp}
\chi_{q,q'}^{\sharp,\tau}(p_{1:N}) \eqdef \1_{|\sqrt{R_\tau} q| \wedge |\sqrt{R_\tau} q'| > 2 \, |\sqrt{R_\tau} p_{1:N}|} \, .
\end{equ}
For future reference, let us also define $\genaFsh{-1}{\sharp}, \,\genaFsh{\pm1}{\flat}$ as 
\begin{equ}[e:Asharp-1Aflats]
\genaFsh{-1}{\sharp}\eqdef \big(-\genaFsh{1}{\sharp}\big)^\ast\qquad \text{and} \qquad \genaFsh{\pm1}{\flat}\eqdef\genaF{2}{\pm1}-\genaFsh{\pm1}{\sharp}\,.
\end{equ} 

\begin{remark}\label{rem:Notations}
Compared to~\cite{CMT}, we adapted the notations to be in line with those in the present paper: 
$\tilde s^\tau$ above is $s^\tau$ therein; 
the operators $\genaFsh{\pm1}{\sharp},\,\genaFsh{\pm1}{\flat}$ above correspond to the operators 
$\genapmsh,\,\genapmfl$; the generic coupling constant $\lambda>0$, 
which multiplied the nonlinearity in~\cite[Eq. (1.1)]{CMT} has been replaced in each of its 
occurrences by $c_2(F)$ and omitted from the definition of $\genaFsh{\pm1}{\sharp},\,\genaFsh{\pm1}{\flat}$ 
(this is the reason why it appears as a multiplicative factor in~\eqref{e:AnsatzFormula}). 
\end{remark}

With the previous definitions at hand, we can now recall the main properties of $\tilde s^\tau$. 
We will provide a thorough proof below providing precise references to~\cite[Sec. 4]{CMT} for the details. 

\begin{proposition}\label{p:Ansatz}
Assume $c_2(F)\neq 0$. Let $n_0\in\N$, $\phi\in\core_{n_0}$ be such that 
$\mathrm{Supp}(\cF \phi) \subset \{- \cS^\tau \leq R\}$ for some $R>2$, 
and $\tilde s^\tau$ be defined according to~\eqref{e:AnsatzFormula}. 
Then, there exists constants $C=C(c_2(F), n_0, R, M)>0$ and $q>0$ independent of $\tau$ such that 
for every $\delta\in(0,\frac14]$, we have 
\begin{equs}
\|\tilde s^\tau\|_{\fock{}{1}{\tau}}&\leq C\|\phi\|_\tau\label{e:AnsatzH1}\\
\|e^{q\cN} (-\nu_\tau\gensyx)^{\frac12}\tilde s^\tau\|_\tau&\leq C \nu_\tau^{\frac14-\delta}\|\phi\|_\tau\label{e:AnsatzAH1}
\end{equs}
for $\cN$ the number operator in Definition~\ref{def:NoOp}, and further 
\begin{equs}
\|\tilde s^\tau - \cG_M^\tau \phi\|_\tau &\leq C\,\nu_\tau^{3/4} \|\phi\|_\tau\label{e:AnsatzL2}\\
\|(1 - \cL^{\tau,2}) \tilde s^\tau - \phi\|_{\fock{}{-1}{\tau}} &\leq C \nu_\tau^{\frac14-\delta} \|\phi\|_\tau\label{e:AnsatzH-1}
\end{equs}
where $\cL^{\tau,2}\eqdef \gensy+c_2(F)\genaF{2}{}$. 
\end{proposition}
\begin{proof}
Let us separately discuss each of the estimates~\eqref{e:AnsatzH1}--\eqref{e:AnsatzH-1}, starting from~\eqref{e:AnsatzAH1}. 
\medskip

\noindent {\it Estimate~\eqref{e:AnsatzAH1}.} 
It is a direct consequence of~\cite[Prop.~4.8 (and Rem.~4.10)]{CMT} as can be seen 
upon taking $q\in[0, \log \sqrt{2})$, $w(k)\eqdef e^{2q k}$ so that the function $f^\tau_w$ in~\cite[Eq. (4.22)]{CMT} 
reads $f_w^\tau(x) = \nu_\tau^{1-\frac12e^{2q}} \, (\nu_\tau + g^\tau(x))^{e^{2q}/2}$. Indeed,~\cite[Eq. (4.21)]{CMT} 
gives 
\begin{equ}
e^{-q\,n_0}\| e^{-q\cN} (-\nu_\tau\gensyx)^{\frac12}\tilde s^\tau\|_\tau\leq 2 \|\big(-f^\tau_w(L^\tau(- \cS^\tau)) \, \gensyx\big)^{1/2}\cG^\tau_M\phi\|_\tau\lesssim \nu_\tau^{\frac12\left(1-\frac12e^{2q}\right)}\|\phi\|_\tau
\end{equ}
where in the last step we used that the operators $-f^\tau_w(L^\tau(- \cS^\tau)) \, \gensyx$ and $\cG^\tau_M$ 
commute, that 
\begin{equs}
\cG^\tau_M&=\big(1-\gensy-g_M^\tau(L^\tau(- \cS^\tau))\gensyx\big)^{-1}
\leq \big(1-\gensy-g^\tau(L^\tau(- \cS^\tau))\gensyx\big)^{-1}
\\
&=\big(1-\gensyy-[\nu_\tau+g^\tau(L^\tau(- \cS^\tau))]\gensyx\big)^{-1}\leq \big(1-[\nu_\tau+g^\tau(L^\tau(- \cS^\tau))]\gensyx\big)^{-1}
\end{equs}
 and $\cG_M^\tau\leq 1$, and that, as a consequence, for any $q\in[0,\log\sqrt{2})$, we get 
 \begin{equs}
(-f^\tau_w(L^\tau(- \cS^\tau)) \, \gensyx)^{1/2}\cG^\tau_M&\leq (-f^\tau_w(L^\tau(- \cS^\tau)) \, \gensyx)^{1/2}(\cG^\tau_M)^{1/2}\\
&\leq \nu_\tau^{\frac12\left(1-\frac12e^{2q}\right)}[\nu_\tau+g^\tau(L^\tau(0))]^{\frac12\left(e^{2q}-1\right)}\lesssim \nu_\tau^{\frac12\left(1-\frac12e^{2q}\right)}\,.
 \end{equs}

\noindent {\it Estimate~\eqref{e:AnsatzH1}.} It follows by~\cite[Eq. (4.13)]{CMT},~\eqref{e:AnsatzH1} 
and the fact that $\cG_M^\tau$ is a bounded 
operator. 
\medskip

\noindent {\it Estimate~\eqref{e:AnsatzL2}.} See Eq. (4.14) in~\cite[Prop. 4.3]{CMT}. 
\medskip

\noindent {\it Estimate~\eqref{e:AnsatzH-1}.} This estimate is not explicitly written in~\cite{CMT} 
but can be easily deduced from the results therein. Let us provide a few details. 
By~\cite[Lem.~4.2]{CMT}, $\tilde s^\tau$ satisfies 
\begin{equs}
	(1-\cL^{\tau,2}) \tilde{s}^\tau - \phi 
	&= 
\big(-c_2(F)^2\genaFsh{-1}{\sharp} \cG_M^\tau\genaFsh{1}{\sharp} -\cg_M^\tau\big) \tilde{s}^\tau- c_2(F)\cA^{\tau,\scalebox{0.6}{2}}_{\flat} \tilde{s}^\tau -  c_2(F)\genaFsh{-1}{\sharp} \cG_M^\tau \phi_\tau
	%
\end{equs}
where $\cA^{\tau,\scalebox{0.6}{2}}_{\flat}\eqdef \genaFsh{1}{\flat}+\genaFsh{-1}{\flat}$. 
Let us separately consider each summand. For the first, we exploit 
the Replacement Lemma~\cite[Lem.~3.6]{CMT}, which gives 
\begin{equ} 
	\|\big(-c_2(F)^2\genaFsh{-1}{\sharp} \cG_M^\tau\genaFsh{1}{\sharp} -\cg_M^\tau\big) \tilde{s}^\tau\|_{\fock{}{-1}{\tau}} \leq C \, \|(- \nu_\tau \gensyx)^{1/2} \tilde{s}^\tau\|_\tau \, .
\end{equ}
For the second, we use instead~\cite[Prop.~3.4]{CMT}, obtaining
\begin{equ}
	\|\genaFsh{\pm}{\flat} \tilde{s}^\tau\|_{\fock{}{-1}{\tau}} 
	\leq C \, \|\cN(- \nu_\tau \gensyx)^{1/2} \tilde{s}^\tau\|_\tau \, 
\end{equ}
and at last, the bound on the third follows by Eq. (3.31) in~\cite[Prop. 3.12]{CMT}. 
Collecting these estimates, and using~\eqref{e:AnsatzAH1} once again, one can easily deduce~\eqref{e:AnsatzH-1}. 
\end{proof}

What the previous proposition says is that, on the one hand, $\tilde s^\tau$ is a good 
approximation of $(1-\cL^{\tau,2})^{-1}\phi$ (by~\eqref{e:AnsatzH-1}) and, on the other, 
it is close in $\fock{}{}{\tau}$ to the resolvent of the effective equation~\eqref{e:limitingSHE} 
applied to $\phi$. This latter point follows by~\eqref{e:AnsatzL2} and the fact that, as $\tau\to\infty$, 
the operator $\cG^\tau_M$ converges to $(1-\geneff)^{-1}$ as argued in~\cite[Eq. (3.20)]{CMT}. 

In other words, $\tilde s^\tau$ satisfies most of the main properties we desired an ansatz to have 
in Section~\ref{sec:Ideas}. At the same time, it lacks arguably the most important: its 
chaos expansion is infinite and its kernels 
do not have a compactly supported Fourier transform. Thus, the 
results of Section~\ref{s:gen_est_nonq} are not applicable, which 
in turn implies that we are unable to control the r.h.s.  of~\eqref{e:HeuH-1}. 

To remedy this, we introduce an {\it amended Ansatz} $s^\tau$ which is obtained from 
$\tilde s^\tau$ by truncating the support of the Fourier transform (and the chaos) of $\tilde s^\tau$ 
in such a way that the assumptions of Corollary~\ref{cor:HighDegreeEstimate} are satisfied. 
More precisely, let us define 
\begin{equ}[e:amended_ansatz]
	s^\tau \eqdef  \Pi_{\Yleft}^\tau \tilde s^\tau \, , \qquad \Pi_{\Yleft}^\tau \eqdef \1_{\{\cN\leq -\frac1{c}\log \nu_\tau,\,- 2 \gensy \leq \tau \,\exp(-\exp(c\cN))\} }
\end{equ}
where the r.h.s. in the definition of $\Pi_{\Yleft}^\tau$ 
is the diagonal operator whose Fourier kernel on the $n$-th chaos is 
$\1_{\{n\leq -\frac1{c}\log \nu_\tau,\,\abs[0]{\sqrt{R_\tau} p_{1:n}}^2 \leq \tau \exp(-\exp(cn) )\}}$, 
and $c>0$ is a parameter to be chosen later. 

The next proposition is the core of our analysis as it is the one that shows that indeed the amended ansatz 
tames both~\ref{i:I} and~\ref{i:II} and provides an approximation thanks to which 
we will be able to control the norms in~\eqref{e:gengen2} and~\eqref{e:gen2geneff}. 

\begin{proposition}\label{p:MAnsatz} 
Assume $c_2(F)\neq 0$. Let $n_0\in\N$, $\phi\in\core_{n_0}$ be such that 
$\mathrm{Supp}(\cF \phi) \subset \{- \cS^\tau \leq R\}$ for some $R>0$. 
Let $q>0$ be such that Proposition~\ref{p:Ansatz} holds. For $c\in(0,q/2)$, 
define $s^\tau$ according to~\eqref{e:amended_ansatz}. 
Then, $s^\tau\in\core$ for every $\tau>0$, and 
there exist $\tau_0>0$ and a constant $C=C(c_2(F), n_0, R, M, c, \tau_0)>0$ independent of $\tau$ 
such that, for every $\tau\geq \tau_0$ and 
 $\eps\in(0,1/8]$ we have 
\begin{equs}
\| s^\tau - \cG_M^\tau \phi\|_\tau &\leq C\,\nu_\tau^{3/4} \|\phi\|_\tau\,,\label{e:MAnsatzL2}\\
\|(1 - \cL^{\tau}) s^\tau - \phi\|_{\fock{}{-1}{\tau}} &\leq C \nu_\tau^{1/8 - \eps} \|\phi\|_\tau\,.\label{e:MAnsatzH-1}
\end{equs}
\end{proposition}

\begin{remark}\label{rem:OtherApprox}
Let us mention that the choice of $s^\tau$ is, clearly, {\it not unique}. In principle, 
any element $s^\tau\in\core$ such that both~\eqref{e:MAnsatzL2} and~\eqref{e:MAnsatzH-1} hold, would do. 
The difficulty lies precisely in the ability to {\it find at least one} such element and we exhibit an explicit one.  
\end{remark}

\begin{remark}\label{rem:Rates}
While the statements at the beginning of the section (see~\eqref{e:gengen2},~\eqref{e:gen2geneff}) 
were purely qualitative, the previous proposition is {\it quantitative} in that 
the bounds in~\eqref{e:MAnsatzL2} and~\eqref{e:MAnsatzH-1} provide {\it rates} at which 
the norms at the l.h.s. converge to $0$. That said, we do not expect them to be optimal 
(especially the second). 
\end{remark}

In the proof of the above proposition, we will need the following lemma 
that allows to control the terms arising from the commutator of the 
operators $\Pi^\tau_{\Yleft}$ in~\eqref{e:amended_ansatz} and 
$\genaFsh{\pm1}{\bullet}$, $\bullet\in\{\sharp,\flat\}$, in~\eqref{e:A+sharp},~\eqref{e:Asharp-1Aflats}. 

\begin{lemma}\label{l:Comm}
Let $c>0$. Then, there exists a constant $C=C(c)>0$ such that for any sufficiently regular $\psi\in\fock{}{}{\tau}$  
and $\beta\in(0,1/2)$ we have 
\begin{equ}[e:ComplPi]
\|\Pi^\tau_{\Yright} \genaFsh{\pm1}{\bullet} \psi\|_{\fock{}{1}{\tau}} \vee \|\genaFsh{\pm1}{\bullet}  \Pi^\tau_{\Yright} \psi\|_{\fock{}{1}{\tau}}  \leq C \, \|e^{\frac{c}{\beta}\cN}(- \nu_\tau \gensyx)^{\frac12}\psi\|_\tau^\beta \| \psi\|_{\fock{}{1}{\tau}}^{1-\beta} \, ,
\end{equ}
where $\Pi^\tau_{\Yright}\eqdef 1-\Pi^\tau_{\Yleft}$, $\Pi^\tau_{\Yleft}$ as in~\eqref{e:amended_ansatz}, 
and $\genaFsh{\pm1}{\bullet}$, for $\bullet\in\{\sharp,\flat\}$, are the operators in~\eqref{e:A+sharp},~\eqref{e:Asharp-1Aflats}. 
\end{lemma} 
\begin{proof}
Clearly, it suffices to consider $\psi\in\fock{n}{1}{\tau}$ for some $n\in\N$. 
To lighten the notations, set $\chi_n\eqdef \exp(-\exp(c\,n))$ and let 
$E_1^\tau(\cN)\eqdef \{\cN> -\frac1{c}\log \nu_\tau\}$ and 
$E_2^\tau(\cN)\eqdef \{- 2 \gensy > \tau \,\chi_n\}$, so that $\Pi^\tau_{\Yright} =\1_{E_1^\tau(\cN)\cup E_2^\tau(\cN)}$, 
where we recall that the Fourier representation of $\gensy$ in~\eqref{e:gensy} clearly depends 
on the order of the Fock space to which it is applied. 
We will split the argument in four steps. 
\medskip

\noindent{\it Step 1.} Notice that, for any $\beta,\gamma\in(0,1)$ with $\beta+\gamma=1$, 
we have $(1-\gensy)^{-1}\leq (-\nu_\tau\gensyx)^{-\beta}(-\gensy)^{-\gamma}$, so that 
we can write
\begin{equ}
\|\psi\|_{\fock{}{-1}{\tau}}=\|(1-\gensy)^{-1/2}\psi\|_\tau\leq \|(-\nu_\tau\gensyx)^{-\frac{\beta}{2}}(-\gensy)^{-\frac{\gamma}{2}}\psi\|_\tau\,.
\end{equ}
We now focus on the latter norm. As, trivially, 
$\1_{E_1^\tau(\cN)\cup E_2^\tau(\cN)}\leq \1_{E_1^\tau(\cN)}+\1_{E_2^\tau(\cN)}$
and, by construction, $(-\genaFsh{1}{\bullet})^{\ast}=\genaFsh{-1}{\bullet}$, the triangle inequality and  
a simple duality argument imply that if we show that, for $i=1,2$, there exists 
$f_i(\tau,n)>0$ such that for all $n\in\N$ and $\psi\in\core_n$  
\begin{equs}\label{e:Duality}
\|(- \nu_\tau \gensyx)^{-\frac{\beta}2} (- \gensy)^{-\frac{\gamma}{2}}\1_{E_i^\tau(\cN)} \genaFsh{1}{\bullet} \psi\|_{\tau}^2 \vee \|(- \nu_\tau &\gensyx)^{-\frac{\beta}{2}} (- \gensy)^{-\frac{\gamma}{2}} \genaFsh{1}{\bullet} \1_{E_i^\tau(\cN)} \psi\|_{\tau}^2 \\
&\lesssim f_i(\tau,n) \|(- \nu_\tau \gensyx)^{\frac{\beta}{2}} (- \gensy)^{\frac{\gamma}{2}} \psi\|_\tau^2 \, ,
\end{equs}
then the same bounds will hold with $\genaFsh{-1}{\bullet}$. 
\medskip

\noindent{\it Step 2.} To verify~\eqref{e:Duality}, let us begin with a preliminary computation 
of its l.h.s.. Setting $\chi^{\flat,\tau}\eqdef 1-\chi^{\sharp,\tau}$, 
and $\chi^{\bullet,\tau}$ to be either $\chi^{\sharp,\tau}$ or $\chi^{\flat,\tau}$, 
the definition of the operator $\genaFsh{1}{\bullet}$ in~\eqref{e:A+sharp} gives
\begin{equs}
\,&\|(- \nu_\tau \gensyx)^{-\frac{\beta}2} (- \gensy)^{-\frac{\gamma}{2}}\1_{E_i^\tau(\cN)} \genaFsh{1}{\bullet} \psi\|_{\tau}^2 \vee \|(- \nu_\tau \gensyx)^{-\frac{\beta}{2}} (- \gensy)^{-\frac{\gamma}{2}} \genaFsh{1}{\bullet} \1_{E_i^\tau(\cN)} \psi\|_{\tau}^2 \\
&=\max_{\theta\in\{0,1\}}\Big\{(n+1)!  \frac{\nu_\tau^{3/2}}{(2\pi)^2} \frac{4}{(n+1)^2}\int \Xi_{n+1}^\tau(\dd p_{1:n+1})(\1_{E_i^\tau(n+1)})^\theta\times\\
&\quad\times\Big| \sum_{1 \leq i < j \leq n+1}  
\frac{ [\fe_1 \cdot (p_i+p_j)]\Theta_\tau(p_i+p_j) \chi_{p_i, p_j}^{\bullet,\tau}(p_i+p_j, p_{1:n+1\setminus\{i,j\}}) }{|\sqrt{\nu_\tau} \fe_1\cdot(p_i+p_j, p_{1:n+1\setminus\{i,j\}})|^{\beta}|\sqrt{R^\tau}(p_i+p_j, p_{1:n+1\setminus\{i,j\}}))|^\gamma}\times\\
&\qquad\qquad\qquad\qquad\qquad\qquad\qquad\qquad \qquad\qquad\times (\1_{E_i^\tau(n)})^{1-\theta}\hat{\psi}(p_i+p_j, p_{1:n+1\setminus\{i,j\}})\Big|^2\Big\}\\
&\lesssim \max_{\theta\in\{0,1\}}\Big\{ n! n^3 \int \Xi_n^\tau(\dd p_{1:n}) \big(\nu_\tau |\fe_1 \cdot p_1|^2\big) \, |\hat{\psi}(p_{1:n})|^2\times \label{e:E21Prelim}\\
&\quad\times \nu_\tau^{1/2} \int_{\R^2} \dd q \, (\1_{E_i^\tau(n+1)})^\theta (\1_{E_i^\tau(n)})^{1-\theta}\frac{\Theta^1(\tau^{-1/2} \sqrt{R^\tau} q) \Theta^1(\tau^{-1/2} \sqrt{R^\tau} (p_1-q)) }{|\sqrt{\nu_\tau} \fe_1 \cdot (q,p_1-q,p_{2:n})|^{2\beta} |\sqrt{R^\tau} (q,p_1-q,p_{2:n})|^{2\gamma}}
\end{equs}
The next two steps describe how to estimate the r.h.s. of~\eqref{e:E21Prelim} 
in the cases $i=1$ or $2$. 
\medskip

\noindent{\it Step 3.} Note that, for $i=1$, then $E_1^\tau(n)\subset E_1^\tau(n+1)$, so that 
$(\1_{E_i^\tau(n+1)})^\theta (\1_{E_i^\tau(n)})^{1-\theta}\leq \1_{E_i^\tau(n)}$. 
Thus, the integral in $q$ is upper bounded by 
\begin{equs}
\1_{E_i^\tau(n)} \nu_\tau^{1/2} \int_{\R^2}  \frac{\dd q \,\Theta^1(\tau^{-1/2} \sqrt{R^\tau} q)  }{|\sqrt{\nu_\tau} \fe_1 \cdot q|^{2\beta} (1+|\sqrt{R^\tau} q|^{2})^{\gamma}} &=\1_{E_i^\tau(n)}\nu_\tau^{1/2} \int_{\R^2} \frac{\dd q \, \Theta^1(q)  }{| \fe_1 \cdot q|^{2\beta} (\tau^{-1}+|q|^{2})^{\gamma}}\lesssim \nu_\tau^{-1}
 \end{equs}
from which, since $\nu_\tau |\fe_1 \cdot p_{1:n}|^2\leq (\nu_\tau |\fe_1 \cdot p_{1:n}|^2)^{\beta}(1+|\sqrt{R^\tau}p_{1:n}|^2)^\gamma$, we ultimately deduce that~\eqref{e:Duality} holds 
with $f_1(\tau,n) \eqdef n^3  \1_{E_i^\tau(n)} \nu_\tau^{-\frac12}$\,.
\medskip

\noindent{\it Step 4.} For $i=2$ instead, we claim that~\eqref{e:E21Prelim} can be estimated 
from above by  
\begin{equs}
\,&\lesssim n^3 n! \int \Xi_n^\tau(\dd p_{1:n}) \, \big(\nu_\tau |\fe_1 \cdot p_1|^2\big) \, |\hat{\psi}(p_{1:n})|^2  \nu_\tau^{\frac12} \int \dd q \, \Theta^1(\tau^{-1/2} \sqrt{R^\tau} q) \Theta^1(\tau^{-1/2} \sqrt{R^\tau} (p_1-q)) \\
&\qquad\qquad\quad\times \frac{\1_{|\sqrt{R^\tau} q|^2 \gtrsim \tau\chi_n /n} + \1_{|\sqrt{R^\tau} (p_1-q)|^2 \gtrsim \tau\chi_n /n} + \1_{|\sqrt{R^\tau} q|^2 \leq |\sqrt{R^\tau} p_{1:n}|^2}}{|\sqrt{\nu_\tau} \fe_1 \cdot (q,p_1-q,p_{2:n})|^{2\beta} |\sqrt{R^\tau} (q,p_1-q,p_{2:n})|^{2\gamma}}\,.\label{e:E21}
\end{equs}
To see this, observe that if $\theta=1$, then $|\sqrt{R^\tau} (q,p_1-q,p_{2:n})|^2>\tau\chi_{n+1}$ which implies 
that 
\begin{equ}[e:MAx]
|\sqrt{R^\tau}q|^2\vee |\sqrt{R^\tau}(p_1-q)|^2\vee\max_{\ell\in\{2,\dots,n\}} |\sqrt{R^\tau}p_\ell|^2\geq \frac{\tau \chi_{n+1}}{n+1}\gtrsim \frac{\tau \chi_{n}}{n}
\end{equ}
where the last bound depends on $c$. 
On the other hand, if $\theta=0$, then $|\sqrt{R^\tau} p_{1:n}|^2>\tau\chi_{n}$ 
so that $\max_{\ell\in\{1,\dots,n\}} |\sqrt{R^\tau}p_\ell|^2\geq \tau\chi_{n}/n$. 
This means that either $\max_{\ell\in\{2,\dots,n\}} |\sqrt{R^\tau}p_\ell|^2\geq \tau\chi_{n}/n$ 
or $|\sqrt{R^\tau}p_1|^2\geq \tau\chi_{n}/n$, and in this latter case, by the triangle inequality, 
$|\sqrt{R^\tau}q|^2\vee |\sqrt{R^\tau}(p_1-q)|^2\gtrsim \tau\chi_{n}/n$. 
In other words, no matter whether $\theta=0$ or $1$,~\eqref{e:MAx} holds. 
Now, there are two scenarios: (1) $|\sqrt{R^\tau}q|^2\vee |\sqrt{R^\tau}(p_1-q)|^2\gtrsim \tau\chi_{n}/n$;  
(2) $|\sqrt{R^\tau}p_\ell|^2\gtrsim \tau\chi_{n}/n$ for some $\ell\in\{2,\dots,n\}$. If we are in 
the setting of (2), then either $|q|\wedge |p_1-q|\geq |p_\ell|$ so that we fall back into (1), 
or $|q|\vee |p_1-q|\leq |p_\ell|$ thus implying that 
$|\sqrt{R^\tau}q|^2\leq |\sqrt{R^\tau}p_\ell|^2\leq |\sqrt{R^\tau} p_{1:n}|^2$. 
In conclusion, we have one of this three possibilities which correspond 
to each of the indicator functions at the r.h.s. of~\eqref{e:E21}: 
$|\sqrt{R^\tau}q|^2\gtrsim \tau\chi_{n}/n$; $ |\sqrt{R^\tau}(p_1-q)|^2\gtrsim \tau\chi_{n}/n$; 
$|\sqrt{R^\tau}q|^2\leq |\sqrt{R^\tau} p_{1:n}|^2$. 

Getting back to~\eqref{e:E21}, we estimate the integral in $q$ separately for 
each summand corresponding to a different indicator function. The first two 
allow for the same bound (by symmetry) which is 
\begin{equs}
\nu_\tau^{\frac12}\int \dd q  &\frac{\1_{|\sqrt{R^\tau} q|^2 \gtrsim \tau\chi_{n}/n} \Theta^1(\tau^{-1/2} \sqrt{R^\tau} q) \Theta^1(\tau^{-1/2} \sqrt{R^\tau} (p_1-q))}{|\sqrt{\nu_\tau} \fe_1 \cdot (q,p_1-q,p_{2:n})|^{2\beta} (1+|\sqrt{R^\tau} (q,p_1-q,p_{2:n})|^2)^{\gamma}}\\
&\leq \nu_\tau^{\frac12}\int \dd q  \frac{\1_{|\sqrt{R^\tau} q|^2 \gtrsim \tau\chi_n /n} \Theta^1(\tau^{-1/2} \sqrt{R^\tau} q) }{|\sqrt{\nu_\tau} \fe_1 \cdot q|^{2\beta} |\sqrt{R^\tau} q|^{2\gamma}}= \int \dd q  \frac{\1_{|q|^2 \gtrsim \chi_n /n} \Theta^1( q) }{|\fe_1 \cdot q|^{2\beta} |q|^{2\gamma}}\\
&\lesssim \int_{\sqrt{\frac{\chi_n}n}}^\infty \frac{\dd r}{r(1+r^4)}\int_0^{2\pi}\frac{\dd \theta}{\cos(\theta)^{2\beta}}\lesssim \log (n/\chi_n)\lesssim e^{c \,n}\label{e:E22}
\end{equs}
where we used the decay properties of $\Theta^1$ to insert $(1+r^4)$ at the denominator 
and the second to last step holds provided $2\beta<1$. 
For the third instead, we note that $|\fe_1 \cdot (q,p_1-q,p_{2:n})|\gtrsim |\fe_1 \cdot p_{1:n}|$ 
and $|\sqrt{R^\tau} (q,p_1-q,p_{2:n})|\gtrsim |\sqrt{R^\tau} p_{1:n}|$, so that 
\begin{equs}
\nu_\tau^{\frac12}\int \dd q  &\frac{\1_{|\sqrt{R^\tau} q|^2 \leq |\sqrt{R^\tau} p_{1:n}|^2} \Theta^1(\tau^{-1/2} \sqrt{R^\tau} q) \Theta^1(\tau^{-1/2} \sqrt{R^\tau} (p_1-q))}{|\sqrt{\nu_\tau} \fe_1 \cdot (q,p_1-q,p_{2:n})|^{2\beta} (1+|\sqrt{R^\tau} (q,p_1-q,p_{2:n})|^2)^{\gamma}}\label{e:E23}\\
&\lesssim \frac{\nu_\tau^{\frac12}}{(\sqrt{\nu_\tau}|\fe_1 \cdot p_{1:n}|)^{2\beta}|\sqrt{R^\tau} p_{1:n}|^{2\gamma}}\int \dd q  \1_{|\sqrt{R^\tau} q|^2 \leq |\sqrt{R^\tau} p_{1:n}|^2}\lesssim \frac{|\sqrt{R^\tau} p_{1:n}|^{2\beta}}{(\sqrt{\nu_\tau}|\fe_1 \cdot p_{1:n}|)^{2\beta}}
\end{equs} 
since $\beta+\gamma=1$. 

Plugging~\eqref{e:E22} and~\eqref{e:E23} back into~\eqref{e:E21}, it is not hard to conclude that~\eqref{e:Duality} 
holds with $f_2(n,\tau)=\exp(c\,n)$
\medskip

\noindent{\it Step 5.} Let us now inspect the r.h.s. of~\eqref{e:Duality}. 
As we did in the last step of~\eqref{e:Gena3+Bound3}, we apply H\"older's inequality 
with exponents $1/\beta$ and $1/\gamma$ and obtain 
\begin{equs}
f_i(\tau,n)&\|(- \nu_\tau \gensyx)^{\frac{\beta}{2}} (- \gensy)^{\frac{\gamma}{2}} \psi\|_\tau^2\\
&=n! f_i(\tau,n)\int \Xi_n^\tau(\dd p_{1:n}) \big(\nu_\tau |\fe_1 \cdot p_{1:n}|^2\big)^\beta|\sqrt{R^\tau}p_{1:n}|^{2\gamma} \, |\hat{\psi}(p_{1:n})|^2\\
&\leq \|  f_i(\tau,\cN)^{\frac1{\beta}}e^{\frac{c}{\beta}\cN}(- \nu_\tau \gensyx)^{\frac12}\psi\|_\tau^\beta\|\psi\|_{\fock{}{1}{\tau}}^\gamma
\end{equs}
and now, for $i=2$, the definition of $f_2$ immediately gives the required upper bound in~\eqref{e:ComplPi}, 
while for $i=1$, we note that 
\begin{equ}
f_1(\tau,\cN)^{\frac1{\beta}}=n^{\frac3{\beta}}\1_{n>-\frac1{c}\log \nu_\tau}\nu_\tau^{-\frac1{2\beta}}\lesssim e^{\frac{c}{\beta}\,n} e^{-\frac{c}{\beta}\log \nu_\tau^{-\frac{1}{2c}}} \nu_\tau^{-\frac1{2\beta}}\lesssim e^{\frac{c}{\beta}\,n}
\end{equ}
and the same conclusion holds.  
\end{proof}
 
We are now ready to complete the proof of Proposition~\ref{p:MAnsatz}. 
 
\begin{proof}[of Proposition~\ref{p:MAnsatz}]
The double truncation in the definition of $s^\tau$ in~\eqref{e:amended_ansatz} immediately implies 
that, for every $\tau>0$ fixed, it belongs to $\core$. 

Concerning~\eqref{e:MAnsatzL2}, note that, as $\phi$ has compactly supported Fourier transform 
and $\cG^\tau_M$ is diagonal, for $\tau$ large enough, 
$ \Pi_{\Yleft}^\tau \cG_M^\tau \phi=\cG_M^\tau \phi$. 
Since $s^\tau$ is obtained via a hard Fourier cut-off of $\tilde s^\tau$ in~\eqref{e:AnsatzFormula} 
and, by removing 
the truncation, we make the $\fock{}{}{\tau}$-norm bigger, we get 
\begin{equs}
\| s^\tau - \cG_M^\tau \phi\|_\tau=\| \Pi_{\Yleft}^\tau(\tilde s- \cG_M^\tau \phi)\|_\tau\leq \|\tilde s-\cG^M_\tau\phi\|_\tau\leq C\nu_\tau^{3/4}\|\phi\|_\tau
\end{equs}
where the last step is a consequence of~\eqref{e:AnsatzL2}. 
\medskip

A similar argument also shows that the $\fock{}{1}{\tau}$ and the weighted anisotropic norms 
of $s^\tau$ can be controlled in terms of those of $\tilde s^\tau$, i.e. 
\begin{equs}
\|s^\tau\|_{\fock{}{1}{\tau}}&\leq \|\tilde s^\tau\|_{\fock{}{1}{\tau}}\leq C\|\phi\|_\tau\,,\label{e:MAnsatzH1}\\
\|e^{q\cN} (-\nu_\tau\gensyx)^{\frac12}s^\tau\|_\tau&\leq \|e^{q\cN} (-\nu_\tau\gensyx)^{\frac12}\tilde s^\tau\|_\tau\leq C \nu_\tau^{\frac14-\delta}\|\phi\|_\tau\,,\label{e:MAnsatzAH1}
\end{equs}
thanks to~\eqref{e:AnsatzH1} and~\eqref{e:AnsatzAH1}. 
\medskip

We now turn our attention to~\eqref{e:MAnsatzH-1}. Since $s^\tau\in\core$ and 
$\core$ is contained in the domain of $\gen$ by Lemma~\ref{l:CylDomain} we can compute 
$\gen s^\tau$ and, in view of Lemma~\ref{lem:ActionGen}, the action of $\gen$ can be explicitly written in terms 
of $\gensy$ in~\eqref{e:gensy} and $\genaF{m}{}$ in~\eqref{e:AmA}. More precisely, we have  
\begin{equs}[e:ExpH-1]
\|(1 - \gen) s^\tau - &\phi\|_{\fock{}{-1}{\tau}}=\|(1-\cL^{\tau,2})s^\tau-\phi-\genaF{\neq 2}{} s^\tau\|_{\fock{}{-1}{\tau}} \\
&=\|(1-\cL^{\tau,2})\tilde s^\tau-\phi +(1-\cL^{\tau,2})\Pi_{\Yright}^\tau\tilde s^\tau -\genaF{\neq 2}{} s^\tau\|_{\fock{}{-1}{\tau}}\\
&\leq \|(1-\cL^{\tau,2})\tilde s^\tau-\phi \|_{\fock{}{-1}{\tau}} +\|(1-\cL^{\tau,2})\Pi_{\Yright}^\tau\tilde s^\tau \|_{\fock{}{-1}{\tau}}+\|\genaF{\neq 2}{} s^\tau\|_{\fock{}{-1}{\tau}}
\end{equs}
where $\Pi_{\Yright}^\tau\eqdef 1-\Pi_{\Yleft}^\tau$ and $\Pi_{\Yleft}^\tau$ is the operator in~\eqref{e:amended_ansatz}. 
The first term is controlled by~\eqref{e:AnsatzH-1}, while, for the last, 
we can use Corollary~\ref{cor:HighDegreeEstimate} as $s^\tau$ satisfies all of 
its assumptions. In particular,~\eqref{e:BoundAne2} 
implies that, for every $\alpha\in(0,1)$, 
\begin{equs}
\|\genaF{\neq 2}{} s^\tau\|_{\fock{}{-1}{\tau}}\lesssim \|\cN^{\frac{3}{\alpha}}(- \nu_\tau \gensyx)^{\frac12}s^\tau\|_\tau^\alpha \,\|  s^\tau\|_{\fock{}{1}{\tau}}^{1-\alpha}\lesssim \nu_\tau^{\alpha\big(\frac14-\delta\big)}\|\phi\|_\tau
\end{equs}
where we used~\eqref{e:MAnsatzH1} and~\eqref{e:MAnsatzAH1}. 
It remains to look at the second summand, which is bounded by 
\begin{equs}[e:CommFinal1]
\|(1-\cL^{\tau,2})\Pi_{\Yright}^\tau\tilde s^\tau \|_{\fock{}{-1}{\tau}}&\leq \|\Pi_{\Yright}^\tau\tilde s^\tau\|_{\fock{}{1}{\tau}}+c_2(F)\|\genaF{2}{}\Pi_{\Yright}^\tau\tilde s^\tau \|_{\fock{}{-1}{\tau}}\\
&\lesssim \|\Pi_{\Yright}^\tau\tilde s^\tau\|_{\fock{}{1}{\tau}} +\sum_{i=\pm 1,\,\bullet\in\{\sharp,\flat\}}\|\genaFsh{i}{\bullet}\Pi_{\Yright}^\tau\tilde s^\tau \|_{\fock{}{-1}{\tau}}\,.
\end{equs}
A priori, it is not immediately clear how to control the $\fock{}{1}{\tau}$-norm of $\Pi_{\Yright}^\tau\tilde s^\tau$. 
Indeed, if we simply neglect the cut-off, then what we obtain is of order 
one by~\eqref{e:AnsatzH1}. But, by~\cite[Lem. 4.2]{CMT} (adapted to our notations, see Remark~\ref{rem:Notations}), 
we know that 
\begin{equ}
\tilde s^\tau=c_2(F)\cG^\tau_M\genaFsh{1}{\sharp}\tilde s^\tau +\cG^\tau_M\phi\,.
\end{equ}
Applying $\Pi_{\Yright}^\tau$ at both sides, the summand containing $\phi$ vanishes 
for $\tau$ large enough, thus we deduce 
\begin{equ}
\|\Pi_{\Yright}^\tau\tilde s^\tau\|_{\fock{}{1}{\tau}}\lesssim \|\Pi_{\Yright}^\tau\cG^\tau_M\genaFsh{1}{\sharp}\tilde s^\tau\|_{\fock{}{1}{\tau}}= \|\cG^\tau_M\Pi_{\Yright}^\tau\genaFsh{1}{\sharp}\tilde s^\tau\|_{\fock{}{1}{\tau}}\leq  \|\Pi_{\Yright}^\tau\genaFsh{1}{\sharp}\tilde s^\tau\|_{\fock{}{-1}{\tau}}
\end{equ}
where we used that $\Pi_{\Yright}^\tau$ and $\cG^\tau_M$ commute, and that, by its definition, 
$\cG^\tau_M\leq (1-\gensy)^{-1}$. 
Plugging the above estimate into~\eqref{e:CommFinal1}, we reach a position in which we can 
apply Lemma~\ref{l:Comm}. Its statement gives  
that, for every $\beta\in(0,1/2)$, we have  
\begin{equs}
\|(1-\cL^{\tau,2})\Pi_{\Yright}^\tau\tilde s^\tau \|_{\fock{}{-1}{\tau}}&\lesssim \max_{i=\pm1, \bullet\in\{\sharp,\flat\}}\big\{\|\genaFsh{i}{\bullet}\Pi_{\Yright}^\tau\tilde s^\tau \|_{\fock{}{-1}{\tau}}, \|\Pi_{\Yright}^\tau\genaFsh{i}{\bullet}\tilde s^\tau \|_{\fock{}{-1}{\tau}}\big\}\\
&\lesssim \|e^{\frac{c}{\beta}\cN}(- \nu_\tau \gensyx)^{\frac12}\tilde s^\tau\|_\tau^\beta \,\| \tilde s^\tau\|_{\fock{}{1}{\tau}}^{1-\beta} \lesssim  \nu_\tau^{\beta\big(\frac14-\delta\big)}\|\phi\|_\tau
\end{equs}
where in the last step we exploited~\eqref{e:AnsatzH1},~\eqref{e:AnsatzAH1}, which is allowed 
since we have taken $c<q/2$. 

Collecting all the observations made so far, choosing $\beta$ sufficiently close to $1/2$ 
and defining $\eps$ accordingly,~\eqref{e:MAnsatzH-1} 
follows and the proof of the proposition is thus complete. 
\end{proof}

\subsection{Resolvent Convergence and proof of Theorems~\ref{th:CLT} and~\ref{th:conv_semigroup}} 
\label{s:pf_summary}

In this subsection, we complete the proof of the main results of the paper, Theorems~\ref{th:CLT} and~\ref{th:conv_semigroup}. 
Before turning to it, let us highlight in the next theorem the main consequence of the analysis carried out so far 
(from which both main results will follow as a corollary), namely the 
convergence of $\gen$ to $\geneff$, the generator of~\eqref{e:limitingSHE}, 
in the resolvent sense. Its proof makes rigorous sense of the analysis sketched in Subsection~\ref{sec:Ideas} and heavily relies 
on Proposition~\ref{p:MAnsatz} and the definition of the operator $\cG^\tau_M$ in~\eqref{e:cgG}. 

\begin{theorem}\label{thm:ResConv}
For every $F\in\L^2(\P)$, we have
\begin{equ}[e:StrongCvResL2]
\iota^\tau (1 - \gen)^{-1} j^\tau F \xrightarrow[\tau \to \infty]{\L^2(\P)} (1 - \cL^{\mathrm{eff}})^{-1} F\,,
\end{equ}
where $\iota^\tau$ and $j^\tau$ are the maps in Definition~\ref{def:iotatau} and 
$\geneff$ is the generator of~\eqref{e:limitingSHE}, i.e. the diagonal operator 
whose action on $\fock{}{}{}$ coincides with that of $D^\eff_{1,1}\gensyx+\gensyy$ 
for $D^\eff_{1,1}$ the first entry of the matrix $D^\eff$ in~\eqref{e:Dlim}.  
\end{theorem}
\begin{proof}
As noted in the proof of~\cite[Thm. 5.1]{CMT}, it suffices to check~\eqref{e:StrongCvResL2} 
on a dense subset of $\L^2(\P)$, e.g. that which consists of elements of the form $I_{n_0}(\phi)$ 
for $n_0\in\N$ and $\phi$ smooth and with compactly supported Fourier transform. 
For such an observable, let $s^\tau$ be our amended ansatz as given in~\eqref{e:amended_ansatz} 
and note that 
\begin{equs}
\|&\iota^\tau (1 - \gen)^{-1} j^\tau \phi- (1 - \geneff)^{-1} \phi\|\\
&\leq \|\iota^\tau (1 - \gen)^{-1} j^\tau \phi- \iota^\tau s^\tau j^\tau \|+ \|\iota^\tau s^\tau j^\tau - \iota^\tau \cG^\tau_M j^\tau \phi \| +\|\iota^\tau \cG^\tau_M j^\tau \phi- (1 - \geneff)^{-1} \phi\|\\
&=\|(1 - \gen)^{-1} \phi-  s^\tau  \|_\tau+ \| s^\tau  - \cG^\tau_M \phi \|_\tau +\|\cG^\tau_M \phi- (1 - \geneff)^{-1} \phi\|\\
&\lesssim \|\phi-  (1 - \gen)s^\tau  \|_{\fock{}{-1}{\tau}}+ \| s^\tau  - \cG^\tau_M \phi \|_\tau +\| \cG^\tau_M  \phi- (1 - \geneff)^{-1} \phi\|
\end{equs}
where we denoted by $\|\cdot\|$ the norm on $\fock{}{}{}$ (without the $\tau$), we used~\cite[Lem. 1.7]{CMT} 
to pass from $\fock{}{}{}$ to $\fock{}{}{\tau}$ and to replace $\iota^\tau \cG^\tau_M j^\tau$ with $\cG^\tau_M$ 
(this is a slight abuse of notation, but the two operators have the same kernels) 
and at last we applied Lemma~\ref{l:ItoTrick} 
to estimate the $\fock{}{}{\tau}$-norm with the $\fock{}{-1}{\tau}$-norm.
Now, thanks to Proposition~\ref{p:MAnsatz}, and in particular~\eqref{e:MAnsatzL2} and~\eqref{e:MAnsatzH-1}, 
the first two summands go to $0$. For the last instead, we exploit the observation 
made in~\cite[Eq. (3.20)]{CMT} according to which the Fourier 
transform of the kernel of $\cg^\tau_M$ converges pointwise in the limit 
for $\tau, M\to\infty$ to that of $D^\eff_{1,1}\gensyx$. As a consequence, the 
kernel of $\cG^\tau_M$ converges in the same sense 
to that of $(1-D^\eff_{1,1}\gensyx-\gensyy)^{-1}=(1-\geneff)^{-1}$. Since $\phi$ 
is smooth, by the dominated convergence theorem, the conclusion follows at once.
\end{proof}
 
We are now ready to present the proof of Theorems~\ref{th:CLT} and~\ref{th:conv_semigroup}.

\begin{proof}[of Theorems~\ref{th:CLT} and~\ref{th:conv_semigroup}] 
In light of Theorem~\ref{thm:ResConv}, the proof follows the exact same steps  as that 
of Theorems 1.3 and 1.4 detailed in~\cite[Sec. 5]{CMT}. The proof therein uses in 
no special way the specific structure of the equation or of its generator 
but only the analog of Theorem~\ref{thm:ResConv} and properties of the limiting 
equation. Since our limiting equation is the same, that proof applies here {\it mutatis mutandis}. 

To witness, an immediate consequence of Theorem~\ref{thm:ResConv} and 
the Trotter--Kato theorem~\cite[Thm. 4.2]{IK} is the convergence 
of the semigroup $P^\tau$ to that of the effective equation, $P^\eff$, 
in $\L^2(\P)$ uniformly on compact time intervals. The extension to 
infinite time horizons and to $\L^\theta(\P)$, for $\theta\in[1,\infty)$ which 
would imply~\eqref{e:semigroups_L_theta}, 
is given in~\cite[{\it Part I}, pg. 45]{CMT}. It only relies on the contractivity of 
$P^\tau$ (which we have by Lemma~\ref{lem:Contractivity}) and of $P^\eff$, 
and on mixing properties of the latter (see~\cite[Lem. 2.11]{CMT}), 
all of which hold in the present context as well. 

The argument in~\cite[{\it Part II}, pg. 46]{CMT} 
ensures that~\eqref{e:semigroups_L_theta} implies that if 
for $\theta \in [1, +\infty)$ and $k\geq 1$  
$(F_1^\tau)_{\tau > 0}, ..., (F_k^\tau)_{\tau > 0}$ converge to $F_1, ..., F_k$ in $\mathbb L^\theta(\mathbb P)$ 
then for any times $0 \leq t_1 \leq \dots \leq t_k$, 
 \begin{equ} \label{e:CLTold}
 		\int  \, \abs[2]{\BE_{u_0}^\tau\big[F^\tau_1\big(u^{\tau}_{t_1}) \dots F^\tau_k\big(u^{\tau}_{t_k})\big)\big] - \BE_{u_0}^{\eff}\big[F_1\big(u^{\eff}_{t_1}) \dots F_k\big(u^{\eff}_{t_k})\big)\big]}^\theta \mathbb P(\dd u_0) \to 0 \, . 
 \end{equ}	
Hence, upon taking $F^\tau_i(\eta^\tau) \eqdef f_i(\eta^\tau(\phi_i))$, for $i=1,\dots,k$, 
$\phi_1,\dots,\phi_k\in\cS(\R^2)$ 
and $f_1,\dots,f_k$ bounded continuous functions on $\R$, 
we immediately deduce Theorem~\ref{th:CLT} for~$\mathbb Q = \mathbb P$;
the extension to $\mathbb Q \ll \mathbb P$ is then obvious.

But in~\eqref{e:CLTold}, there is no need to take the $f_i$'s to be {\it bounded}, thus we also 
deduce that~\eqref{e:CLT} holds for $f$ polynomial. To move to a generic $f$ 
continuous and of at most polynomial growth, 
we proceed by approximation. Indeed, for such $f$, 
there exists a sequence of polynomials $(P_n)_{n \geq 1}$ such that for every $c > 0$ it holds that 
\begin{equ}
\sup_{x_{1:k}\in\R^k} \frac{|f(x_{1:k}) - P_n(x_{1:k})|}{\sum_{i=1}^k e^{c x_i^2} } \to 0 \,, \qquad \text{as $n \to \infty$} .
\end{equ}
But then, by Jensen's inequality and stationarity, we get 
   \begin{equs}
    \int \mathbb P(\dd u_0) \, &\Big|\BE_{u_0}^\tau\big[(f - P_n)\big(u^{\tau}_{t_1}(\phi_1), \dots, u^{\tau}_{t_k}(\phi_k)\big)\big] - \BE_{u_0}^{\eff}\big[(f-P_n)\big(u^{\eff}_{t_1}(\phi_1), \dots, u^{\eff}_{t_k}(\phi_k)\big)\big]\Big|^p \\
    &\leq \BE^\tau\big[|f - P_n|^p\big(u^{\tau}_{t_1}(\phi_1), \dots, u^{\tau}_{t_k}(\phi_k\big)\big] + \BE^{\eff}\big[|f-P_n|^p\big(u^{\eff}_{t_1}(\phi_1), \dots, u^{\eff}_{t_k}(\phi_k)\big)\big] \\
    &\lesssim \sup_{x_{1:k}\in\R^k} \frac{|f(x_{1:k}) - P_n(x_{1:k})|}{\sum_{i=1}^k e^{c x_i^2} } \sum_{i=1}^k \Big(\E[e^{p c \langle \eta^\tau, \phi_i\rangle^2}] + \E[e^{p c \langle \eta, \phi_i\rangle^2}] \Big)\,.
    \end{equs}
Now, the sum is finite uniformly in $\tau$ and $n$ provided 
$c < \frac{1}{2 p \sigma^2}$ with 
$\sigma = \max\{\|\phi_i\|_{L^2(\R^2)}\colon i=1,\dots,n\}$ since $\langle \eta^\tau, \phi_i\rangle$ and 
$\langle \eta, \phi_i\rangle$ are centred Gaussian random variables of variance 
$\|\phi_i\|_{L_\tau^2(\R^2)} (\leq \|\phi_i\|_{L^2(\R^2)})$ and 
$\|\phi_i\|_{L^2(\R^2)}$, 
and the prefactor converges to $0$ as $n\to\infty$ from which~\eqref{e:CLT} follows at once.  
\end{proof}

\begin{appendix}

\section{Regularity, approximation, and Hermite coefficients decay}\label{a:DecayCoeff} 

The goal of this appendix is to provide some insight over the relation 
between the regularity properties of a function $F\colon \R\to\R$ 
and its coefficients in the Hermite polynomials expansion, so 
to have a better understanding of Assumption~\ref{Assumption:F}. 
What follows is probably classical but since we could not 
find any good reference, we present it here.

Recall that $\pi_1$ denotes the law of a real valued centred Gaussian random variable with unit variance, 
and $L^2(\pi_1)$ is the space of square-integrable functions with respect to~$\pi_1$. 
For $m\in\N$, denote by $H_m$ the $m$-th Hermite polynomial which is defined by 
\begin{equ}[e:HermitePol]
H_m(x)\eqdef \frac{(-1)^m}{m!} e^{\frac{x^2}{2}}\frac{\dd^m}{\dd x^m}\big(e^{-\frac{x^2}{2}}\big)\,,\qquad x\in\R\,.
\end{equ}
The family $(\sqrt{m!} \, H_m)_{m \geq 0}$ forms an orthonormal basis of $L^2(\pi_1)$, 
so that any $F\in L^2(\pi_1)$ can be written in terms of the Hermite coefficients $(c_m(F))_{m\geq 0}$ as 
\begin{equ}[e:HermiteCoeff]
F = \sum_{m \geq 0} c_m(F) \, H_m \, , \qquad c_m(F) \eqdef m! \int_{\R} F(x) \, H_m(x) \, \frac{e^{-x^2/2}}{\sqrt{2 \pi}} \, \dd x \, ,
\end{equ}
and its $L^2(\pi_1)$-norm satisfies 
\begin{equ}[e:L2NormGaussian]
\|F\|_{L^2(\pi_1)}^2 = \sum_{m \geq 0} \frac{c_m(F)^2}{m!} \, .
\end{equ}
For convenience, let us introduce the {\it normalised} Hermite coeffients of $F$, i.e. 
the coefficients that corresponds to the coordinates of $F$ in the orthonormal basis 
$(\sqrt{m!} \, H_m)_{m \geq 0}$, as 
\begin{equ}[e:NormCoeff]
\hat{c}_m(F) \eqdef \frac{c_m(F)}{\sqrt{m!}} = \langle F, \sqrt{m!}H_m\rangle_{L^2(\pi_1)}\,.
\end{equ}
At first, we will establish the relation 
between the decay in $m$ of $\hat c_m(F)$ and the $L^2(\pi_1)$-norms 
of the derivatives of $F$. The following lemma provides 
an upper bound of the former in terms of the latter. 

\begin{lemma}\label{l:HerAppDecay}
Let $F\in\CC^k(\R)$ be such that $F^{(\ell)}\in L^2(\pi_1)$ for all $\ell\leq k$. Then, 
for every $\ell\leq k$ and $m\geq \ell$, we have 
\begin{equ}[e:HerAppDecay]
|\hat c_m(F)|\leq \frac{\|F^{(\ell)}\|_{L^2(\pi_1)}}{\sqrt{(m-\ell+1) \dots m}}\,.
\end{equ} 
\end{lemma}

\begin{proof}
The result is a direct consequence of basic properties of Hermite polynomials. 
In particular, the key identity is~\cite[Eq. (1.2)]{Nualart} which states that for every $n\geq 0$, $H_{n+1}' \equiv H_{n}$.
This, together with a density argument, implies that for any $\ell \leq k$, 
\begin{equ}[e:CoeffDer]
\hat{c}_m(F^{(\ell)}) = \sqrt{(m+1) \dots (m+\ell) } \, \hat{c}_{m+\ell}(F)\,,
\end{equ} 
and thus, by~\eqref{e:L2NormGaussian} we get 
\begin{equ}[e:LinkRegDecay]
\|F^{(\ell)}\|_{L^2(\pi_1)}^2 = \sum_{m \geq \ell} (m-\ell+1) \dots m \, \hat{c}_m(F)^2
\end{equ}
which, since the l.h.s. is finite by assumption, immediately gives~\eqref{e:HerAppDecay}. 
\end{proof}

\begin{remark}
Compare \eqref{e:CoeffDer} to the classical analogue for the Fourier transform, say for $F \in \cS(\R)$,
\begin{equ}
\cF(F^{(\ell)}) \, (m) = (\iota m)^\ell \ \cF(F) \, (m) \, .
\end{equ}
Here the square root of the ascending $\ell$-power of $m$ has been replaced by $(\iota m)^\ell$. 
This suggests that, as for the Fourier transform, the regularity of 
$F$ is tightly connected to the decay of $(c_m(F))_{m \geq 0}$.
\end{remark}

The next proposition addresses the main question of the appendix as it states a necessary 
and sufficient condition for the normalised Hermite coefficients of a function 
to decay faster than any polynomial.

\begin{proposition} \label{p:cm(F)asymptotic}
Let $F \in L^2(\pi_1)$. Then, the sequence $(\hat{c}_m(F))_{m \geq 0}$ decays  
faster than any inverse power of $m$ if and only if the derivatives of every order of $F$ belong to $L^2(\pi_1)$.
\end{proposition}

\begin{remark}
    In view of the analogy with the Fourier transform as explained above, this is the exact analogue of the classical fact that a (say periodic, to avoid integrability issues) function is smooth if and only if its Fourier transform decays faster than any inverse powers.
\end{remark}

\begin{proof}
``$\Longleftarrow$'' is a straightforward consequence of Lemma~\ref{l:HerAppDecay}. 

\noindent ``$\Longrightarrow$'' Note that~\eqref{e:LinkRegDecay}
gives that 
for any fixed $\delta > 0$ and every $k \geq 0$, we have
\begin{equs}
\|F^{(k)}\|_{L^2(\pi_1)} &\leq \Big(\sum_{m \geq k} \frac{1}{1 + m^{1+2\delta}}\Big)^{1/2} \sup_{m \geq k} (1 + m^{k/2+1/2+\delta}) \, |\hat{c}_m(F)| \\
&\lesssim_\delta \sup_{m \geq k} (1+m^{k/2+1/2+\delta}) \, |\hat{c}_m(F)| \, , \label{e:FromDecayToReg}
\end{equs}
from which the implication, and thus the statement, follow at once. 
Strictly speaking, to apply~\eqref{e:LinkRegDecay} we would need to a priori know that the derivatives of $F$ belong to $L^2(\pi_1)$; to avoid this difficulty, one can apply~\eqref{e:LinkRegDecay} to the Hermite expansion $(\sum_{m=0}^N c_m(F) H_m)_{N \geq 0}$ to show that all its derivatives are Cauchy in $L^2(\pi_1)$: it is then easy to see that their $L^2(\pi_1)$-limit must coincide with the derivatives of $F$ which concludes the proof.
\end{proof}

Let us now see how Assumption~\ref{Assumption:F} fits into the previous statement. 
A first immediate result is that the normalised Hermite coefficients of $F$ decay faster than any inverse power, but we can obtain a more quantitative control on this fast decay. 

\begin{proposition}\label{p:AssF}
Let $F\colon \R\to\R$ be such that Assumption~\ref{Assumption:F} holds. Then, 
there exists a constant $C>0$ such that 
for every $n\in\N$
\begin{equ}[e:Decay]
\| F^{(n)}\|_{L^2(\pi_1)}\leq C\,\exp(\exp(o(n)))\,.
\end{equ}
As a consequence, the normalised Hermite coefficients of $F$, $(\hat c_m(F))_{m\geq 0}$ in~\eqref{e:NormCoeff}, 
decay faster than any inverse power of $m$, and, for every $\gamma>0$ there exists a constant $C=C(\gamma)$ 
such that for every $k\geq 0$ it holds 
\begin{equ} \label{e:expexpMomentGrowth}
		\sum_{m \geq 1}  m^k \, |\hat{c}_m(F)| \leq C \exp(\exp(\gamma k)) \, .
	\end{equ}
\end{proposition}

\begin{proof}
Note that if $G$ is an arbitrary continuous function on $\R$, then we can bound (using that $\kappa < 1/4$)
\begin{equ}[e:NormComp]
\|G\|_{L^2(\pi_1)} \leq \Big(\int e^{-(1/2 - 2 \kappa) \, x^2} \, \dd x\Big)^{1/2} \sup_{x \in \R} |G(x)| e^{-\kappa x^2} \lesssim_\kappa \sup_{x \in \R} |G(x)| e^{-\kappa x^2} \, .
\end{equ}
Therefore, we immediately deduce \eqref{e:Decay} from Assumption \ref{Assumption:F}. Then, Proposition~\ref{p:cm(F)asymptotic} implies the fast decay
of the normalised Hermite coefficients of $F$. 

At last, we turn to~\eqref{e:expexpMomentGrowth}.
Using Cauchy--Schwartz, the trivial bound $m^k \leq (2k)^k + 2^k \, m \dots (m-k+1)$, 
and the identity~\eqref{e:CoeffDer}, we obtain
	\begin{equs}
		\sum_{m \geq 1} m^k \, |\hat{c}_m(F)| &\lesssim \Big(\sum_{m \geq 1} m^{k+2} \, \hat{c}_m(F)^2\Big)^{1/2} \\
		&\lesssim \Big((2k+4)^{k+2} \sum_{m \geq 1} \hat{c}_m(F)^2 + 2^{k+2} \sum_{m \geq 1} m \dots (m - k - 1) \, \hat{c}_m(F)^2\Big)^{1/2} \, \\
		&=\Big((2k+4)^{k+2} \|F\|_{L^2(\pi_1)}^2 + 2^{k+2} \|F^{(k+2)}\|_{L^2(\pi_1)}^2\big)^{1/2} \, , \\
		&\lesssim_\gamma \exp(\exp(\gamma k))\, ,
\end{equs}
where we applied~\eqref{e:Decay} thanks to which $\gamma>0$ can be chosen arbitrary small. 
\end{proof}

In the next lemma, we show that functions for which Assumption~\ref{Assumption:F} holds 
can be approximated arbitrarily well by suitably chosen polynomials. The statement is a bit  
more general as in Section~\ref{sec:WellPosed} our assumption on $F$ can be relaxed. 

\begin{lemma} \label{l:DensityPolynomial}
Let $\kappa\in(0,1/4)$, $k\in\N\cup\{+\infty\}$ and $F\in\CC^k(\R)$ be such that, 
for every $\ell\leq k$, $\sup_x |F^{(\ell)}(x)| e^{-\kappa x^2}<\infty$ and $F^{(\ell)}(x)=o(C e^{\kappa x^2})$ for $|x|$ large, 
where $C=C(\ell)>0$ is a positive constant. 
Then, there exists a sequence of polynomials $(P_M)_{M \geq 1}$ such that for every $\ell\leq k$, 
\begin{equ}[e:PolyApprox]
\lim_{M\to\infty}\sup_{x \in \R} |F^{(\ell)}(x) - P^{(\ell)}_M(x)| e^{- \kappa x^2} = 0 \, .
\end{equ}
\end{lemma}
\begin{remark}\label{rem:LemmaRed}
Upon choosing $k=1$ in the above statement, its assumptions on $F$ coincide with~\eqref{e:AssF}, while 
for $k=\infty$ we are in the setting of Assumption~\ref{Assumption:F}.
\end{remark}
\begin{proof}
We first claim that the statement holds provided it does for $F$ smooth and compactly supported. 
Let $(\vartheta_N)_{N \geq 1}$ be a sequence of mollifiers, and 
$(\chi_N)_{N \geq 1}$ a sequence of smooth cut-off functions that converges locally to $1$ and 
have uniformly bounded $\cC^k(\R)$-norm. 
Then, for $N \geq 1$, define $F_N = \vartheta_N * [\chi_N F]$. 
Since $F\in\cC^k$ and, for every $\ell\leq k$, $F^{(\ell)}(x)=o(C e^{\kappa x^2})$, it is easy to check that 
\begin{equ}
\sup_{x \in \R} |F^{(\ell)}(x) - F_N^{(\ell)}(x)| e^{ -\kappa x^2} \xrightarrow[N \to \infty]{} 0 \, ,
\end{equ}
from which the claim follows. 

We now consider $F$ smooth and compactly supported. 
Recall that for any $\sigma>0$, the family $(H_m^\sigma)_{m\geq 0}$, 
where $H_m^\sigma(\cdot)\eqdef \sqrt{m!} H_m(\cdot/\sigma)$ and $H_m$ is the 
$m$-th Hermite polynomial on $\R$, forms an orthonormal basis of $L^2(\pi_\sigma)$, for $\pi_\sigma$ the 
law of a real-valued centred Gaussian random variable with variance $\sigma^2$. 
Denoting by $P_M^\sigma(F)$ the polynomial obtained by projecting $F$ onto 
the space generated by $(H_m^\sigma)_{m\leq M}$, i.e. 
\begin{equ}[e:PolApprox]
P^\sigma_M(F) = \sum_{m=0}^M \langle F, H_m^\sigma \rangle_{L^2(\pi_\sigma)} H_m^\sigma
\end{equ}
we have $\|F-P_M^\sigma(F)\|_{L^2(\pi_\sigma)}\to 0$ as 
$M\to\infty$. Applying~\cite[Eq.'s (1.2), (1.3)]{Nualart} and the scaling relation above, it is 
not hard to see that, for any $k\geq 1$, $[P^\sigma_M(F)]^{(\ell)} = P^\sigma_{M-k}(F^{(\ell)})$, 
so that the previous argument shows that for every $k\geq 1$, 
$\|F^{(\ell)}-P_M^\sigma(F)^{(\ell)}\|_{L^2(\pi_\sigma)}\to 0$ 
as $M\to\infty$. As a consequence,~\eqref{e:PolyApprox} follows provided 
we are able to bound the (semi-)norm at the l.h.s.  by the $L^2(\pi_\sigma)$-norm, 
which is indeed possible upon choosing $\sigma$ appropriately. In formulas, 
we claim that for a function $G$ on $\R$ such that $G,G'\in L^2(\pi_{\sigma})$, we have 
\begin{equ}[e:ClaimFinalPoly]
\sup_{x\in\R}|G(x)| e^{-\frac{x^2}{2\sigma^2}}\lesssim \|G\|_{L^2(\pi_{\sigma})}+\|G'\|_{L^2(\pi_{\sigma})}
\end{equ}
from which we deduce~\eqref{e:PolyApprox} by taking $\sigma^2=\frac{1}{2\kappa}>1$ 
and $G=F^{(\ell)}-P_M^{1/\sqrt{2\kappa}}(F)^{(\ell)}$. 

It remains to prove~\eqref{e:ClaimFinalPoly}. By Cauchy-Schwartz inequality, we 
easily see that, for any $x,y\in\R$, we have 
\begin{equs}
|G(y)-G(x)|&\leq \int_{x\wedge y}^{x\vee y}|G'(z)|\dd z\lesssim \|G'\|_{L^2(\pi_\sigma)}\Big(\int_{x\wedge y}^{x\vee y}e^{\frac{z^2}{2\sigma^2}}\dd z\Big)^{\frac12}\\
&\lesssim \|G'\|_{L^2(\pi_\sigma)} \sqrt{|x-y|}\Big(e^{\frac{x^2}{2\sigma^2}}+e^{\frac{y^2}{2\sigma^2}}\Big)^{\frac12}\lesssim \|G'\|_{L^2(\pi_\sigma)} \Big(e^{\frac{x^2}{3\sigma^2}}+e^{\frac{y^2}{3\sigma^2}}\Big)\,.
\end{equs}
and thus
\begin{equs}
|G(x)|&\leq \Big|G(x)-\int_\R G(y) \frac{1}{\sqrt{2\pi}\sigma} e^{-\frac{y^2}{2\sigma^2}}\dd y\Big| + \int_\R |G(y)| \frac{1}{\sqrt{2\pi}\sigma} e^{-\frac{y^2}{2\sigma^2}}\dd y\\
&\leq \int_\R |G(x)-G(y)| \frac{1}{\sqrt{2\pi}\sigma} e^{-\frac{y^2}{2\sigma^2}}\dd y +\|G\|_{L^2(\pi_\sigma)}\\
&\lesssim \|G'\|_{L^2(\pi_\sigma)} e^{\frac{x^2}{2\sigma^2}}+\|G\|_{L^2(\pi_\sigma)}
\end{equs}
from which~\eqref{e:ClaimFinalPoly} follows at once.
\end{proof}

We conclude this section with a few concrete examples of functions that satisfy or not the Assumption~\ref{Assumption:F}.

\begin{lemma} \label{lem:Analytic}
Let $F \colon \R \to \R$ admit a complex analytic extension on the strip $\{z \colon |\Im(z)| < \delta\}$ for some $\delta > 0$, such that the sup-norm of $F$ on the rectangle $Q_{M,\delta} := [-M, M] \times [-\delta, \delta]$ is dominated by $\exp(\kappa M^2)$ for some $\kappa \in (0,1/4)$. Then, for any $\kappa' \in (\kappa, 1/4)$, 
there exists a constant $C \in (0, \infty)$ such that for every $n \geq 0$
\begin{equ} \label{e:BoundDerivativeAnalytic}
|F^{(n)}(x)| \leq C \, \delta^{-n} \, n! \, \exp(\kappa x^2) \, .
\end{equ}
In particular, any such $F$ satisfies Assumption \ref{Assumption:F}. 
\end{lemma}

\begin{corollary} \label{coro:examples_F}
The functions $F(x) = \sqrt{a+x^2}$ for $a > 0$ and $G(x) = \Re \exp(\omega x)$ for $\omega \in \C$ satisfy Assumption \ref{Assumption:F}.
\end{corollary}

\begin{remark}
Note that the growth of the derivates in \eqref{e:BoundDerivativeAnalytic} is much smaller 
than what is required by Assumption \ref{Assumption:F}. Indeed, it can be shown that it implies the existence 
of constants $c, C \in (0, \infty)$ such that for every $m \geq 0$
\begin{equ} \label{e:DecayHermiteAnalytic}
|\hat c_m(F)| \leq C \exp(- c \sqrt{m}) \, .
\end{equ}
\end{remark}

\begin{proof}[of Lemma \ref{lem:Analytic}]
By the Cauchy integral formula we deduce
\begin{equ}
\frac{|F^{(n)}(x)|}{n!} \leq \delta^{-n} \sup_{z \in D(x, \delta)} |F(z)| \lesssim \delta^{-n} \exp(\kappa \, (x+\delta)^2) \lesssim \delta^{-n} \exp(\kappa' x^2) \, ,
\end{equ}
which concludes.
\end{proof}

Now, let us give an example of a function that does not satisfy Assumption \ref{Assumption:F}. 
Of course, any function that is not smooth trivially fails, such as $F(x) = |x|$. 
Let us comment that for the latter, an explicit computations involving integrations by parts shows that
\begin{equ}
c_m(F) \sim (-1)^{\frac{m-2}{2}} \frac{2^{3/4}}{\pi^{3/4}} \frac{1}{m^{5/4}} \, , \qquad \text{for $m$ even, as $m\to\infty$,}
\end{equ}
which is consistent with Proposition \ref{p:cm(F)asymptotic}.

Let us now give an example of a smooth function that does not satisfy Assumption \ref{Assumption:F}.
Let $F$ be a smooth function which is flat at every order at $0$, 
such that for $x \in [-\delta, \delta]$ (for some $\delta \in (0,e^{-1})$) we have
\begin{equ}
	F(x) = |x|^{(\gamma + o(1)) \log \log(1/x)} \, , 
\end{equ}
for some $\gamma > 0$. Then, we claim that
\begin{equ}
	\sup_{x \in [-\delta,\delta]} |F^{(n)}(x)| \geq \exp(\exp((\gamma^{-1} + o(1)) \, n)) \,, 
\end{equ}
so that~\eqref{e:assumptionF} fails. 
First, the Taylor-Lagrange inequality (relative to the trivial Taylor expansion of $F$ at $0$) gives for every $x \in (0, \delta)$
\begin{equ}
	F(x) \leq \|F\|_{\infty, [0, \delta]} \frac{x^n}{n!} \, ,
\end{equ}
so that we can lower bound (using the change of variable $\ell = \log(1/x)$)
\begin{equ}
	\frac{\|F^{(n)}\|_{\infty, [0, \delta]}}{n!} \geq \sup_{x \in (0, \delta)} \frac{F(x)}{x^n} = \sup_{\ell \in (\log(1/\delta), \infty)} \exp(- (\gamma + o(1)) \log(\ell) \ell + n \ell) \, .
\end{equ}
Optimizing over $\ell$ (the maximizer is $\ell_* = \exp(n/(\gamma + o(1)) - 1)$, provided $n$ is sufficiently large so that $\ell_* \geq \log(1/\delta)$), we deduce
\begin{equ}
	\|F^{(n)}\|_{\infty, [0, \delta]} \geq \exp((\gamma + o(1)) \exp(n/(\gamma + o(1)) - 1)) = \exp(\exp((\gamma^{-1} + o(1)) \,  n)) \, .
\end{equ}

\section{Some technical results} \label{a:technical}

In this section, we collect a few technical results which are needed in the paper. 
The first ensures that for every $\tau>0$ fixed, the operators $\gensyx$, $\gensyy$, $\gensy$ and $\genaF{m}{a}$ 
are continuous from $\core$ equipped with its natural Fréchet metric to $\fock{}{}{}$ and derives 
a quantitative (but $\tau$-dependent) estimate on the latter. 

\begin{lemma}\label{l:ContOp}
Let $\tau>0$ be fixed. The operators $\gensyx$, $\gensyy$, $\gensy$ and $\genaF{m}{a}$, for 
$m \geq 1$ and $a \in \{-m+1,-m+3...,m-1\}$, respectively 
given in~\eqref{e:L0w},~\eqref{e:gensy},~\eqref{e:Gena1} and~\eqref{e:Genamj} 
are continuous on $\core$. Furthermore, 
for every $n\geq 1$, there exists a constant $C=C(\tau, n)>0$ such that 
for every $m$ and every $\psi\in\core_n$ we have 
\begin{equ} \label{e:SummabilityA}
\|\sqrt{m!} \,\cA^{\tau, m} \psi\|_\tau \leq C \, m^{n/2} \, \|\psi\|_\tau \, ,
\end{equ}
\end{lemma}
\begin{proof}
The continuity of the operators $\gensyx$, $\gensyy$ and $\gensy$ 
from $\core$ to $\fock{}{}{}$ is straightforward. 
For $\genaF{m}{a}$, it follows from the more quantitative estimate~\eqref{e:SummabilityA}. 
The case $m=1$ is trivial, so let us look at $m>1$ (notice that the bound for $m=2$ was given 
in~\cite[Lemma B.2]{CMT}). Let $\psi\in\core_n$ and recall the definition of 
$\cA_j^{\tau, m} \psi$ in~\eqref{e:Genamj}. 
Since $\hat\psi$ is symmetric with respect to permutation of its variables, we see that 
\begin{equs}
\|\cA_j^{\tau, m} &\psi\|_\tau^2 \lesssim \tau^2\nu_\tau \frac{n!^2}{(n+a)! (j+1)!^2} \binom{n+a}{m-j}^2  \\
& \times \tau^{-m} \nu_\tau^{\frac{m}{2}}\int \Xi_{n+a}^\tau(\dd p_{1:n+a}) \, \Big|\int_{r_{[1:j+1]} = p_{[1:m-j]}} \Xi_{j+1}^\tau(\dd r_{1:j+1}) \, \psi(r_{1:j+1}, p_{m-j+1:n+a}) \Big|^2 \\
&\lesssim \tau^2\nu_\tau \frac{n!^2}{(n+a)! (j+1)!^2} \binom{n+a}{m-j}^2\|\psi\|_\tau^2
\end{equs}
where in the last step we applied Cauchy-Schwarz to the integral with respect to the measure $\Xi_{j+1}^\tau$, 
used that, by~\eqref{e:RegMeas}, 
$ \Xi_{n+a}^\tau(\dd p_{1:n+a})=\Xi_{m-j}^\tau(\dd p_{1:m-j})\Xi_{n+a-m+j-1}^\tau(\dd p_{m-j+1:n+a})$ 
and that, by the assumptions on $\rho$ and the definition of $\rho_\tau$ and $\Theta^\tau$ in~\eqref{e:Mollifiers}, 
the total mass of $\Theta^\tau$ is proportional to $\tau \nu_\tau^{-1/2}$. 
Therefore, it remains to estimate the prefactor of the $\fock{}{}{\tau}$-norm of $\tau$, which satisfies 
\begin{equ}
\frac{n!}{(n+a)! (j+1)!^2} \binom{n+a}{m-j}^2=\frac{1}{(m+1)!} \binom{n}{j+1} \binom{m+1}{j+1} \binom{n+m-2j-1}{n-j-1}\lesssim_n  \frac{m^n}{m!}
\end{equ}
the last step being a consequence of $\binom{M}{N} \leq M^N$. Putting the two together,~\eqref{e:SummabilityA} 
follows at once. 
\end{proof}

Next, we want to check that if $F$ is a Hermite polynomial of given degree then 
the operator $A^\tau$ given in~\eqref{e:GenCylFunctional} coincides 
with $\gena$ in Lemma~\ref{lem:ActionGen}. The proof is rather standard and, for $F$ quadratic, can be found 
in, e.g.,~\cite{GPGen}, but since here 
the computations are more involved and the case of Hermite polynomials of degree different 
than $2$ has never been treated, we provide below full details. 

\begin{lemma}\label{l:PolGen}
For any $\tau>0$, $m\in\N$, $n\in\N$ and $\phi=I_n^\tau(h^{\otimes n}) = n! H_n( \eta^\tau(h))$ 
with $h$ such that $\hat h\in \cC_c^\infty(\R^2)$ and $\|h\|_{L^2_\tau(\R^2)} = 1$, 
we have 
\begin{equ}
A^\tau_{H_m}\phi=\genaF{m}{}\phi
\end{equ}
where $A^\tau_{H_m}=A^\tau$ is defined in~\eqref{e:GenCylFunctional} 
with $F$ given by the $m$-th Hermite polynomial $H_m$, and $\genaF{m}{}$ in~\eqref{e:AmA}. 
Further, on the set of $\phi$'s as above, $\genaF{m}{}$ is skew-symmetric, i.e. 
$(\genaF{m}{})^\ast=-\genaF{m}{}$. 
\end{lemma}
\begin{proof}
Let $\phi$ be as in the statement. Recall that by~\cite[Eq. (1.2)]{Nualart}, for any Hermite 
polynomial $H_n$ we have $H_n'\equiv H_{n-1}$. Hence, 
\begin{equs}[e:AmPreFormula]
A^\tau \Phi(\eta^\tau) &= \langle \cN_{H_m}^\tau(\eta^\tau), h \rangle_{L^2(\R^2)} n! H_{n-1}( \eta^\tau(h)) = \langle \cN_{H_m}^\tau(\eta^\tau), h \rangle_{L^2(\R^2)} \, n I_{n-1}^\tau(h^{\otimes n-1})\\
&=- \tau \nu_\tau^{1/4} \,  \langle H_m(\tau^{-1/2} \nu_\tau^{1/4} \eta^\tau), \partial_1 (\rho_\tau^{*2} - \1_{m=1} \delta_0) * h \rangle_{L^2(\R^2)} \, n I_{n-1}(h^{\otimes n-1})\,.
\end{equs}
We want to write the scalar product in terms of the isometry $I_m$. To do so, notice that, 
for every $x\in\R^2$
\begin{equ}
\tau^{-1/2} \nu_\tau^{1/4} \eta^\tau(x) =  \eta^\tau\big(\tau^{-1/2} \nu_\tau^{1/4} \delta_x \big)\qquad\text{and}\qquad \|\tau^{-1/2} \nu_\tau^{1/4} \delta_x\|_{L^2_\tau(\R^2)} = 1\,,
\end{equ}
and thus 
\begin{equs}
H_m\big(\tau^{-1/2} \nu_\tau^{1/4} \eta^\tau(x)\big) &= \frac{1}{m!} \, I_m^\tau\big((\tau^{-1/2} \nu_\tau^{1/4} \delta_x)^{\otimes m}\big) = \frac{\tau^{-m/2} \nu_\tau^{m/4}}{m!} \, I_m^\tau\big(\delta_x^{\otimes m}\big) \, .
\end{equs}
We will focus on the case $m\geq 2$, as for $m=1$ the computations are similar and simpler. 
The definition of $I_m^\tau$ and its linearity ensure that  
\begin{equs}
\langle &H_m\big(\tau^{-1/2} \nu_\tau^{1/4} \eta^\tau\big), \partial_1 \rho_\tau^{*2} * h \rangle_{L^2(\R^2)} = \frac{\tau^{-m/2} \nu_\tau^{m/4}}{m!} \, I_m^\tau\Big(\int_{\R^2} \dd x \, [\partial_1 \rho_\tau^{*2} * h](x) \, \delta_{x}^{\otimes m}\Big) \ ,
\end{equs}
so that the product rule in~\eqref{e:ProdRule} gives 
\begin{equ}[e:AtauInter]
A^\tau \Phi(\eta^\tau)=n\frac{\tau^{1-\frac{m}{2}}\nu_\tau^{\frac{m+1}{4}}}{m!}\sum_{j=0}^{m\wedge(n-1)} j! \binom{m}{j} \binom{n-1}{j} I_{m+n-1-2j}\big( {\rm Sym}(\psi_{m,j})\big)
\end{equ}
where $\psi_{m,j}\in L_\tau^2(\R^{2(n+m-2j-1)})$ is defined via~\eqref{e:ProdMixing}, i.e. 
its Fourier transform is given by 
\begin{equs}
\cF&(\psi_{m,j})(p_{1:m-j}, q_{1:n-1-j})\\
&\eqdef \int \Xi^\tau_j(\dd r_{1:j}) \, \cF\Big(\int_{\R^2} \dd x \, [\partial_1 \rho_\tau^{*2} * h](x) \, \delta_{x}^{\otimes m}\Big)(p_{1:m-j}, r_{1:j})\cF(h^{\otimes n-1})(q_{1:n-1-j}, -r_{1:j})\\
&=- \frac{\iota}{(2\pi)^{m-1}}\int \Xi^\tau_j(\dd r_{1:j}) \, \Theta^\tau(p_{[1:m-j]} - r_{[1:j]}) \, \fe_1 \cdot (p_{[1:m-j]} -r_{[1:j]}) \\
&\qquad \qquad \times \hat{h}(p_{[1:m-j]} - r_{[1:j]}) \, \hat{h}(q_1) ... \hat{h}(q_{n-1-j}) \, \hat{h}(r_1) ... \hat{h}(r_j)
\end{equs}
Now, if $j=m\leq n-1$, we extract those $\hat h$ depending on $q_i$'s and the remaining integral 
reads  
\begin{equs}
\frac{\iota}{(2\pi)^{m-1}}\int  \, \Theta^\tau(- r_{[1:m]}) &\, \fe_1 \cdot (-r_{[1:m]})\hat{h}(- r_{[1:m]}) \hat{h}(r_1) ... \hat{h}(r_m) \,\Xi^\tau_m(\dd r_{1:m}) \\
&=\int  [\partial_{\fe_1}(\rho_\tau^{\ast 2}\ast h)](x)\, (\rho_\tau^{\ast 2}\ast h)^m(x)\,\dd x=0
\end{equs}
the first equality being a consequence of the definition of $\Xi^\tau_m$ in~\eqref{e:RegMeas} 
and Fourier inversion formula, and the second from the fact that the integrand is a total derivative. 
Instead, for $j=0,\dots (m-1)\wedge(n-1)$ we have 
\begin{equs}
\cF&(\psi_{m,j})(p_{1:m-j}, q_{1:n-1-j})\\
&=- \frac{\iota}{(2\pi)^{m-1}}\int_{r_{[1:j+1]} = p_{[1:m-j]}} \Xi_{j+1}^\tau(\dd r_{1:j+1})\,\fe_1 \cdot r_{j+1}\, \hat\phi(r_{1:j+1}, q_{1:n-1-j})\\
&=- \frac{\iota}{(2\pi)^{m-1}}\frac{\fe_1 \cdot p_{[1:m-j]}}{j+1} \int_{r_{[1:j+1]} = p_{[1:m-j]}} \Xi_{j+1}^\tau(\dd r_{1:j+1}) \, \hat{\Phi}(r_{1:j+1}, q_{1:n-1-j})
\end{equs}
where in the second step we introduced the fictitious variable $r_{j+1}=p_{[1:m-j]} - r_{[1:j]}$ and 
identified $\phi$ with its kernel $h^{\otimes n}$ and in the last 
we symmetrised the integral with respect to the variables $r_{1:j+1}$. 
At this point, for $j\in\{0,\dots (m-1) \wedge (n-1)\}$ the above computations show that, 
before symmetrisation, the 
Fourier transform of the kernel of the $j$-th summand 
at the r.h.s. of~\eqref{e:AtauInter} equals 
\begin{equs}
-\iota &\frac{\tau^{1-\frac{m}2} \nu_\tau^{\frac{1+m}{4}}}{(2\pi)^{m-1}} \frac{n}{j+1}\, j! \binom{m}{j} \binom{n-1}{j}\times  \\
&\times \fe_1 \cdot p_{[1:m-j]} \int_{r_{[1:j+1]} = p_{[1:m-j]}} \Xi_{j+1}^\tau(\dd r_{1:j+1}) \, \hat{\Phi}(r_{1:j+1}, q_{1:n-1-j})
\end{equs}
which, by setting $a\eqdef m-1-2j$, relabelling the variables $q_{1:n-j}=p_{1:n+a\setminus \{1:m-j\}}$ and 
symmetrising with respect to $p_{1:n+a}$, coincides with the r.h.s. of~\eqref{e:Genamj}. 
In other words, we proved that 
\begin{equ}
A^\tau\phi=\sum_{j=0}^{m-1}\genaF{m}{m-2j-1}\phi=\genaF{m}{}\phi
\end{equ}
and since, for $F=H_m$,  $\gena=\genaF{m}{}$ the result follows. 

It remains to prove that $\genaF{m}{}$ is skew-symmetric. But this is an immediate consequence of 
the fact that $(\genaF{m}{a})^* = - \genaF{m}{-a}$ which in turn can be verified 
via a direct computation using the expression of these operators in~\eqref{e:Genamj}. 
\end{proof}

\end{appendix}

\bibliography{refs}

\def\cprime{$'$} \def\polhk#1{\setbox0=\hbox{#1}{\ooalign{\hidewidth
  \lower1.5ex\hbox{`}\hidewidth\crcr\unhbox0}}}
\begin{thebibliography}{CMOW25}
\def\myhref#1#2{\href{#2}{\nolinkurl{#1}}}

\bibitem[ABRK24]{armstrong}
\textsc{S.~Armstrong}, \textsc{A.~Bou-Rabee}, and \textsc{T.~Kuusi}.
\newblock Superdiffusive central limit theorem for a {B}rownian particle in a
  critically-correlated incompressible random drift.
\newblock \emph{arXiv preprint arXiv:2404.01115} (2024).

\bibitem[BC98]{BC}
\textsc{L.~Bertini} and \textsc{N.~Cancrini}.
\newblock The two-dimensional stochastic heat equation: renormalizing a
  multiplicative noise.
\newblock \emph{J. Phys. A, Math. Gen.} \textbf{31}, no.~2, (1998), 615--622.
\newblock
  \myhref{doi:10.1088/0305-4470/31/2/019}{https://dx.doi.org/10.1088/0305-4470/31/2/019}.

\bibitem[BQS11]{BQS}
\textsc{M.~Bal{\'a}zs}, \textsc{J.~Quastel}, and
  \textsc{T.~Sepp{\"a}l{\"a}inen}.
\newblock Fluctuation exponent of the {KPZ}/stochastic {Burgers} equation.
\newblock \emph{J. Am. Math. Soc.} \textbf{24}, no.~3, (2011), 683--708.
\newblock
  \myhref{doi:10.1090/S0894-0347-2011-00692-9}{https://dx.doi.org/10.1090/S0894-0347-2011-00692-9}.

\bibitem[BS95]{BS}
\textsc{A.-L. Barab{\'a}si} and \textsc{H.~E. Stanley}.
\newblock \emph{Fractal concepts in surface growth}.
\newblock Cambridge: Cambridge Univ. Press, 1995.

\bibitem[CD20]{CD}
\textsc{S.~Chatterjee} and \textsc{A.~Dunlap}.
\newblock Constructing a solution of the {{$(2+1)$}}-dimensional {KPZ}
  equation.
\newblock \emph{Ann. Probab.} \textbf{48}, no.~2, (2020), 1014--1055.
\newblock
  \myhref{doi:10.1214/19-AOP1382}{https://dx.doi.org/10.1214/19-AOP1382}.

\bibitem[CD25]{CasDun}
\textsc{B.~Castillo} and \textsc{A.~Dunlap}.
\newblock Mckean-vlasov limits of scaling-critical reaction-diffusion equations
  with random initial data (2025).
\newblock \myhref{arXiv:2509.06260}{https://arxiv.org/abs/2509.06260}.

\bibitem[CES21]{CES}
\textsc{G.~Cannizzaro}, \textsc{D.~Erhard}, and \textsc{P.~Sch{\"o}nbauer}.
\newblock 2d anisotropic {KPZ} at stationarity: scaling, tightness and
  nontriviality.
\newblock \emph{Ann. Probab.} \textbf{49}, no.~1, (2021), 122--156.
\newblock
  \myhref{doi:10.1214/20-AOP1446}{https://dx.doi.org/10.1214/20-AOP1446}.

\bibitem[CET23a]{CET}
\textsc{G.~Cannizzaro}, \textsc{D.~Erhard}, and \textsc{F.~Toninelli}.
\newblock The {Stationary} {AKPZ} equation: logarithmic superdiffusivity.
\newblock \emph{Commun. Pure Appl. Math.} \textbf{76}, no.~11, (2023),
  3044--3103.
\newblock \myhref{doi:10.1002/cpa.22108}{https://dx.doi.org/10.1002/cpa.22108}.

\bibitem[CET23b]{AKPZweak}
\textsc{G.~Cannizzaro}, \textsc{D.~Erhard}, and \textsc{F.~Toninelli}.
\newblock Weak coupling limit of the anisotropic {KPZ} equation.
\newblock \emph{Duke Math. J.} \textbf{172}, no.~16, (2023), 3013--3104.
\newblock
  \myhref{doi:10.1215/00127094-2022-0094}{https://dx.doi.org/10.1215/00127094-2022-0094}.

\bibitem[CG25]{CG}
\textsc{G.~Cannizzaro} and \textsc{H.~Giles}.
\newblock An invariance principle for the 2d weakly self-repelling {B}rownian
  polymer.
\newblock \emph{Probab. Theory Relat. Fields} (2025).
\newblock
  \myhref{doi:https://doi.org/10.1007/s00440-025-01363-y}{https://dx.doi.org/https://doi.org/10.1007/s00440-025-01363-y}.

\bibitem[CGT24]{CGT}
\textsc{G.~Cannizzaro}, \textsc{M.~Gubinelli}, and \textsc{F.~Toninelli}.
\newblock Gaussian fluctuations for the stochastic {Burgers} equation in
  dimension {$d \geq 2$}.
\newblock \emph{Commun. Math. Phys.} \textbf{405}, no.~4, (2024), 60.
\newblock Id/No 89.
\newblock
  \myhref{doi:10.1007/s00220-024-04966-z}{https://dx.doi.org/10.1007/s00220-024-04966-z}.

\bibitem[CHST22]{CHT}
\textsc{G.~Cannizzaro}, \textsc{L.~Haunschmid-Sibitz}, and
  \textsc{F.~Toninelli}.
\newblock {{$\sqrt{\log t}$}}-superdiffusivity for a {Brownian} particle in the
  curl of the 2d {GFF}.
\newblock \emph{Ann. Probab.} \textbf{50}, no.~6, (2022), 2475--2498.
\newblock
  \myhref{doi:10.1214/22-AOP1589}{https://dx.doi.org/10.1214/22-AOP1589}.

\bibitem[CK24]{CK}
\textsc{G.~Cannizzaro} and \textsc{J.~Kiedrowski}.
\newblock Stationary stochastic {Navier}-{Stokes} on the plane at and above
  criticality.
\newblock \emph{Stoch. Partial Differ. Equ., Anal. Comput.} \textbf{12}, no.~1,
  (2024), 247--280.
\newblock
  \myhref{doi:10.1007/s40072-022-00283-5}{https://dx.doi.org/10.1007/s40072-022-00283-5}.

\bibitem[CMOW25]{Morfe}
\textsc{G.~Chatzigeorgiou}, \textsc{P.~Morfe}, \textsc{F.~Otto}, and
  \textsc{L.~Wang}.
\newblock {The Gaussian free-field as a stream function: Asymptotics of
  effective diffusivity in infra-red cut-off}.
\newblock \emph{The Annals of Probability} \textbf{53}, no.~4, (2025), 1510 --
  1536.
\newblock
  \myhref{doi:10.1214/24-AOP1740}{https://dx.doi.org/10.1214/24-AOP1740}.

\bibitem[CMT25]{CMT}
\textsc{G.~Cannizzaro}, \textsc{Q.~Moulard}, and \textsc{F.~Toninelli}.
\newblock Superdiffusive central limit theorem for the stochastic burgers
  equation at the critical dimension, 2025.
\newblock \myhref{arXiv:2501.00344}{https://arxiv.org/abs/2501.00344}.

\bibitem[CSZ17]{CSZ1}
\textsc{F.~Caravenna}, \textsc{R.~Sun}, and \textsc{N.~Zygouras}.
\newblock Universality in marginally relevant disordered systems.
\newblock \emph{Ann. Appl. Probab.} \textbf{27}, no.~5, (2017), 3050--3112.
\newblock
  \myhref{doi:10.1214/17-AAP1276}{https://dx.doi.org/10.1214/17-AAP1276}.

\bibitem[CSZ20]{CSZ}
\textsc{F.~Caravenna}, \textsc{R.~Sun}, and \textsc{N.~Zygouras}.
\newblock The two-dimensional {KPZ} equation in the entire subcritical regime.
\newblock \emph{Ann. Probab.} \textbf{48}, no.~3, (2020), 1086--1127.
\newblock
  \myhref{doi:10.1214/19-AOP1383}{https://dx.doi.org/10.1214/19-AOP1383}.

\bibitem[CSZ23]{caravenna2023critical}
\textsc{F.~Caravenna}, \textsc{R.~Sun}, and \textsc{N.~Zygouras}.
\newblock The critical 2d stochastic heat flow.
\newblock \emph{Invent. Math.} \textbf{233}, no.~1, (2023), 325--460.
\newblock
  \myhref{doi:10.1007/s00222-023-01184-7}{https://dx.doi.org/10.1007/s00222-023-01184-7}.

\bibitem[DG22]{DG}
\textsc{A.~Dunlap} and \textsc{Y.~Gu}.
\newblock A forward-backward {SDE} from the 2d nonlinear stochastic heat
  equation.
\newblock \emph{Ann. Probab.} \textbf{50}, no.~3, (2022), 1204--1253.
\newblock
  \myhref{doi:10.1214/21-AOP1563}{https://dx.doi.org/10.1214/21-AOP1563}.

\bibitem[DGHS24]{de2024log}
\textsc{D.~De~Gaspari} and \textsc{L.~Haunschmid-Sibitz}.
\newblock {{$(\log t)^{2/3}$}-superdiffusivity for the 2d stochastic Burgers
  equation}.
\newblock \emph{Electronic Journal of Probability} \textbf{29}, no. none,
  (2024), 1 -- 34.
\newblock
  \myhref{doi:10.1214/24-EJP1249}{https://dx.doi.org/10.1214/24-EJP1249}.

\bibitem[DHL25]{DHL}
\textsc{A.~Dunlap}, \textsc{M.~Hairer}, and \textsc{X.-M. Li}.
\newblock A critical stochastic heat equation with long-range noise (2025).
\newblock \myhref{arXiv:2509.23790}{https://arxiv.org/abs/2509.23790}.

\bibitem[EX22]{erhard_xu_22}
\textsc{D.~Erhard} and \textsc{W.~Xu}.
\newblock {Weak universality of dynamical ${\Phi _{3}^{4}}$: polynomial
  potential and general smoothing mechanism}.
\newblock \emph{Electronic Journal of Probability} \textbf{27}, no. none,
  (2022), 1 -- 43.
\newblock \myhref{doi:10.1214/22-EJP833}{https://dx.doi.org/10.1214/22-EJP833}.

\bibitem[FG19]{FG19}
\textsc{M.~Furlan} and \textsc{M.~Gubinelli}.
\newblock Weak universality for a class of 3d stochastic reaction–diffusion
  models \textbf{173}, no.~3, (2019), 1099--1164.
\newblock
  \myhref{doi:10.1007/s00440-018-0849-6}{https://dx.doi.org/10.1007/s00440-018-0849-6}.

\bibitem[GHL25]{GHL}
\textsc{L.~Gerolla}, \textsc{M.~Hairer}, and \textsc{X.-M. Li}.
\newblock {Fluctuations of stochastic PDEs with long-range correlations}.
\newblock \emph{The Annals of Applied Probability} \textbf{35}, no.~2, (2025),
  1198 -- 1232.
\newblock
  \myhref{doi:10.1214/24-AAP2140}{https://dx.doi.org/10.1214/24-AAP2140}.

\bibitem[GIP15]{GIP15}
\textsc{M.~Gubinelli}, \textsc{P.~Imkeller}, and \textsc{N.~Perkowski}.
\newblock {Paracontrolled Distributions and Singular PDEs}.
\newblock \emph{Forum of Mathematics, Pi} \textbf{3}, (2015), e6.
\newblock
  \myhref{doi:10.1017/fmp.2015.2}{https://dx.doi.org/10.1017/fmp.2015.2}.

\bibitem[GJ13]{GJ13}
\textsc{M.~Gubinelli} and \textsc{M.~Jara}.
\newblock Regularization by noise and stochastic {Burgers} equations.
\newblock \emph{Stoch. Partial Differ. Equ., Anal. Comput.} \textbf{1}, no.~2,
  (2013), 325--350.
\newblock
  \myhref{doi:10.1007/s40072-013-0011-5}{https://dx.doi.org/10.1007/s40072-013-0011-5}.

\bibitem[GJ14]{GJ14}
\textsc{P.~Gon\c{c}alves} and \textsc{M.~Jara}.
\newblock Nonlinear fluctuations of weakly asymmetric interacting particle
  systems.
\newblock \emph{Arch. Ration. Mech. Anal.} \textbf{212}, no.~2, (2014),
  597--644.
\newblock \myhref{arXiv:1309.5120}{https://arxiv.org/abs/1309.5120}.
\newblock
  \myhref{doi:10.1007/s00205-013-0693-x}{https://dx.doi.org/10.1007/s00205-013-0693-x}.

\bibitem[GP16]{GP16}
\textsc{M.~Gubinelli} and \textsc{N.~Perkowski}.
\newblock The {H}airer-{Q}uastel universality result at stationarity.
\newblock In \emph{Stochastic analysis on large scale interacting systems},
  RIMS K\^oky\^uroku Bessatsu, B59,  101--115. Res. Inst. Math. Sci. (RIMS),
  Kyoto, 2016.
\newblock \myhref{arXiv:1602.02428}{https://arxiv.org/abs/1602.02428}.

\bibitem[GP18]{GP18}
\textsc{M.~Gubinelli} and \textsc{N.~Perkowski}.
\newblock Energy solutions of {KPZ} are unique.
\newblock \emph{J. Amer. Math. Soc.} \textbf{31}, no.~2, (2018), 427--471.
\newblock \myhref{arXiv:1508.07764}{https://arxiv.org/abs/1508.07764}.
\newblock \myhref{doi:10.1090/jams/889}{https://dx.doi.org/10.1090/jams/889}.

\bibitem[GP20]{GPGen}
\textsc{M.~Gubinelli} and \textsc{N.~Perkowski}.
\newblock The infinitesimal generator of the stochastic {Burgers} equation.
\newblock \emph{Probab. Theory Relat. Fields} \textbf{178}, no. 3-4, (2020),
  1067--1124.
\newblock
  \myhref{doi:10.1007/s00440-020-00996-5}{https://dx.doi.org/10.1007/s00440-020-00996-5}.

\bibitem[GPP24]{GPP}
\textsc{L.~Gr{\"a}fner}, \textsc{N.~Perkowski}, and \textsc{S.~Popat}.
\newblock Energy solutions of singular {SPDEs} on {Hilbert} spaces with
  applications to domains with boundary conditions (2024).

\bibitem[GRZ24]{GRZallen}
\textsc{S.~Gabriel}, \textsc{T.~Rosati}, and \textsc{N.~Zygouras}.
\newblock The {Allen}-{Cahn} equation with weakly critical random initial
  datum.
\newblock \emph{Probab. Theory Relat. Fields} (2024).
\newblock
  \myhref{doi:10.1007/s00440-024-01312-1}{https://dx.doi.org/10.1007/s00440-024-01312-1}.

\bibitem[Gu20]{Gu2020}
\textsc{Y.~Gu}.
\newblock Gaussian fluctuations from the 2{D} {KPZ} equation.
\newblock \emph{Stoch. Partial Differ. Equ. Anal. Comput.} \textbf{8}, no.~1,
  (2020), 150--185.
\newblock
  \myhref{doi:10.1007/s40072-019-00144-8}{https://dx.doi.org/10.1007/s40072-019-00144-8}.

\bibitem[Hai13]{Hai13}
\textsc{M.~Hairer}.
\newblock Solving the {KPZ} equation.
\newblock \emph{Ann. Math.} (2013).
\newblock To appear.

\bibitem[Hai14]{Hai14}
\textsc{M.~Hairer}.
\newblock A theory of regularity structures.
\newblock \emph{Invent. Math.} \textbf{198}, no.~2, (2014), 269--504.
\newblock \myhref{arXiv:1303.5113}{https://arxiv.org/abs/1303.5113}.
\newblock
  \myhref{doi:10.1007/s00222-014-0505-4}{https://dx.doi.org/10.1007/s00222-014-0505-4}.

\bibitem[HQ18]{hairer_quastel_18}
\textsc{M.~Hairer} and \textsc{J.~Quastel}.
\newblock A class of growth models rescaling to {KPZ}.
\newblock \emph{Forum Math. Pi} \textbf{6}, (2018), e3, 112.
\newblock \myhref{arXiv:1512.07845}{https://arxiv.org/abs/1512.07845}.
\newblock
  \myhref{doi:10.1017/fmp.2018.2}{https://dx.doi.org/10.1017/fmp.2018.2}.

\bibitem[HS17]{HS17}
\textsc{M.~Hairer} and \textsc{H.~Shen}.
\newblock {A central limit theorem for the KPZ equation}.
\newblock \emph{The Annals of Probability} \textbf{45}, no.~6B, (2017), 4167 --
  4221.
\newblock
  \myhref{doi:10.1214/16-AOP1162}{https://dx.doi.org/10.1214/16-AOP1162}.

\bibitem[HX18]{HX18}
\textsc{M.~Hairer} and \textsc{W.~Xu}.
\newblock {Large-Scale Behavior of Three-Dimensional Continuous Phase
  Coexistence Models}.
\newblock \emph{Communications on Pure and Applied Mathematics} \textbf{71},
  no.~4, (2018), 688--746.
\newblock \myhref{doi:10.1002/cpa.21738}{https://dx.doi.org/10.1002/cpa.21738}.

\bibitem[HX19]{HX19}
\textsc{M.~Hairer} and \textsc{W.~Xu}.
\newblock Large scale limit of interface fluctuation models.
\newblock \emph{Ann. Probab.} \textbf{47}, no.~6, (2019), 3478--3550.
\newblock \myhref{arXiv:1802.08192}{https://arxiv.org/abs/1802.08192}.
\newblock
  \myhref{doi:10.1214/18-aop1317}{https://dx.doi.org/10.1214/18-aop1317}.

\bibitem[IK02]{IK}
\textsc{K.~Ito} and \textsc{F.~Kappel}.
\newblock \emph{Evolution equations and approximations}, vol.~61 of \emph{Ser.
  Adv. Math. Appl. Sci.}
\newblock Singapore: World Scientific, 2002.

\bibitem[KLO12]{KLO}
\textsc{T.~Komorowski}, \textsc{C.~Landim}, and \textsc{S.~Olla}.
\newblock \emph{Fluctuations in {Markov} processes. {Time} symmetry and
  martingale approximation.}, vol. 345 of \emph{Grundlehren Math. Wiss.}
\newblock Berlin: Springer, 2012.
\newblock
  \myhref{doi:10.1007/978-3-642-29880-6}{https://dx.doi.org/10.1007/978-3-642-29880-6}.

\bibitem[KPZ86]{KPZ}
\textsc{M.~Kardar}, \textsc{G.~Parisi}, and \textsc{Y.-C. Zhang}.
\newblock Dynamic scaling of growing interfaces.
\newblock \emph{Phys. Rev. Lett.} \textbf{56}, no.~9, (1986), 889--892.

\bibitem[KRYZ25]{KRYZ}
\textsc{S.~Kotitsas}, \textsc{M.~Romito}, \textsc{Z.~Yang}, and
  \textsc{X.~Zhu}.
\newblock Gaussian fluctuations for the stochastic landau-lifshitz
  navier-stokes equation in dimension $d\geq2$ (2025).
\newblock \myhref{arXiv:2512.04567}{https://arxiv.org/abs/2512.04567}.

\bibitem[KWX24]{KWX24}
\textsc{F.~Kong}, \textsc{H.~Wang}, and \textsc{W.~Xu}.
\newblock {Hairer-Quastel universality for KPZ -- polynomial smoothing
  mechanisms, general nonlinearities and Poisson noise}, 2024.
\newblock \myhref{arXiv:2403.06191}{https://arxiv.org/abs/2403.06191}.

\bibitem[KZ22]{KZ22}
\textsc{F.~Kong} and \textsc{W.~Zhao}.
\newblock {A frequency-independent bound on trigonometric polynomials of
  Gaussians and applications}, 2022.
\newblock \myhref{arXiv:2208.05200}{https://arxiv.org/abs/2208.05200}.

\bibitem[Nua06]{Nualart}
\textsc{D.~Nualart}.
\newblock \emph{The {M}alliavin calculus and related topics}.
\newblock Probability and its Applications (New York). Springer-Verlag, Berlin,
  second ed., 2006,  xiv+382.

\bibitem[QS20]{QS1}
\textsc{J.~Quastel} and \textsc{S.~Sarkar}.
\newblock The {KPZ} equation converges to the {KPZ} fixed point.
\newblock \emph{arXiv preprint} (2020).
\newblock \myhref{arXiv:2008.06584}{https://arxiv.org/abs/2008.06584}.

\bibitem[Spo91]{Spo}
\textsc{H.~Spohn}.
\newblock \emph{Large scale dynamics of interacting particles}.
\newblock Texts Monogr. Phys. Berlin etc.: Springer-Verlag, 1991.

\bibitem[Spo14]{Spo1}
\textsc{H.~Spohn}.
\newblock Nonlinear fluctuating hydrodynamics for anharmonic chains.
\newblock \emph{J. Stat. Phys.} \textbf{154}, no.~5, (2014), 1191--1227.
\newblock
  \myhref{doi:10.1007/s10955-014-0933-y}{https://dx.doi.org/10.1007/s10955-014-0933-y}.

\bibitem[SX18]{SX18}
\textsc{H.~Shen} and \textsc{W.~Xu}.
\newblock {Weak universality of dynamical $\Phi ^4_3$: non-Gaussian noise}.
\newblock \emph{Stochastics and Partial Differential Equations: Analysis and
  Computations} \textbf{6}, no.~2, (2018), 211--254.
\newblock
  \myhref{doi:10.1007/s40072-017-0107-4}{https://dx.doi.org/10.1007/s40072-017-0107-4}.

\bibitem[Tao24]{Tao}
\textsc{R.~Tao}.
\newblock Gaussian fluctuations of a nonlinear stochastic heat equation in
  dimension two.
\newblock \emph{Stoch. Partial Differ. Equ., Anal. Comput.} \textbf{12}, no.~1,
  (2024), 220--246.
\newblock
  \myhref{doi:10.1007/s40072-022-00282-6}{https://dx.doi.org/10.1007/s40072-022-00282-6}.

\bibitem[Ton18]{Ton}
\textsc{F.~Toninelli}.
\newblock {{$(2+1)$}}-dimensional interface dynamics: mixing time, hydrodynamic
  limit and anisotropic {KPZ} growth.
\newblock In \emph{Proceedings of the international congress of mathematicians
  2018, ICM 2018, Rio de Janeiro, Brazil, August 1--9, 2018. Volume III.
  Invited lectures},  2733--2758. Hackensack, NJ: World Scientific; Rio de
  Janeiro: Sociedade Brasileira de Matem{\'a}tica (SBM), 2018.
\newblock
  \myhref{doi:10.1142/9789813272880_0158}{https://dx.doi.org/10.1142/9789813272880_0158}.

\bibitem[Tsa24]{Tsai}
\textsc{L.-C. Tsai}.
\newblock Stochastic heat flow by moments (2024).

\bibitem[TV12]{TothValko}
\textsc{B.~T\'{o}th} and \textsc{B.~Valk\'{o}}.
\newblock Superdiffusive bounds on self-repellent {B}rownian polymers and
  diffusion in the curl of the {G}aussian free field in {$d=2$}.
\newblock \emph{J. Stat. Phys.} \textbf{147}, no.~1, (2012), 113--131.
\newblock
  \myhref{doi:10.1007/s10955-012-0462-5}{https://dx.doi.org/10.1007/s10955-012-0462-5}.

\bibitem[vBKS85]{vBKS85}
\textsc{H.~van Beijeren}, \textsc{R.~Kutner}, and \textsc{H.~Spohn}.
\newblock Excess noise for driven diffusive systems.
\newblock \emph{Phys. Rev. Lett.} \textbf{54}, (1985), 2026--2029.
\newblock
  \myhref{doi:10.1103/PhysRevLett.54.2026}{https://dx.doi.org/10.1103/PhysRevLett.54.2026}.

\bibitem[Vir20]{Virag}
\textsc{B.~Virag}.
\newblock The heat and the landscape {I}.
\newblock \emph{arXiv preprint} (2020).
\newblock \myhref{arXiv:2008.07241}{https://arxiv.org/abs/2008.07241}.

\bibitem[Yan23]{Y23}
\textsc{K.~Yang}.
\newblock {Hairer--Quastel universality in non-stationarity via energy solution
  theory}.
\newblock \emph{Electronic Journal of Probability} \textbf{28}, (2023), 1--26.

\bibitem[Yan25]{Y25}
\textsc{K.~Yang}.
\newblock Kpz equation from a class of nonlinear spdes in infinite volume
  (2025).
\newblock \myhref{arXiv:2507.10545}{https://arxiv.org/abs/2507.10545}.

\bibitem[Yau04]{Yau}
\textsc{H.-T. Yau}.
\newblock {$(\log t)^{2/3}$} law of the two dimensional asymmetric simple
  exclusion process.
\newblock \emph{Ann. of Math. (2)} \textbf{159}, no.~1, (2004), 377--405.
\newblock
  \myhref{doi:10.4007/annals.2004.159.377}{https://dx.doi.org/10.4007/annals.2004.159.377}.

\bibitem[YY24]{yang2024weak}
\textsc{H.~Yang} and \textsc{Z.~Yang}.
\newblock Weak coupling limit of a {B}rownian particle in the curl of the {2D
  GFF}.
\newblock \emph{arXiv preprint arXiv:2405.05778} (2024).

\end{thebibliography}
\bibliographystyle{Martin}
\addcontentsline{toc}{section}{References}

\end{document}